\documentclass[11pt,reqno]{amsart}

\usepackage{etoolbox}

\newbool{bForSubmission}
\booltrue{bForSubmission}

\newbool{bExperimental}

\newbool{bForUs}

\usepackage{amscd,amsmath}
\usepackage{mathrsfs}
\usepackage{amsfonts}

\usepackage{amssymb, bm, xspace}
\usepackage{enumerate}

\usepackage{enumitem}

\usepackage{setspace}

\usepackage[OT2,OT1]{fontenc}

\usepackage{datetime}

\usepackage{enumitem}

\usepackage[dvipsnames]{xcolor}

\usepackage[margin=1.2in]{geometry}

\usepackage{mathtools}

\mathtoolsset{centercolon}

\usepackage{stackengine}
\stackMath

\usepackage{tcolorbox}
\usepackage{tikz}

\newbool{HaveBBM}
\booltrue{HaveBBM}

\ifbool{HaveBBM}{
	\usepackage{bbm}
    }
    {
    }

\usepackage{titletoc}

\usepackage[pagebackref=true, colorlinks=true, citecolor=blue]{hyperref}

\usepackage{cleveref}

\tcbuselibrary{skins}
\usetikzlibrary{shadings}
\tcbset{
    myimage/.style={
        enhanced,
        overlay={
            \begin{scope}[shift={([xshift=1mm, yshift=7mm]frame.north west)}]
            \end{scope}}}}

\tcbset{
    skin=enhanced,
    fonttitle=\bfseries,
    interior style={white},
    segmentation style={black,solid,opacity=0.2,line width=1pt}}

\newtcolorbox{TitledBox}[2][]{
    myimage,              
    coltitle=black,       
    colbacktitle=white,   
    title=My title,
    attach boxed title to top center={
        yshift=-3mm,
        yshifttext=-1mm},
    attach boxed title to top left={
        xshift=1cm,
        yshift=-2mm},
    boxed title style={
        size=small},
    title={#2},#1}

\newcommand{\IncludeFigure}[2]{
    \begin{center}
            \includegraphics[scale=#1]{#2.jpg}
    \end{center}
}

\newcommand{\TwoColumn}[6]%
{%
\begin{minipage}[#3]{#1\textwidth}%
#5%
\end{minipage}%
\begin{minipage}[#4]{#2\textwidth}%
#6%
\end{minipage}%
}

\newcommand{\TwoColumnTop}[4]%
{%
\begin{minipage}[t]{#1\textwidth}
#3
\end{minipage}%
\begin{minipage}[t]{#2\textwidth}
#4
\end{minipage}%
}

\newcommand{\PullMarginsIn}[1]%
{%
\begin{minipage}[c]{0.3\textwidth}%
\phantom{x}%
\end{minipage}%
\begin{minipage}[c]{0.4\textwidth}%
#1%
\end{minipage}%
\begin{minipage}[c]{0.3\textwidth}%
\phantom{x}%
\end{minipage}%
}

\usepackage{yhmath}

\newcommand{\Comment}[1]{{\color{Brown}#1}}

\newcommand{\OptionalDetails}[1]{
    \ifbool{bForSubmission}
        {%
        }%
        {\begin{quote}\Comment{\footnotesize
        \medskip
        
        \noindent\textbf{Details not for submission}: \\
        \noindent#1}
        \end{quote}
        }
    }

\newbool{arXivFormat}
\boolfalse{arXivFormat}
    
\newcommand{\IfarXivElse}[2]{
    \ifbool{arXivFormat}
        {#1}{#2}
    }

\renewcommand{\mathbf}[1]{\bm{#1} \textbf{ *** Use bm instead of mathbf ***}}

\newcommand{\eqn}{\begin{eqnarray}}
\newcommand{\een}{\end{eqnarray}}

\newtheorem{theorem}{Theorem}[section]
\newtheorem*{theorem*}{Theorem}				
\newtheorem{prop}[theorem]{Proposition}
\newtheorem{lemma}[theorem]{Lemma}
\newtheorem{cor}[theorem]{Corollary}

\newtheorem{definition}[theorem]{Definition}
\newtheorem{remark}[theorem]{Remark}
\newtheorem*{remark*}{Remark}

\numberwithin{equation}{section}

\newcommand{\innp}[1]{\ensuremath{\langle #1 \rangle}}

\newcommand{\BoldTau}
    {{\mbox{\boldmath $\tau$}}}

\newcommand{\bmu}
    {{\mbox{\boldmath $\mu$}}}
\newcommand{\bnu}
    {{\mbox{\boldmath $\nu$}}}
\newcommand{\bxi}
    {{\mbox{\boldmath $\xi$}}}

\newcommand{\BoldAlpha}             
    {{\mbox{\boldmath $\alpha$}}}
\newcommand{\bgamma}
    {{\mbox{\boldmath $\gamma$}}}
\newcommand{\bdelta}
    {{\mbox{\boldmath $\delta$}}}

\newcommand{\BB}[1]{\ensuremath{\mathbb{#1}}}
\newcommand{\R}{{\ensuremath{\BB{R}}}}

\newcommand{\Z}{\ensuremath{\BB{Z}}}

\newcommand{\T}{{\ensuremath{\BB{T}}}}
\newcommand\Tn{{\vec{\T}}}

\newcommand{\iny}{\ensuremath{\infty}}
\newcommand{\grad}{\ensuremath{\nabla}}

\newcommand{\Cover}{\bm{p}}
\newcommand{\periodic}{$1$-periodic in $x_1$\xspace}

\newcommand{\CharFunc}{
    \ifbool{HaveBBM}{
        {\ensuremath{\mathbbm{1}}}
        }
        {
        {\ensuremath{\bm{1}}}
        }
    }

\DeclareMathOperator{\dv}{div} %
\DeclareMathOperator{\curl}{curl} %

\DeclareMathOperator{\dist}{dist} %
\DeclareMathOperator{\supp}{supp} %
\DeclareMathOperator{\image}{image} %
\DeclareMathOperator{\Patch}{Patch} %

\newcommand{\prt}{\ensuremath{\partial}}
\newcommand{\brac}[1]{\ensuremath{\left[ #1 \right]}}

\newcommand{\pr}[1]{\ensuremath{\left( #1 \right)}}

\newcommand{\set}[1]{\ensuremath{\left\{ #1 \right\}}}

\DeclarePairedDelimiterX{\norm}[1]{\lVert}{\rVert}{#1}

\DeclarePairedDelimiterX{\abs}[1]{\lvert}{\rvert}{#1}

\newcommand\tenq[2][1]{%
	\def\useanchorwidth{T}%
	\ifnum#1>1%
		\stackunder[0pt]{\tenq[\numexpr#1-1\relax]{#2}}{\scriptscriptstyle\sim}%
	\else%
		\stackunder[1pt]{#2}{\scriptscriptstyle\sim}%
	\fi%
	}

\DeclarePairedDelimiter{\pbrac}{\lbrack}{\rbrack}

\newcommand{\n}{{\bm{n}}}

\newcommand{\wh}{\widehat}

\usepackage{mathtools}
\usepackage{ragged2e}
\newlength\ubwidth

\renewcommand{\epsilon}{\varepsilon}
\newcommand{\eps}{\ensuremath{\varepsilon}}

\newcommand{\Cal}[1]{\ensuremath{\mathcal{#1}}}
\newcommand{\al}{\ensuremath{\alpha}}
\newcommand{\la}{\ensuremath{\lambda}}

\newcommand{\diff}[2]{\frac{ d#1}{d#2}}

\newcommand{\ol}{\overline}

\renewcommand{\le}{\leqslant}
\renewcommand{\ge}{\geqslant}

\newcommand{\Holder}
    {H\"{o}lder\xspace}

\newcommand{\Gronwalls}
    {Gr\"{o}nwall's\xspace}

\newcommand{\Ignore}[1]{}

\newcommand{\Experimental}[1]%
{%
\ifbool{bExperimental}%
	{\bigskip
	\noindent
	\textbf{\color{Brown}*** Start experimental}\\
	#1
}
{
}
}

\definecolor{Correction}{named}{red}

\crefname{cor}{Corollary}{Corollaries} 
									   
\crefname{lemma}{Lemma}{Lemmas}	       

\crefname{section}{Section}{Sections}
\Crefname{section}{Section}{Sections}

\crefname{appendix}{Appendix}{Appendices}
\Crefname{appendix}{Appendix}{Appendices}

\crefname{theorem}{Theorem}{Theorems}
\Crefname{theorem}{Theorem}{Theorems}

\crefname{prop}{Proposition}{Propositions}
\Crefname{prop}{Proposition}{Propositions}

\crefname{conj}{Conjecture}{Conjectures}
\Crefname{conj}{Conjecture}{Conjectures}

\crefname{definition}{Definition}{Definitions}
\Crefname{definition}{Definition}{Definitions}

\crefname{remark}{Remark}{Remarks}
\Crefname{remark}{Remark}{Remarks}

\crefname{figure}{figure}{figures}
\Crefname{figure}{Figure}{Figures}

\crefname{assumption}{Assumption}{Assumptions}
\Crefname{assumption}{Assumption}{Assumptions}

\crefformat{equation}{(#2#1#3)}
\crefrangeformat{equation}{(#3#1#4) through (#5#2#6)}
\crefmultiformat{equation}
    {(#2#1#3)}%
    { and~(#2#1#3)}
    {, (#2#1#3)}
    { and~(#2#1#3)}

\newcommand{\e}{\bm{\mathrm{e}}}

\renewcommand{\H}{{\bm{\mathrm{H}}}}

\newcommand{\uu}{{\bm{\mathrm{u}}}}   

\newcommand{\vv}{{\bm{\mathrm{v}}}}

\newcommand{\ww}{{\bm{\mathrm{w}}}}
\newcommand{\x}{{\bm{\mathrm{x}}}}
\newcommand{\y}{{\bm{\mathrm{y}}}}
\newcommand{\bb}{{\bm{\mathrm{b}}}}
\renewcommand{\L}{{\Cal{L}}}

\newcommand{\z}{{\bm{\mathrm{z}}}}

\newcommand{\Rep}{\Cal{R}ep}

\DeclareMathOperator{\Res}{Res}

\def\XXint#1#2#3{{\setbox0=\hbox{$#1{#2#3}{\int}$ }
\vcenter{\hbox{$#2#3$ }}\kern-.6\wd0}}

\newcommand{\SQG}{$\al$-SQG\xspace}
\newcommand{\SQGConst}{\cref{e:SQG}$_2$\xspace}

\newcommand{\symdiff}{{\footnotesize\mathbin{\text{$\triangle$}}}}

\allowdisplaybreaks

\begin{document}
\newdateformat{mydate}{\THEDAY~\monthname~\THEYEAR}

\title
	[Periodic $\alpha$-SQG Patches and Fronts]
	{Horizontally Periodic Generalized Surface Quasigeostrophic Patches and Layers}

\author[D. Ambrose, F. Hadadifard, and J. Kelliher]
{David M. Ambrose$^{1}$, Fazel Hadadifard$^{2}$, and James P. Kelliher$^{3}$}
\address{$^1$ Department of Mathematics, Drexel University, 3141 Chestnut Street, Philadelphia, PA 19104, USA}
\email{dma68@drexel.edu}
\address{$^2$ Department of Mathematics, California State University San Marcos, San Marcos, CA 92096, USA}
\email{fhadadifard@csusm.edu}
\address{$^3$ Department of Mathematics, University of California, Riverside, 900 University Ave., Riverside, CA 92521, USA}
\email{kelliher@math.ucr.edu}

\begin{abstract}
We study solutions to the $\alpha$-SQG equations, which interpolate between the incompressible Euler and surface quasi-geostrophic equations. We extend prior results on existence of bounded patches, proving propagation of $H^k$-regularity of the patch boundary, $k \ge 3$, for finite time for patches that are periodic in one spatial dimension. Such periodic patches also encompass layers, or two-sided fronts. As the authors have treated the Euler case in prior work, 
we now primarily focus on the range of $\al$ for which $\al$-SQG lies strictly between the Euler and SQG equations.
\end{abstract}

\maketitle

\markleft{ D.M. AMBROSE, F. HADADIFARD, and J. KELLIHER}

\vspace{-2.5em}

\begin{center}

\medskip
Compiled on {\dayofweekname{\day}{\month}{\year} \mydate\today} at \currenttime\xspace Pacific Time
		
\end{center}

\bigskip

\vspace{-1.5em}

{
\footnotesize
\renewcommand\contentsname{}	

\begin{center}
\begin{minipage}{0.6\linewidth}
\centering
\tableofcontents
\end{minipage}
\end{center}

}



\newpage
\normalsize

%


%
%
\section{Introduction}

\noindent The generalized surface quasi-geostrophic \SQG equations in strong form can be written,
\begin{align}\label{e:SQG}
	(\al\text{-}SQG) \qquad
	\begin{cases}
		\prt_t \theta + \uu \cdot \grad \theta = 0
			&\text{ in }  [0, T] \times \R^2, \\
		\uu  = -\grad^\perp (-\Delta)^{-(1 - \frac{\al}{2})} \theta
			&\text{ in }  [0, T] \times \R^2, \\
		\theta(0) = \theta_0
			&\text{ in }  \R^2.
	\end{cases}
\end{align}
Here, $\uu$ is a velocity field by which the active scalar $\theta$ is transported, with $\uu$ recovered from the scalar via the constitutive law in \cref{e:SQG}$_2$. Because of the presence of the operator $\grad^\perp := (-\prt_2, \prt_1)$, we have  
$\dv \uu = 0$. A solution is to hold until some time $T > 0$, which may be finite.

When used to model atmospheric turbulence for small Rossby and Ekman numbers and constant potential velocity, $\theta$ is the temperature, which is advected by the divergence-free velocity field $\uu$. (See, for instance, \cite{pedlosky1987geophysical}.)

We focus on values of $\al \in [0, 1]$, the two cases of most widespread interest being $\al = 0, 1$, in which case \cref{e:SQG} becomes the following:
\begin{align*}
	\begin{array}{ll}
		\al = 0: & \text{Euler equations}, \\
		\al = 1: & \text{SQG equations}.
	\end{array}
\end{align*}
As in \cite{Gancedo2008}, for $\al \in [0, 2)$, we can write the constitutive law, \cref{e:SQG}$_2$, in the form,
\begin{align}\label{e:ConstitutiveLaw}
	\uu = \grad^\perp G^\al * \theta,
\end{align}
where
\begin{align}\label{e:Greensalpha}
	G^\al(\x) :=
	\begin{cases}
		\dfrac{c_\al}{\abs{\x}^\al} &\text{ if } 0 < \al < 2, \\[12pt]
		c_0 \log \abs{\x} &\text{ if } \al = 0
	\end{cases}
\end{align}
is the Green's function for $-(-\Delta)^{1 - \frac{\al}{2}}$ with $c_\al = \Gamma(\frac{\al}{2})/(\pi 2^{2 - \al} \Gamma(1 - \frac{\al}{2}))$ for $0 < \al < 2$ and $c_0 = (2 \pi)^{-1}$. Observe that
\begin{align}\label{e:alCal}
    \begin{split}
    \lim_{\al \to 0^+} \al c_\al
        &= \lim_{\al \to 0^+}
            \frac{2 \frac{\al}{2} \Gamma(\frac{\al}{2})}
            {4 \pi \Gamma(1)}
        = \frac{\lim\limits_{\beta \to 0} \beta \Gamma(\beta)}{2 \pi}
        = \frac{\Res_\Gamma(0)}{2 \pi} 
        = \frac{1}{2 \pi}
        = c_0,
    \end{split}
\end{align}
since the $\Gamma$ function has a simple pole of residue $1$ at the origin. This means, in particular, that $\al c_\al$ is continuous on $[0, \iny)$ and bounded over $\al \in [0, 1]$.

We are interested here in patch solutions to \SQG, a special type of weak solution in which $\theta(t)$ is the characteristic function of a domain. Our special concern in this paper is a subclass of patch data for which $\theta(t)$ is periodic in one direction. Periodic patches, as we develop them, encompass both periodic \textit{layers} (two-sided \textit{fronts}) and periodic copies of a single bounded domain. See the illustrations in \cite{AHK}, for the Euler equations, the $\al = 0$ case. 

\begin{remark*}
    For concreteness, we choose periodicity in the $x_1$-direction with a period of $1$.   
\end{remark*}

More precisely, we make the following definitions:

\begin{definition}\label{D:WeakSolution}
    Let $\theta \in L^\iny([0, T] \times \R^2)$ and suppose that  $\uu := -\grad^\perp (-\Delta)^{-(1 - \frac{\al}{2})} \theta \in L^\iny([0, T] \times \R^2)$. Then $\theta(t, \x)$ is a weak solution of \SQG with initial scalar $\theta_0$ if for any test function $\phi \in C^1_c([0, T] \times \R^2)$,
    \begin{align}\label{weak:01}
        \int_0^T \int_{\R^2} \theta (t, \x)
            \brac{\prt_t \phi + \uu(t, \x)
                 \cdot \nabla \phi (t, \x)} \, d \x \, dt
            &= \int_{\R^2}
                \brac{(\theta \phi)(t, \x)
                    - \theta_0(\x) \phi(0, \x)} \, d \x.
    \end{align}
    We call a weak solution \textit{periodic}, or, more fully, \periodic, if
    \begin{align*}
        \theta(t, \x + (n, 0)) = \theta(t, \x)
            \text{ for all }(t, \x) \in [0, T] \times \R^2.   
    \end{align*}
\end{definition}

\begin{definition}\label{D:PatchSolution}
Fix $a_1, a_2 \in \R$. We say that a weak solution to \SQG is an \SQG-\textit{patch} solution or simply a \text{patch} on $[0, T] \times \R^2$ if at time $t \in [0, T]$ it is of the form
\begin{align}\label{e:PatchForm}
	\theta(t, \x)
		= \Patch(\Omega(t), a_1, a_2)
		:=
		\begin{cases}
			a_1& \text{if } \x \in \Omega(t), \\
			a_2& \text{if } \x \in (\Omega(t))^{c},
		\end{cases}
\end{align}
for some domain $\Omega(t)$ in $\R^2$.
\end{definition}

\begin{remark*}
Because
$
    \grad^\perp (-\Delta)^{-(1 - \frac{\al}{2})} c
        = (-\Delta)^{-(1 - \frac{\al}{2})} \grad^\perp c
        = 0
$
for any constant $c$,
\begin{align*}
    \grad^\perp &(-\Delta)^{-(1 - \frac{\al}{2})}
            \Patch(\Omega(t), a_1, a_2)
        = \grad^\perp (-\Delta)^{-(1 - \frac{\al}{2})}
            [\Patch(\Omega(t), a_1, a_2) - a_2] \\
        &= \grad^\perp (-\Delta)^{-(1 - \frac{\al}{2})} 
            [\Patch(\Omega(t), a_1 - a_2, 0)]
        = \grad^\perp (-\Delta)^{-(1 - \frac{\al}{2})}
            (a_1 - a_2) \CharFunc_{\Omega(t)}.
\end{align*}
Also, the value of $a_1 - a_2$ plays little role in our analysis so, losing no generality, we assume that $a_1 = 1$, $a_2 = 0$, making our patches of the form,
\begin{align}\label{e:PatchData}
\theta(t, \x)
		= \CharFunc_{\Omega(t)}.
\end{align}
\end{remark*}

We will formulate, for any $\al \in (0, 1]$,  a contour dynamics equations (CDE) for \periodic patches that are compactly supported in the vertical direction. Our main result, stated precisely in \cref{T:ExistenceAlphaLt1}, is as follows:
\begin{theorem*}[Main result, roughly stated]
	For $\al \in (0, 1)$, let $\bgamma_0$ be the boundary of a periodic domain in $\R^2$ bounded in the vertical direction, and having an $H^m$ boundary, $m \ge 3$. There exists a time $T > 0$ for which a unique periodic solution $\bgamma$ to the CDE for \SQG exists in $C([0, T]; H^m)$ with $\prt_t \bgamma \in C([0, T]; L^\iny) \cap L^\iny([0, T]; H^{m - 1})$ and $\bgamma(0) = \bgamma_0$.
\end{theorem*}

The case $\al = 1$, which is not covered in \cref{T:ExistenceAlphaLt1}, is the subject of a future work. In preparation for that work, we include $\al = 1$ in as many of our estimates as feasible.

\begin{remark}\label{R:WeakNotUnique}
    We stress that our uniqueness results apply to solutions to the CDE,
    but not to weak solutions to \SQG, for which uniqueness is not known. That is, we cannot rule out the possibility that a patch at time zero evolves over time to a weak solution that is not a patch; we can show only that exactly one such solution can remain a patch.
    (More precisely, see \cref{P:WeakIfCDE,C:WeakIfCDE}.)
\end{remark}

\subsection{Euler equations: $\alpha = 0$}\label{S:TypesOfSolutions}

Much of our analysis for \SQG for $\al \in (0, 1]$ applies to the Euler equations, $\al = 0$, as well, but in that case the stronger result in \cite{AHK} can be obtained. In \cite{AHK}, three distinct types of weak solution are defined, Type 1 being a solution to $\al$-SQG (for $\al = 0$) on $[0, T] \times \R^2$ with periodic initial data, Type 2 being a solution on a periodic strip ($\Pi$, defined below in \cref{e:Pi}), and Type 3 again being a solution on $[0, T] \times \R^2$ for a single bounded domain, but with the constitutive law adapted to obtain a velocity field that corresponds to periodic copies. Much effort was spent in \cite{AHK} showing the equivalence of these three types of weak solutions, taking great advantage of the techniques developed in \cite{AfendikovMielke2005,GallaySlijepcevic2014,GallaySlijepcevic2015} to work with Type 2 solutions.

In large part because the constitutive law, \SQGConst, for $\al = 0$ gives $\uu$ one full derivative of regularity over that of $\theta$, it is possible to obtain uniqueness of weak solutions in which one only has $\theta$ and $\uu$ in $L^\iny(\R^2)$. One still has $\curl \uu = \theta$, but the constitutive law no longer holds, being replaced by the Serfati identity. This result is due to Serfati \cite{Serfati1995Bounded} (and see \cite{AKLL2015} for a fuller version of the proof). Combined with the results in \cite{AfendikovMielke2005,GallaySlijepcevic2014,GallaySlijepcevic2015}, this makes it possible to put the three types of weak solutions on equal footing without specializing to patch data.

Working with periodic patch data in \cite{AHK}, a contour dynamics equation (CDE) is developed for each type of solution, and it is shown that any such CDE solution is a weak solution. Because uniqueness of weak solutions is known to hold, it immediately follows that the three types of CDE solution are equivalent. This simplifies the subsequent analysis of the boundary regularity of the patch, in which it is shown that $C^{1, \eps}$ regularity is propagated for all time, and allows the results to be stated in terms of weak solutions to the Euler equations. In contrast, see \cref{R:WeakNotUnique}.

Moreover, for the 2D Euler equations, the velocity field is log-Lipschitz, which gives a unique flow map along which vorticity is transported. For any time, the flow map induces a diffeomorphism of $\R^2$ to $\R^2$ that maps any point to its transported position at that time, which prevents self-intersection of a patch boundary. (This also uses that weak solutions are unique and that any CDE solution is a weak solution.) By contrast, for $\al \in (0, 1]$, we  need to incorporate a measure of non-self-intersection into the energy arguments, as in \cref{S:SelfIntersection}.

\subsection{Prior work}
Rodrigo in \cite{Rodrigo2004,Rodrigo2005} considered the evolution of a simplified model for SQG with spatially periodic patches of halfspace-type; that is, with $\theta$ taking one value above a horizontally periodic curve and another value below the curve.  Such solutions are known as  sharp fronts.  Rodrigo's model was formed by  approximating the Green's function rather than using the properly  periodized Green's function.  Fefferman and Rodrigo then considered analytic sharp fronts for a related model problem in \cite{FeffermanRodrigo2011} (again not considering the exact periodic Green's function).

Bounded patch solutions have been studied for the exact SQG evolution, beginning with Gancedo \cite{Gancedo2008}.  This work left open the question of uniqueness for $\al = 1$, subsequently settled by  Cordoba, Cordoba, and Gancedo \cite{CordobaCordobaGancedo2018}. Gancedo and Patel established a blowup criterion for SQG patches,  among other results in \cite{GancedoPatel2021}. In \cite{GancedoNguyenPatel2022}, Gancedo, Nguyen, and Patel extended existence theory for patches to low regularity, allowing the initial bounded patch to have  large curvature (the patch boundary is taken to lie $H^{2+s}$ with $s<1/2$).

Front solutions have been considered in the series of papers \cite{HunterShu2018,HunterShuZhang2018,HunterShuZhang2020Nonlinearity, hunterCMS, hunterPAA, hunterDCDS} by Hunter, Shu, and Zhang.  This work began with the development of approximate models for the evolution of sharp SQG fronts, and the  development of existence theory for these approximate models  \cite{HunterShu2018, HunterShuZhang2018}.  Subsequently, for the full SQG equations of motion, the contour dynamics equation was developed for the motion of sharp fronts  
\cite{HunterShuZhang2020Nonlinearity, hunterCMS}.  One difference between the contour dynamics equations developed in  \cite{HunterShuZhang2020Nonlinearity, hunterCMS} 
and that of the present work is that they consider only curves which are a graph with respect to the horizontal, while we consider more general parameterized curves.  Hunter, Shu, and Zhang proved the existence of global front solutions with small, smooth data in
\cite{hunterPAA}, \cite{hunterDCDS}.
Subsequently, Ai and Avadanei gave a refined
analysis of the front equation, still in the graph case, giving a local well-posedness 
result for data of any size and a global well-posedness result for small data in a 
low-regularity setting \cite{aiAvadanei}.

In this paper, we adapt the approach of Gancedo in \cite{Gancedo2008} to cover the case of periodic patches.

\subsection{Organization of this paper} Our periodic patches can equivalently be treated as patches on a periodic strip: in \cref{S:Lifting} we describe how to move back and forth between the two settings. We derive expressions for the constitutive law for a periodic scalar in \cref{S:ConstitutiveLawScalar}, which we specialize to patch data in \cref{S:ConstitutiveLawPatch}. In \cref{S:CDE}, we derive the CDE for a periodic patch and state the definition of what we mean by a periodic solution to the CDE in \cref{S:CDESolution}. That a periodic CDE solution is a weak solution is shown in \cref{S:CDEIsWeak}.

The key differences between the analysis of the single, bounded-domain patches of \cite{Gancedo2008} and the periodic patches that we study arise from the difference in the Green's function that lies at the heart of the constitutive law. We describe these differences in  \cref{S:AdaptingGancedo}, and explain our approach to adapting the analysis of \cite{Gancedo2008}. A key aspect of the argument in \cite{Gancedo2008} involves a function $F$ that measures how close to self-intersecting the boundary of a patch is---so as to ensure a finite time up to which self-intersection is avoided. For a periodic patch, self-intersection effectively also occurs if the patch can no longer arise as a periodic copy of a single patch: in \cref{S:SelfIntersection}, we describe how we incorporate that type of self-intersection into the function $F$. (Another type of self-intersection occurs for a multiply connected domain when two components of the boundary intersect, a possibility that we control, for finite time, as the last step in our proof of existence.)

We write the CDE equation as $\prt_t \bgamma = L(\bgamma)$, where $L$ is the CDE solution operator. We give detailed estimates on $L$ in \cref{S:LEstimates}. In \cref{S:WellPosednessAlphalt1}, we give the proof of our main result, the short time existence and uniqueness of solutions to the periodic CDE with $H^k$ boundary regularity, $k \ge 3$, for $\al \in (0, 1)$.

In \cref{A:RalBounds}, we give bounds on the singularities and growth at infinity of a key portion of the Green's function from which the constitutive law for a patch is derived. Finally, we give a number of utility lemmas in \cref{A:SomeLemmas}.

\section{Periodizing and Lifting}\label{S:Lifting}

\subsection{Periodizing functions and domains}

To treat \periodic patches, we define the periodic strip,
\begin{align}\label{e:Pi}
	\Pi := \brac{-\tfrac{1}{2}, \tfrac{1}{2}} \times \R \text{ with }
			\set{-\tfrac{1}{2}} \times \R
			\text{ identified with }
			\set{\tfrac{1}{2}} \times \R.
\end{align}
We also define the same subset of $\R^2$ without identifying its sides, setting
\begin{align*} 
	\Pi_p := \brac{-\tfrac{1}{2}, \tfrac{1}{2}} \times \R
        \subseteq \R^2.
\end{align*}

We define the one-dimensional lattice,
\begin{align}\label{e:L}
    \L
		:= \Z \times \set{0}, \quad
	\L^*
		:= \L \setminus (0, 0),
\end{align}
and use the convenient abbreviation
\begin{align}\label{e:nj}
    \n_j := (j, 0)
\end{align}
for the points in $\L$.

In moving back and forth from functions or domains defined in $\Pi$ and their periodic analogs in $\R^2$, we view $\R^2$ as a covering space for $\Pi$, with the covering map,
\begin{align}\label{e:CoveringMap}
    \Cover \colon \R^2 \to \Pi, \quad
	\Cover(x_1, x_2) = (x_1 - \lfloor x_1 + \tfrac{1}{2} \rfloor, x_2).
\end{align}

\begin{definition}\label{D:Identification}
    We will often make, usually without comment, the identification of $\x = (x_1, x_2)$ in $\Pi$ with $\bxi \x := \x = (x_1, x_2) \in \Pi_p \subseteq \R^2$. Such identifications will be done in the context of periodic functions, so the choice of whether $(x_1, \tfrac{1}{2}) = -(x_1, \tfrac{1}{2})$ in $\Pi$ corresponds to $\pm (x_1, \tfrac{1}{2})$ in $\Pi_p$ will be irrelevant. Also, for two points $\x, \y \in \Pi$, we define $\x - \y = \Cover(\bxi \x - \bxi \y)$.
\end{definition}

For any function $f$ on $\Pi$ we define $\Rep(f)$ on $\R^2$ by
\begin{align*}
	\Rep(f)(\x)
		:= f \circ \Cover(\x)
        = f(x_1 - \lfloor x_1 + \tfrac{1}{2} \rfloor, x_2).
\end{align*}
This produces a \periodic version of the function $f$ on $\Pi$---a ``replication'' of $f$.

For any $\Omega_\Pi \subset \Pi$, we also define $\Rep (\Omega_\Pi) \subseteq \R^2$ by
\begin{align*}
	\Rep (\Omega_\Pi)
        := \Cover^{-1}(\Omega_\Pi)
		= \set{\x \in \R^2 \colon (x_1 - \lfloor x_1 + \tfrac{1}{2} \rfloor, x_2) \in \Omega_\Pi},
\end{align*}
and we note that
\begin{align*}
	\Rep (\CharFunc_{\Omega_\Pi}) = \CharFunc_{\Rep (\Omega_\Pi)}.
\end{align*}

Finally, for measurable $f$ on $\R^2$, we define
\begin{align*}
    \Rep_{\R^2}(f)(\x)
        &= \sum_{\bm{\ell} \in \L} f(\x - \bm{\ell}).
\end{align*}
The sum may diverge, but if $f \in L_c^\iny(\R^2)$ then $\Rep_{\R^2}(f)(\x)$ will be finite for all $\x \in \R^2$.

Moreover, we see that if $\Cover(\Omega) = \Omega_\Pi$ then
\begin{align*} 
    \Rep_{\R^2}(\CharFunc_{\Omega})
        &= \Rep(\CharFunc_{\Omega_\Pi}).
\end{align*}
This also allows us to define
\begin{align*} 
    \Rep_{\R^2}(\Omega):= \Rep(\Omega_\Pi),
\end{align*}
whenever $\Omega$ is a lift of $\Omega_\Pi$, noting, then, that $\Rep_{\R^2}(\Omega)$ does not depend upon the particular lift.

\subsection{Lifting domains and paths}
A \textbf{path} in the topological space $X$ is a continuous map from an interval $I$ to $X$. Because each of the paths we utilize will be either a closed path or a path in $\R^2$ that projects via $\Cover$ to a closed path in $\Pi$, we define our paths as follows:

\begin{definition}\label{D:PathLift}
	A \textbf{(closed) path} in the topological space $X$ is a continuous map from the unit circle $\T$ to $X$. A \textbf{(quasi-closed)} path in $\R^2$ is a map $\bgamma$ from $\T$ to $\R^2$ that is continuous except possibly at one point $\z_0 \in \T$, at which $\bgamma(\z_0-) - \bgamma(\z_0+) \in \L$. The quasi-closed path $\bgamma$ in $\R^2$ is a \textbf{lift} or \textbf{lifting} of the closed path $\bgamma_\Pi$ in $\Pi$ if $\Cover \circ \bgamma = \bgamma_\Pi$. 
\end{definition}

\cref{L:Lift} follows as in Section 2.4 of \cite{AHK}, where it was, however, expressed in terms of complex contour integrals.

\begin{lemma}\label{L:Lift}
    Let $\bgamma_\Pi \colon \T \to \Pi$ be a finite-length $C^1$ closed path in $\Pi$ with initial point $\x_{0, \Pi} \in \Pi$. For any $\x_0 \in \Cover^{-1}(\x_{0, \Pi})$ there exists a unique lift to a quasi-closed $C^1$ path $\bgamma$ on $\R^2$ with initial point $\x_0$. If $\bgamma_\Pi$ is $C^1$ then so is $\bgamma$. For any continuous scalar function $f$ on $\Pi$,
    \begin{align}\label{e:PathLiftIntScalar}
        \int_\T f(\bgamma_\Pi(\eta)) \, \abs{\bgamma_\Pi'(\eta)} \, d \eta
            &= \int_\T f \circ \Cover(\bgamma(\eta)) \, \abs{\bgamma'(\eta)}
                \, d \eta
    \end{align}
    and for any continuous vector-valued function $\vv$ on $\Pi$,
    \begin{align}\label{e:PathLiftIntVector}
        \int_\T \vv(\bgamma_\Pi(\eta)) \cdot \bgamma_\Pi'(\eta) \, d \eta
            &= \int_\T \vv \circ \Cover(\bgamma(\eta)) \cdot \bgamma'(\eta)
                \, d \eta.
    \end{align}
\end{lemma}

Suppose that $\Omega \subset \R^2$ is such that for some $\Omega_\Pi \subseteq \Pi$, $\Cover(\Omega) = \Omega_\Pi$ and $\Cover|_\Omega$ is injective. The condition that $\Cover|_\Omega$ be injective is equivalent to requiring that
\begin{align*} 
    \Omega \cap (\Omega + \L^*) = \emptyset.
\end{align*}
This yields what we will refer to as a pseudo-lift of $\Omega_\Pi$ to $\Omega$; such pseudo-lifts are non-unique.

We wish to apply \cref{L:Lift} to lift a parameterization $\bgamma_j$ of each boundary component of $\Omega_\Pi \subset \Pi$
in such a way that the boundary components of some pseudo-lift $\Omega$ are parameterized by the lifts of $\bgamma_j$. This is not always possible, because the homotopy class of $\Pi$ and $\R^2$ are not the same (see the example depicted in \Cref{f:SuitableLifting}).

Nonetheless, we can come close enough by following the approach in Section 2.4 of \cite{AHK}. For this purpose, it will be useful to extend our definition of a path to allow it to apply to all boundary components of a domain at once.

\begin{definition}\label{D:ChainLift}
    A \textbf{chain} in the topological space $X$ is a finite collection of paths in $X$. Let $\bgamma_{\Pi, 1}$, $\dots$, $\bgamma_{\Pi, n}$ be $n$ closed paths in $\Pi$. We parameterize the chain by a disjoint union of $n$ copies of $\T$, so that for integrable $f$,
    \begin{align*}
        \Tn
            &:= \bigsqcup_{j = 1}^n \T, \quad
        \int_\Tn f(\bgamma_\Pi(\eta)) \, d \eta:=
            \sum_{j = 1}^n \int_\T
                f(\bgamma_{\Pi, j}(\eta)) \, d \eta.
    \end{align*}
    Abusing notation, we define the iterated integral,
    \begin{align*}
        \int_\Tn \int_\Tn f(\bgamma_\Pi(\beta),
              \bgamma_\Pi(\eta))
                \, d \beta \, d \eta
            &:= \sum_{j = 1}^n
            \int_\T \int_\T f(\bgamma_{\Pi, j}(\beta),
              \bgamma_{\Pi, j}(\eta))
                \, d \beta \, d \eta.
    \end{align*}
    A chain $\bgamma \colon \Tn \to \R^2$ for which $\Cover \circ \bgamma = \bgamma_\Pi$ is a \textbf{lift} or \textbf{lifting} of the chain $\bgamma_\Pi \colon \Tn \to \Pi$.

    Let $\Omega_\Pi$ be a bounded domain in $\Pi$ with $C^1$ boundary, $\prt \Omega$, with a finite number of boundary components, $\Gamma_1, \cdots, \Gamma_n$. Orient each $\Gamma_j$  so that its unit normal vector $\n$ points outward, and parameterize each $\Gamma_j$ by a $C^1$ path $\bgamma_{\Pi, j}$ with the property that its unit tangent vector $\BoldTau = \n^\perp := (-n_2, n_1)$. Define the chain $\bgamma_\Pi \colon \Tn \to \Pi$ as the collection of those $n$ paths.
\end{definition}

The following is a re-expression of Lemma 2.10 of \cite{AHK} (illustrated in \Cref{f:SuitableLifting}):

\begin{lemma}\label{L:LiftedBoundaryIntegral}
    Let $\Omega_\Pi$ be a bounded domain in $\Pi$ having $C^1$ boundary, and let $\bgamma_\Pi \colon \Tn \to \Pi$ be a chain parameterizing $\prt \Omega_\Pi$ as in \cref{D:ChainLift}. There exists a bounded connected set $\Omega \subseteq \R^2$, such that for any continuous vector-valued function $\vv$ on $\Pi$,
    \begin{align*}
        \int_{\prt \Omega_\Pi} \vv \cdot \BoldTau
            = \int_{\prt \Omega} (\vv \circ \Cover) \cdot \BoldTau, \quad
        \int_{\prt \Omega_\Pi} \vv \cdot \n
            = \int_{\prt \Omega} (\vv \circ \Cover) \cdot \n,
    \end{align*}
    where $\BoldTau$, $\n$ are the unit tangent, unit normal vector, respectively (oriented so that $(\n, \BoldTau)$ are in the standard orientation). Moreover, there exists a lift $\bgamma$ of $\bgamma_\Pi$ to a chain in $\R^2$ with $\image \bgamma \subseteq \prt \Omega$, for which we have,    
    as in \cref{e:PathLiftIntVector},
    \begin{align*}
        \int_\Tn \vv(\bgamma_\Pi(\eta)) \cdot \bgamma_\Pi'(\eta) \, d \eta
            &= \int_\Tn \vv \circ \Cover(\bgamma(\eta)) \cdot \bgamma'(\eta)
                \, d \eta
            = \int_\Tn \Rep_\Pi(\vv)(\bgamma(\eta)) \cdot \bgamma'(\eta)
                \, d \eta.
    \end{align*}
\end{lemma}

\begin{figure}[ht]
\IncludeFigure{0.2}{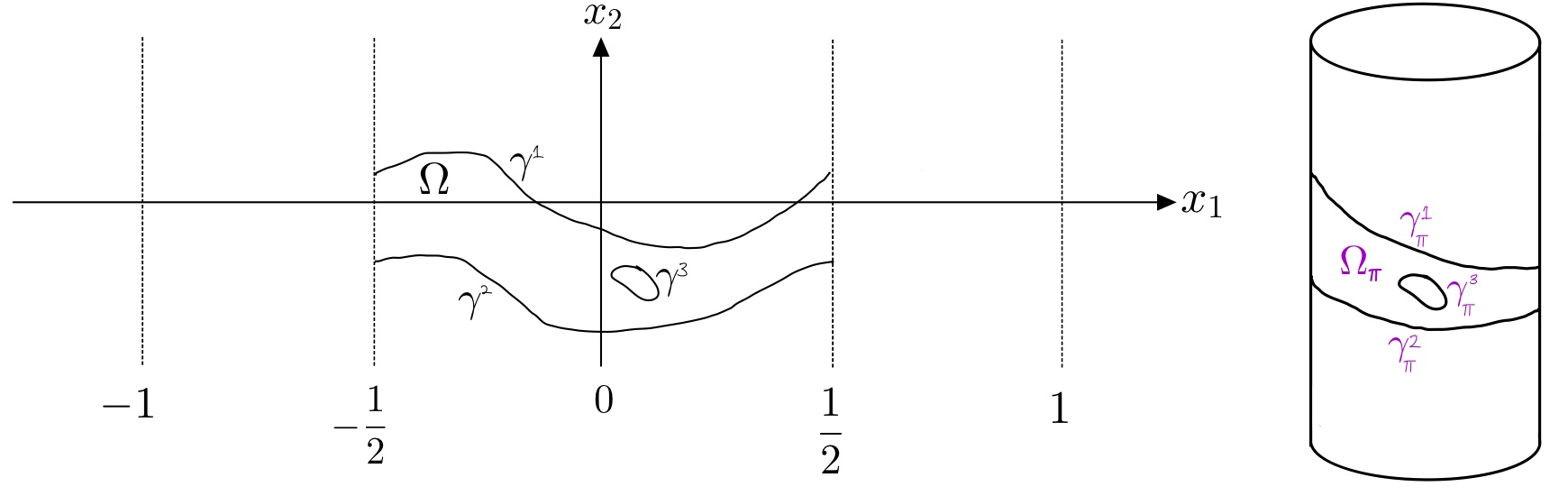}
		
\vspace{-1.2em}
	
\caption{\textit{Example of suitable lift of a non-simply connected domain}}\label{f:SuitableLifting}
\end{figure}

\begin{definition}\label{D:SuitableRep}
    We call $\Omega$ of \cref{L:LiftedBoundaryIntegral} a \textbf{suitable representation} of $\Omega_\Pi$
    and $\bgamma$ a \textbf{suitable lift} of $\bgamma_\pi$.
\end{definition}

\begin{remark*}
    In the context of \cref{L:LiftedBoundaryIntegral}, if $f$ is a continuous scalar-valued function on $\Pi$, it may be that $\image \bgamma \subsetneq \prt \Omega$ so $\int_{\prt \Omega_\Pi} f \ne \int_{\prt \Omega} f \circ \Cover$, because, as explained above Lemma 2.10 of \cite{AHK}, when $\Omega_\Pi$ is not simply connected, $\prt \Omega$ may include portions that are vertical line segments that do not derive from $\prt \Omega_\Pi$ (as occurs in the example in \Cref{f:SuitableLifting}).  By contrast, in integrating the tangential or normal components of a vector-valued function the integrals along these line segments cancel in pairs (as they do for the complex contour integrals in \cite{AHK}). 
\end{remark*}

\subsection{Notation and function spaces}

\begin{remark}[On notation]\label{R:Notation}
    If $f(z_1, \dots, z_n) $ is a function of $n$ variables, we will occasionally find it useful to write $\prt_j f$ for the derivative with respect to the $j^{th}$ variable. So, for instance, $\prt_2 f(t, \eta) = \prt_\eta f(t, \eta)$, but $\prt_2 f(t, \beta - \eta) = - \prt_\eta f(t, \beta -\eta) = \prt_\beta f(t, \beta -\eta)$.
\end{remark}

For any $r \in (0, 1]$, we define the space $C^{0, r}(\Tn)$ to be the space of fractional \Holder continuous functions on $\Tn$, the norm being
\begin{align*}
    \norm{f}_{C^{0, r}(\Tn)}
        &:= \norm{f}_{L^\iny(\Tn)} + \norm{f}_{\dot{C}^{0, r}(\Tn)}, \quad
    \norm{f}_{\dot{C}^{0, r}(\Tn)}
        := \sup_{\beta \ne \eta} \frac{\abs{f(\beta) - f(\eta)}}{\abs{\beta - \eta}^r}.
\end{align*}
When $r \in (0, 1)$ we also write $C^r(\Tn)$ for $C^{0, r}(\Tn)$. When $r = 1$ we get the space of  Lipschitz continuous functions on $\Tn$, $Lip(\Tn) = C^{0, 1}(\Tn)$, $\dot{Lip}(\Tn) = \dot{C}^{0, 1}(\Tn)$.

%
%
\section{The constitutive law for a scalar $\theta$}\label{S:ConstitutiveLawScalar}

\noindent 
In this section, we present the explicit form of the constitutive law \cref{e:ConstitutiveLaw} in three contexts: first, for a single compactly supported scalar field $\theta$; second for a periodic scalar of the form $\Rep_{\R^2}(\theta)$ on $\R^2$; third, for a compactly supported $\theta$ in the periodic strip $\Pi$. The function $\theta$ need not be for patch data---in the process of deriving the CDEs in \cref{S:CDE}, we will then specialize the constitutive law to the form it takes for patch data.

\subsection{Constitutive law for $\theta$ compactly supported}
Let $\theta \in L_c^\iny(\R^2)$. In that case, we can define the stream function $\psi := G^\al * \theta$ for $\uu$, which solves
\begin{align}\label{e:GreensEqOnePatchA}
 	-(-\Delta)^{1 - \frac{\al}{2}} \psi = \theta,
\end{align}
and then set $\uu = \grad^\perp \psi$. Then
\begin{align}\label{e:ConstLawSingle}
	\uu
		&= \grad^\perp (G^\al * \theta)
		= (\grad^\perp G^\al) * \theta
        = K^\al * \theta,
\end{align}
where,
\begin{align}\label{e:Kalpha}
	K^\al(\x)
		&:= \grad^\perp G^\al(\x)
		= - \frac{\al c_\al \x^\perp}{\abs{\x}^{2 + \al}}
            \text{ for } \al \in (0, 1]
\end{align}
and $K^0(x) = (2 \pi \abs{x})^{-2} x^\perp$. For $\al \in (0, 1)$, $K^\al$ is locally integrable and has sufficient decay at infinity so that
\begin{align}\label{e:uLinfSingle}
    \norm{\uu}_{L^\iny}
        &\le \norm{\CharFunc_{B_1(0)} K^\al}_{L^1} \norm{\theta}_{L^\iny}
            + \norm{\CharFunc_{B_1(0)^C} K^\al}_{L^\iny} \norm{\theta}_{L^1}
        \le C \norm{\theta}_{L^1 \cap L^\iny},
\end{align}
so $\uu \in L^\iny$. However, $K^1$ is not locally integrable, and the convolution must be treated as a singular integral operator.  (The case $\al \in (1, 2)$, treated in \cite{CCCGW2012}, is even more singular.)

\subsection{Constitutive law for $\theta$ periodic}\label{S:ConstForPeriodic}

Now, in place of \cref{e:GreensEqOnePatchA}, we want to start with $\theta \in L_c^\iny(\R^2)$ and solve
\begin{align}\label{e:GreensEqR2}
 	-(-\Delta)^{1 - \frac{\al}{2}} \psi = \Rep_{\R^2}(\theta)
\end{align}
by finding a \periodic Green's function, $G^\al_p$, defined on $\R^2$ so that
\begin{align}\label{e:PeriodicBSLaw}
    \psi = G^\al_p * \theta, \quad
    \uu = \grad^\perp \psi = K^\al_p * \theta,
\end{align}
where
\begin{align}\label{e:Kalp}
    K^\al_p := \grad^\perp G^\al_p.
\end{align}

The straightforward way of obtaining $G^\al_p$ and $K^\al_p$ is to start with $G^\al_p$ and then obtain $K^\al_p$ as in  \cref{e:Kalp}. Since $\psi$ is supposed to solve \cref{e:GreensEqR2}, we would expect of $G^\al_p$ that, at least formally,
\begin{align*}
    \psi(\x)
        &= G^\al * \Rep_{\R^2}(\theta)
        = \sum_{j \in \Z} G^\al * \theta(\x - \n_j)
        = \sum_{j \in \Z}
            \int_{\R^2} G^\al(\y) \theta(\x - \n_j - \y) \, d \y \\
        &= \sum_{j \in \Z}
            \int_{\R^2} G^\al(\y - \n_j) \theta(\x - \y) \, d \y
        = \sum_{j \in \Z}
            G^\al(\cdot - \n_j) * \theta(\x)
        = G^\al_p * \theta,
\end{align*}
where $\n_j$ is defined in \cref{e:nj} and
\begin{align}\label{e:GSum}
    \begin{split}
    G^\al_p(\x)
	   &= G^\al(\x) + \sum_{j \in \Z^*}
					\brac{G^\al(\x - \n_j)}
	   = G^\al(\x) + \sum_{j \in \Z^*}
					\brac{G^\al(\x - \n_j) - \frac{1}{\abs{\n_j}^\al}} \\
	   &= \frac{1}{|\x|^{\al}} + \sum_{j \in \Z^*} 
					\brac{\frac{1}{|\x- \n_j|^{\al}} -  \frac{1}{|\n_j|^{\al}}}.
    \end{split}
\end{align}
Here, we have subtracted $\abs{\n_j}^{-\al}$ to allow better convergence of the sum, but even doing so, one finds that the sum converges pointwise only if $\al > \tfrac{1}{2}$. When it does converge, \cref{e:Kalp} gives the corresponding constitutive kernel.

Since the straightforward approach limits us to the range $\al > \tfrac{1}{2}$, we instead take the approach in \cite{AHK}, which is to first obtain a constitutive kernel $K^\al_p$, then from it obtain $G^\al_p$.  So in place of the sum in \cref{e:GSum} for the Green's function, we sum in a similar manner for the kernel $K^\al_p$, writing,
\begin{align}\label{e:KandH}
    \begin{split}
        K^\al_p
            &:= K^\al + H^\al, \\
        H^\al(\x)
            &:= \sum_{j \in \Z^*} K^\al(\x - \n_j).
        \end{split}
\end{align}
(For $\al = 0$, the sum in \cref{e:KandH}$_2$ does not converge, so in \cite{AHK}, the sum was rearranged to pair the terms for $j$ and $-j$. For $\al > 0$, we have sufficient convergence without rearranging the terms, nor would summing in pairs improve the rate of convergence. Also, for $\al = 0$, the resulting sum can be viewed as a complex analytic function, which yields an explicit and convenient expression in \cite{AHK} for $K^0_p$ and $G^0_p$, but that approach is not available for $\al > 0$.)

\begin{remark}
    In our notation, a subscript $p$ indicates that a function is $1$-periodic in $x_1$ as a function on $\R^2$, and so also makes sense as a function on $\Pi$. In \cref{e:KandH}, neither $K^\al$ nor $H^\al$ are periodic, although their sum, $K^\al_p$, is periodic. This is a formal statement, but once we establish the convergence of the sum defining $H^\al(\x)$ we will see that it holds rigorously.
\end{remark}

Since for $j > 2 \abs{\x}$,
\begin{align}\label{e:KTermsConverge}
    &\abs{K^\al(\x - \n_j)}
        = \al c_\al
            \abs[\bigg]{ 
                \frac{(\x - \n_j)^\perp}
                    {\abs{\x - \n_j}^{\al + 2}}
            }
        \le \frac{\al c_\al}{\abs{\x - \n_j}^{\al + 1}}
        \le \frac{C}{j^{1 + \al}},
\end{align}
for any $\al > 0$, the sum defining $H^\al$ in \cref{e:KandH} converges uniformly on compact subsets of 
\begin{align*}
    U
        &:= \R^2 \setminus \L^*,
\end{align*}
where $\L^* = \Z^* \times \set{0}$ as in \cref{e:L}. 
Notice that any derivative of $K^\al(\x - \n_j)$ of order $k \ge 0$ 
decays like $C \abs{x}^{-(\al + 1 + k)};$
then, writing
\begin{align*}
    H^\al(x)
        &= \lim_{N \to \iny} H_N(\x), \quad
        H_N(x) := \sum_{\substack{j = -N\\j \ne 0}}^N
                K^\al(\x - \n_j),
\end{align*}
we see that all the derivatives of $H_N$ 
converge uniformly on compact subsets of $U$ as well. Since each 
$K^\al(\x - \n_j)$ is in $C^\iny(U)$, 
it follows that $H^\al \in C^\iny(U)$. 
Moreover, $\dv K^\al(\x - \n_j) = 0$ on $U$ 
so $\dv H^N = 0$ on $U$ and so in the limit we have $\dv H^\al = 0$ 
on $U$.

Now let
\begin{align*}
    U^*
        := \R^2 \setminus \set{\x \in \R^2 \colon \abs{x_1} \ge \tfrac{1}{2}, x_2 = 0}
        \subseteq U.
\end{align*}
Then $U^*$ is simply connected and $\dv \H^\al = 0$ on $U^*$, so there exists a stream function $R^\al$ on $U^*$ for $H^\al$; that is,
\begin{align*}
    H^\al := \grad^\perp R^\al, \quad
    R^\al(0) = 0
\end{align*}
on $U^*$, where we have fixed the value of $R^\al$ at the origin for uniqueness. We can write the stream function in the form,
\begin{align}\label{e:RAsPathIntegral}
    R^\al(\x) = \int_{\bgamma(\x)} (H^\al)^\perp \cdot d \BoldTau,
\end{align}
where $\bgamma$ is any $C^1$ path in $U^*$ from the origin to $\x \in U^*$ and $\BoldTau$ is the unit tangent vector.

Further, each $K^\al(\x - \n_j)$ is tangential to any ball $B_r(\n_j)$ centered at $\n_j$, and because $K^\al(\x - \n_j)$ is divergence-free, it follows that
\begin{align*}
    \int_\bgamma (K^\al(\x - \n_j))^\perp \cdot d \BoldTau
        &= 0
\end{align*}
for any closed path $\bgamma$ in $\R^2$ for which $\n_j \notin \image \bgamma$. From this it follows that 
\begin{align*} 
    \int_\bgamma (H^\al)^\perp \cdot d \BoldTau
        &= \lim_{N \to \iny} \int_\bgamma \left(H_N\right)^\perp
            \cdot d \BoldTau = 0
\end{align*}
for any closed path $\bgamma$ in $\R^2$ not intersecting $\L^*$. This is enough to show that defining $R^\al(\x)$ as in \cref{e:RAsPathIntegral} for all $\x \in U$ provides a continuous extension of $R^\al$ from $U^*$ to all of $U$.

Finally, we set
\begin{align}\label{e:Galp}
    G^\al_p := G^\al + R^\al.
\end{align}

\begin{lemma}\label{L:uLinfPeriodic}
   If $\theta \in L^\iny_c(\R^2)$ then $\uu := K^\al_p * \theta \in L^\iny(\R^2)$. 
\end{lemma}
\begin{proof}
    Let $\theta \in L^\iny_c(\R^2)$ be 
    such that the support of $\theta$ is contained 
    in the ball $B_R(0)$. Since $\uu$ is $1$-periodic in $x_1$, we can assume that $\x \in \Pi_p$. Then,
    \begin{align*}
        \abs{\uu(\x)}
            &\le \sum_{\abs{j} \le 2 R} \abs{K^\al(\cdot - \n_j) * \theta(\x)}
                + \sum_{\abs{j} > 2 R} \abs{K^\al(\cdot - \n_j) * \theta(\x)}.
    \end{align*}
    But from \cref{e:uLinfSingle},
    \begin{align*}
        \sum_{\abs{j} \le 2 R} \abs{K^\al(\cdot - \n_j) * \theta(\x)}
            &\le C (2R + 1) \norm{\theta}_{L^1 \cap L^\iny}
            \le C (2R + 1)^3 \norm{\theta}_{L^\iny},
    \end{align*}
    and
    \begin{align*}
        &\sum_{\abs{j} > 2 R} \abs{K^\al(\cdot - \n_j) * \theta(\x)}
            \le \sum_{\abs{j} > 2 R}
                    \norm{K^\al(\x - \n_j)}_{L^\iny_\x(\supp \theta)}
                        \norm{\theta}_{L^1} \\
            &\qquad
            \le \sum_{\abs{j} > 2 R}
                    \frac{\al c_\al}{\abs{\x - \n_j}^{1 + \al}}
                        \norm{\theta}_{L^1}
            \le \sum_{\abs{j} > 2 R}
                    \frac{\al c_\al}{(\abs{j}/2)^{1 + \al}}
                        \norm{\theta}_{L^1}
            \le C (2R + 1)^2 \norm{\theta}_{L^\iny}.
            \qedhere
    \end{align*}
\end{proof}

We have the following bounds on $\grad^n R^\al$:
\begin{prop}\label{P:BoundsOnR}
    Let $\al \in (0, 1)$. For any $\x \in \R^2$ with $x_2 \ne 0$,
    \begin{align}\label{e:RalSimplified}
        \begin{split}
        R^\al(\x)
           &= c_\al \sum_{j \in \Z^*}^\iny 
                \brac{
                    \frac{1}{\abs{j}^\al}
                - 
                        \frac{1}{(x_2^2 + (x_1 - j)^2)
                            ^{\frac{\al}{2}}}
                }.
        \end{split}
    \end{align}
    Moreover, for all $\x \in \R^2 \setminus \L^*$,
    \begin{align}\label{e:RBoundOnR2MinusLStar}
        \abs{\grad^n R^\al(\x)}
            &\le
            \frac{C}{\dist(\x, \L^*)^{n + \al}}
            +
            C \begin{cases}
                \abs{x_2}^2 & \text{if } n = 0, \\
                \abs{x_2}^n & \text{if } n > 0.
            \end{cases}
    \end{align}
\end{prop}
\begin{proof}
    See \cref{A:RalBounds}.
\end{proof}

\subsection{Constitutive law for $\theta$ in the periodic strip $\Pi$}\label{S:ConstForPi}

To develop the constitutive law in $\Pi$, let us first define convolution on $\Pi$: If $f, g \colon \Pi \to \R$ then
\begin{align*}
    f * g (\x)
        &:= \int_\Pi f(\x - \y) g(\y) \, d \y,
\end{align*}
where $\x - \y$ is treated as defined in \cref{D:Identification}.

To work in $\Pi$, we take $\theta_\Pi \in L_c^\iny(\Pi)$ and solve, on $\R^2$,
\begin{align}\label{e:GreensEqPi}
 	-(-\Delta)^{1 - \frac{\al}{2}} \psi = \Rep(\theta_\Pi)
\end{align}
by finding a Green's function $G^\al_\Pi$ defined on $\Pi$ so that
\begin{align}\label{e:BSLawOnPi}
    \psi = \Rep(G^\al_\Pi * \theta_\Pi), \quad
    \uu = \grad^\perp \psi = \Rep(K^\al_\Pi * \theta_\Pi),
\end{align}
where
\begin{align}\label{e:KalpPi}
    K^\al_\Pi := \grad^\perp G^\al_\Pi.
\end{align}

This process is almost identical to that expressed in \crefrange{e:GreensEqR2}{e:Kalp}. Indeed, if we repeat the same analysis in \cref{S:ConstForPeriodic} on $\Pi$, we will find that 
\begin{align}\label{e:GKR2AndPi}
    G^\al_p = G^\al_\Pi \circ \Cover, \quad
    K^\al_p = K^\al_\Pi \circ \Cover,
\end{align}
which relies on the observation that $K^\al_p$ and $G^\al_p$ are \periodic.

%
%
\section{The constitutive law for a patch}\label{S:ConstitutiveLawPatch}

\noindent In this section, we develop the constitutive law as specialized to patch data.

\subsection{Constitutive law for a single patch}
First, we consider the interpretation of the integrals in the constitutive law for a single (that is, non-periodic) patch. As in  (10) of \cite{Gancedo2008}, we have the following:

\begin{lemma}\label{L:ConstLawSingle}
    Let $\Omega$ be a bounded domain in $\R^2$ with $C^1$ boundary parameterized by the chain $\bgamma \colon \Tn \to \R^2$. Let $\al \in (0, 1]$ and $\x \in \R^2$. If $\al = 1$, assume that $\x \notin \prt \Omega$.
    Then
    \begin{align}\label{e:BSLawSinglePatch}
    	\uu(\x)
    		&  = \int_\Tn
    					G^\al(\x - \bgamma(\beta))\prt_\beta \bgamma(\beta)
    					\, d \beta
    \end{align} 
    satisfies the constitutive law in \SQGConst.
\end{lemma}
\begin{proof}
We derive this as for the Euler equations (see Chapter 8 of \cite{MB2002}). 
From \cref{e:ConstLawSingle},
\begin{align*}
    \uu(\x)
        &= \int_{\R^2} K^\al(\x - \y) \theta(\y) \, d \y
        = \int_{\Omega} \grad^\perp G^\al(\x - \y) \, d \y.
\end{align*}
For $\al \in [0, 1)$, $K^\al$ is locally integrable, so the first equality holds for all $\x \in \R^2$. For $\al = 1$, as long as $\x \notin \prt \Omega$, $\abs{\x - \y}$ is bounded away from zero for $\y$ in the support of $\theta$, and again the first equality holds.

Integrating by parts, and parameterizing the boundary by arc length, $s$,
\begin{align*}
    \int_{\Omega} \prt_i G^\al(\x - \y) \, d \y
        &= \int_{\prt \Omega} G^\al(\x - \y(s))
            n^i(\y(s)) \, d s,
\end{align*}
so
\begin{align*}
    \uu(\x)
        &= \int_\Omega
            (-\prt_2 G^\al(\x - \y), \, \prt_1 G(\x - \y)) \, d \y
        = \int_{\prt \Omega}
            G^\al(\x - \y(s)) \n^\perp(\y(s)) \, d s \\
        &= \int_\Tn
				G^\al(\x - \bgamma(\beta))
                \prt_\beta \bgamma(\beta) \, d \beta.
\end{align*}
In the last equality, we used that
\begin{align*}
    \n^\perp(\y(s)) \, d s
        &= \BoldTau(\y(s)) \, d s
        = \prt_\beta \bgamma(\beta)
					\, d \beta.
        \qedhere
\end{align*}
\end{proof}

\subsection{Constitutive law for a patch in $\Pi$}
We can derive the analog of \cref{e:BSLawSinglePatch} in the same manner as for a non-periodic patch, giving \cref{L:ConstLawPi}.
\begin{lemma}\label{L:ConstLawPi}
    Let $\Omega_\Pi$ be a domain in $\Pi$ with $C^1$ boundary parameterized by the chain $\bgamma_\Pi \colon \Tn \to \Pi$. Let $\al \in (0, 1]$ and $\x \in \Pi$. If $\al = 1$, assume that $\x \notin \prt \Omega_\Pi$.
    Then
    \begin{align}\label{e:uGintPi}
    	\uu(\x)
    		&=
    				\int_\Tn
    				G^\al_\Pi(\x - \bgamma_\Pi(\beta))
                        \prt_\beta \bgamma_\Pi(\beta)
    					\, d \beta
    \end{align}
    satisfies the constitutive law in \cref{e:BSLawOnPi}, which we can write as
    \begin{align*} 
        \uu
            &= - \grad^\perp (-\Delta)^{-(1 - \frac{\al}{2})}
                \CharFunc_{\Rep(\Omega_\Pi)}.
    \end{align*}
\end{lemma}

\begin{remark}\label{R:IntegrandMeaning}
    In \cref{e:uGintPi}, $\x - \bgamma_\Pi(\beta)$ is treated as in \cref{D:Identification}, and                       $\prt_\beta \bgamma_\Pi(\beta)$ as a vector in $\R^2$. The integral then produces a vector field on $\Pi$.
\end{remark}

\subsection{Constitutive law for a periodic patch}\label{S:ConstLawPeriodic}
To obtain the constitutive law for a periodic patch, we have to assume that the patch is of the form $\Rep(\Omega_\Pi)$ for some domain in $\Pi$. For then, if $\Omega$ is a suitable representation of $\Omega_\Pi$, the periodic patch will be formed by replicating copies of $\Omega$ translated horizontally by every integer value, giving $\Rep(\Omega)$,  and those replications will not overlap each other. This leads to \cref{L:ConstLawPeriodic}.

\begin{lemma}\label{L:ConstLawPeriodic}
    Let $\Omega_\Pi$ be a domain in $\Pi$ with $C^1$ boundary parameterized by the chain $\bgamma_\Pi \colon \Tn \to \Pi$. Let $\Omega$ be a suitable representation of $\Omega_\Pi$ and $\bgamma \colon \Tn \to \R^2$ a suitable lift of $\bgamma_\Pi$ as in \cref{D:SuitableRep}. Let $\al \in (0, 1]$ and $\x \in \R^2$. If $\al = 1$, assume that $\x \notin \prt \Omega$.
    Then
    \begin{align}\label{e:uGintPeriodic}
    	\uu(\x)
    		&=
    				\int_\Tn
    				G^\al_p(\x - \bgamma(\beta))\prt_\beta \bgamma(\beta)
    					\, d \beta
    \end{align}
    satisfies the constitutive law in \cref{e:BSLawOnPi}, which we can write as
    \begin{align*}
            \uu
            &= \grad^\perp (-\Delta)^{-(1 - \frac{\al}{2})}
                \CharFunc_{\Rep(\Omega_\Pi)}.
    \end{align*}
\end{lemma}
\begin{proof}
    Follows by applying \cref{L:LiftedBoundaryIntegral} to \cref{L:ConstLawPi}.
\end{proof}

%
%
\section{The contour dynamics equation (CDE)}\label{S:CDE}

\subsection{CDE for a single patch}\label{S:CDESingle}

The CDE for a single patch in $\R^2$ is developed in \cite{Gancedo2008}, paralleling that for the Euler equations. We can write it, for $\al \in (0, 1)$, as
\begin{align*} 
	\prt_t \bgamma(t, \eta)
		&= \int_\Tn
				G^\al(\bgamma(t, \eta) - \bgamma(t, \beta))
					\prt_\beta \bgamma(t, \beta)
					\, d \beta,
\end{align*}
where $\bgamma$ is a $C^1$-parameterization of the patch boundary.

\subsection{CDE for a patch in the periodic strip $\Pi$}\label{S:CDEOnPi}

For $\al \in (0, 1)$, \cref{{L:ConstLawPi}} leads directly to a CDE by writing $\x \in \prt \Omega_\Pi(t)$ as $\x = \bgamma_\Pi(t, \eta)$ for some $\eta \in \Tn$. This yields
\begin{align}\label{e:PreliminaryCDEPi}
	\prt_t \bgamma_\Pi(t, \eta)
		&= \int_\Tn
				G^\al_\Pi(\bgamma_\Pi(t, \eta) - \bgamma_\Pi(t, \beta)) \prt_\beta \bgamma_\Pi(t, \beta)
					\, d \beta.
\end{align}
See \cref{R:IntegrandMeaning} for the interpretation of the integrand in \cref{e:PreliminaryCDEPi} and the expressions that follow. We note that this derivation is formal; if we assume sufficient time regularity of the boundary, though, it follows immediately.

For $\al = 1$, however, the integral in \cref{e:PreliminaryCDEPi} is not convergent, since the singularity in $G^\al_\Pi$, which comes from $G^\al$, is like $C \abs{\eta - \beta}^{-1}$ in the contour integral, which is not integrable in one-dimension. Following \cite{Resnick,Rodrigo2005}, since it is only the normal component of the velocity field that determines the motion of the boundary, we can subtract from the velocity field $\prt_t \bgamma_\Pi(t, \eta)$ in \cref{e:PreliminaryCDEPi} any multiple of the tangential velocity field on the boundary. This is done in a limiting sense as $\x \to \prt \Omega_\Pi(t)$ in \cite{Gancedo2008}, or we can simply do it formally, giving a CDE that applies for all $\al \in (0, 1]$:
\begin{align}\label{e:CDEOnPiProto}
	\prt_t \bgamma_\Pi(t, \eta)
		&= \int_\Tn
				G^\al_\Pi(\bgamma_\Pi(t, \eta) - \bgamma_\Pi(t, \beta))
					(\prt_\eta \bgamma_\Pi(t, \eta) - \prt_\beta \bgamma_\Pi(t, \beta))
					\, d \beta.
\end{align}

The term $\prt_\eta \bgamma_\Pi(t, \eta)$ contributes nothing to the normal velocity field, since $\prt_\beta \bgamma_\Pi(t, \eta) \cdot \prt_\beta \bgamma_\Pi(t, \eta)^\perp = 0$. For $\al = 1$, it eliminates the singularity in $G^\al_\Pi$ that appears in $G^\al$ in the decomposition of $G^\al_\Pi$ in \cref{e:Galp}. (Also see \cref{R:IBPTerm}.)

Making the change of variables, $\beta \mapsto \eta - \beta$,
we can write \cref{e:CDEOnPiProto} as 
\begin{align*}
    \prt_t \bgamma_\Pi(t, \eta)
        &= \int_\Tn G^\al_\Pi(\bgamma_\Pi(t, \eta)
                - \bgamma_\Pi(t, \eta - \beta))
			(\prt_2 \bgamma_\Pi(t, \eta)
                - \prt_2 \bgamma_\Pi(t, \eta - \beta))
			\, d \beta.
\end{align*}
Or, paralleling (13) of \cite{Gancedo2008}, we have $\prt_t \bgamma_\Pi(t) = L_\Pi(\bgamma_\Pi(t))$, where
\begin{align}\label{e:CDEOnPi}
    L_\Pi(\bgamma_\Pi)(\eta)
        &:= \int_\Tn G^\al_\Pi(\bgamma_\Pi(\eta)
                - \bgamma_\Pi(\eta - \beta))
			(\prt_\eta \bgamma_\Pi(\eta)
                - \prt_\eta \bgamma_\Pi(\eta - \beta))
			\, d \beta.
\end{align}

\subsection{CDE for a periodic patch}\label{S:CDEPeriodic}

To obtain a CDE equation for a periodic patch, we assume that we have a solution $\theta$ to \SQG for which, up to some time $T > 0$,
\begin{align*}
    \theta(t) = \CharFunc_{\Rep(\Omega_\Pi(t))},
\end{align*}
where $\prt \Omega_\Pi(t)$ is $C^1$. We then let $\Omega(t)$ be a suitable lifting of $\Omega_\Pi(t)$ as in Section 2.4 of \cite{AHK} and let $\bgamma(t)$ be a $C^1$ parameterization of it.

In light of \cref{e:GKR2AndPi}, we should expect that the CDE equation for a periodic patch should have the same form as that for a patch on $\Pi$. Hence, we simply define the CDE exactly as in \cref{{e:CDEOnPi}}, now for a curve $\bgamma(t)$ in $\R^2$, so that $\prt_t \bgamma(t) = L(\bgamma(t))$, where
\begin{align*}
	L(\bgamma)(\eta)
		&:= \int_\Tn
				G^\al_p(\bgamma(\eta) - \bgamma(\eta - \beta))
					(\prt_\eta \bgamma(\eta) - \prt_\eta \bgamma(\eta - \beta))
					\, d \beta.
\end{align*}
In the analysis of periodic CDE patches that follow, we will write the operator $L$ more concisely as,
\begin{align}\label{e:CDEPeriodic}
    \begin{split}
    L(\bgamma)(\eta)
        &= \int_\Tn
            G^\al_p(\bdelta_\beta(\eta))
                \prt_\eta \bdelta_\beta(\eta)
                \, d \beta, \\
    \bdelta_\beta(\eta)
        &:= \bgamma(\eta) - \bgamma(\eta - \beta).
    \end{split}
\end{align}

\begin{remark}\label{R:IBPTerm}
    Subtracting $\prt_\eta \bgamma_\Pi(t, \eta)$ from \cref{e:PreliminaryCDEPi} to give \cref{e:CDEOnPiProto} was done to remove the singularity when $\al = 1$. It also plays a critical role for $\al \in (0, 1)$, however, for it allows us to apply \cref{C:IBP} to integrate by parts expressions involving the operator $L$.
\end{remark}

%
%
\section{CDE solution}\label{S:CDESolution}

Given that the CDEs for \SQG on $\Pi$ and on $\R^2$ are essentially identical, so too are the definitions of a CDE solution. The preparation of the initial data, however, is a somewhat delicate matter.

\begin{definition}\label{D:CDESolutionOnPi}
    For $\al \in (0, 1)$, let $\Omega_{\Pi, 0}$ be a bounded domain in the periodic strip $\Pi$ with $C^1$ boundary, and let $\bgamma_{\Pi, 0}$ be a $C^1$-parameterization of $\prt \Omega_{\Pi, 0}$ as in \cref{D:ChainLift}. A non-self intersecting chain $\bgamma_\Pi \colon [0, T] \times \Tn \to \Pi$ solving $\prt_t \bgamma_\Pi(t) = L_\Pi(\bgamma_\Pi(t))$, $\bgamma_\Pi(0) = \bgamma_{\Pi, 0}$,
    is called a \textbf{CDE solution on $\Pi$} to \SQG with initial patch $\Omega_{\Pi, 0}$. We also let $\Omega_\Pi(t)$ be the domain defined by $\bgamma_\Pi(t)$ and, in accordance with \cref{e:uGintPi}, let
    \begin{align*}
        \uu_\Pi(t, \x)
    		&=
    				\int_\Tn
    				G^\al_\Pi(\x - \bgamma_\Pi(t, \beta))\prt_\beta \bgamma_\Pi(t, \beta)
    					\, d \beta.
    \end{align*}
\end{definition}

\begin{definition}\label{D:CDESolutionPeriodic}
    For $\al \in (0, 1)$, let $\Omega_{\Pi, 0}$ be as in \cref{D:CDESolutionOnPi}.
    Let $\Omega_0$ be a suitable representation of $\Omega_{\Pi, 0}$ and $\bgamma_0 \colon \Tn \to \R^2$ a suitable lift of $\bgamma_{\Pi, 0}$ as in \cref{D:SuitableRep}.
    A non-self intersecting chain $\bgamma \colon [0, T] \times \Tn \to \R^2$ solving $\prt_t \bgamma(t) = L(\bgamma(t))$, $\bgamma(0) = \bgamma_0$,
    is called a \textbf{periodic CDE solution} to \SQG with initial patch $\Omega_0$. We also let $\Omega(t)$ be the domain defined by $\bgamma(t)$ and, in accordance with \cref{e:uGintPeriodic}, let
    \begin{align*}
        \uu(t, \x)
    		&=
    				\int_\Tn
    				G^\al_p(\x - \bgamma(t, \beta))\prt_\beta \bgamma(t, \beta)
    					\, d \beta.
    \end{align*}
\end{definition}

\begin{remark}\label{R:SolsSatisfyConstLaws}
    It follows from \cref{L:ConstLawPi,L:ConstLawPeriodic} that the expressions for $\uu_\Pi$, $\uu$ in \cref{D:CDESolutionOnPi,D:CDESolutionPeriodic} satisfy the constitutive laws as stated in \cref{L:ConstLawPi}, \cref{L:ConstLawPeriodic}, respectively. Also, in \cref{D:CDESolutionPeriodic}, we do not prescribe the regularity of $\bgamma$ (beyond being $C^1$) or $\prt_t \bgamma$ in \cref{D:CDESolutionPeriodic} (and similarly for \cref{D:CDESolutionOnPi}); the regularity will follow from our proof of existence (\cref{T:ExistenceAlphaLt1}), and will require us to explore the properties of the operator $L$ in some depth.
\end{remark}

Although we will not prove it, because we require the initial patch domain in \cref{D:CDESolutionPeriodic} to be a suitable lift, $\bgamma(t)$ continuing to avoid self-intersection is equivalent to $\Omega(t)$ continuing to be a suitable lift. From this it is easy to see that the following holds: 

\begin{prop}\label{P:CDESolEquivalence}
    If $\bgamma_\Pi$ is a CDE solution $\bgamma_\Pi$ on $[0, T] \times \Pi$ as in \cref{D:CDESolutionOnPi} then $\bgamma$ as given in \cref{D:CDESolutionPeriodic} is a periodic CDE solution. The converse holds as well.
\end{prop}

\section{CDE solutions are weak solutions}\label{S:CDEIsWeak}

\noindent In this section, we show that a periodic CDE solution $\bgamma$ as in \cref{D:CDESolutionPeriodic} gives a weak \SQG-\textit{patch} solution to \SQG as in \cref{D:PatchSolution}.

\begin{prop}\label{P:WeakIfCDE}
    Let $\al \in (0, 1]$ and suppose that  $(\uu, \theta)$ are such that
    $\theta \in L^\iny([0, T] \times \R^2)$ is of the
    form given in \cref{e:PatchForm}, where $\prt \Omega(t)$ has $C^1$ boundary for all
    $t \in [0, T]$, and $\uu$ satisfies the constitutive law,
     $\uu = K^\al_p * \theta$.
    Then $\Rep_{\R^2}(\theta)$ is  a weak periodic patch solution of \SQG as in \cref{D:PatchSolution} with initial scalar $\Rep_{\R^2}(\theta_0)$.
\end{prop}

Before proving \cref{P:WeakIfCDE}, we give its main application:
\begin{cor}\label{C:WeakIfCDE}
Let $\bgamma$ be a $C^1$ periodic CDE solution to \SQG as in \cref{D:CDESolutionPeriodic}, let
$
    \theta(t)
        = \CharFunc_{\Omega(t)}
$,
and let $\uu$ be given by \cref{e:uGintPeriodic}. Then $\Rep_{\R^2}(\theta)$ is  a weak periodic patch solution of \SQG as in \cref{D:PatchSolution} with initial scalar $\Rep_{\R^2}(\theta_0)$.
\end{cor}
\begin{proof}
From \cref{{L:ConstLawPeriodic}}, $\uu = \grad^\perp (-\Delta)^{-(1 - \frac{\al}{2})} \Rep_{\R^2}(\theta) = K^\al_p * \theta$; thus, $\uu$ is as appears in \cref{D:WeakSolution} for the scalar $\Rep_{\R^2}(\theta)$. Then $(\uu, \theta)$ satisfy the hypotheses, and hence conclusion, of \cref{P:WeakIfCDE}.
\end{proof}

\begin{proof}[\textbf{Proof of \cref{P:WeakIfCDE}}]
Let $\phi \in C^1_c ([0, T] \times \R^2)$ and define, as in \cref{weak:01},
\begin{align*}
	I
		&:= \int_0^T \int_{\R^2} \Rep_{\R^2}(\theta(t, \x))
                \brac{\prt_t \phi+ \uu(t, \x) \cdot \nabla \phi (t, \x)}
                    \, d \x \, dt.
\end{align*}
Moreover, $\uu = K^\al_p * \theta \in L^\iny([0, T] \times \R^2)$ by \cref{L:uLinfPeriodic}.

From \cref{L:uIsCinf} below, 
we can see that $\uu$ has a measure-preserving flow map, $X \colon [0, T] \times \Rep_{\R^2}(\Omega_0) \to \Rep_{\R^2}(\Omega(t))$, with $X(t, \cdot) \in C^\iny(\Omega(t))$ for all $t \in [0, T]$. Being a flow map means that $\uu(t, X(t, \x)) = \prt_t X(t, \x)$ for all $\x \in \Rep_{\R^2}(\Omega_0)$.

Write $D/Dt$ for the material derivative with respect to $\uu$, so $(D/Dt) \phi(t, \x) := \prt_t \phi+ \uu(t, \x) \cdot \nabla \phi (t, \x)$. Because $X$ is measure-preserving, the change of variables $\x = X(t, \y)$ has Jacobian determinant 1, so, proceeding as on page 334 of \cite{MB2002},
\begin{align*}
    \int_0^T &\diff{}{t} \int_{\Rep_{\R^2}(\Omega(t))}
        \phi (t, \x)
        \, d \x \, dt
    = \int_0^T \diff{}{t} \int_{\Rep_{\R^2}(\Omega_0)}
        \phi (t, X(t, \y))
        \, d \y \, dt \\
    &= \int_0^T \int_{\Rep_{\R^2}(\Omega_0)}
        \diff{}{t} \phi (t, X(t, \y))
        \, d \y \, dt
    = \int_0^T \int_{\Rep_{\R^2}(\Omega(t))}
            \frac{D}{Dt}\phi (t, \y)
            \, d \y \, dt \\
    &= \int_0^T \int_{\R^2}
            \Rep_{\R^2}(\theta(t, \x)) \frac{D}{Dt}\phi (t, \x)
            \, d \x \, dt
    = I.
\end{align*}
We brought $d/dt$ inside the integral using Theorem 2.27 of \cite{Folland}. But also,
\begin{align*}
    \int_0^T &\diff{}{t} \int_{\Rep_{\R^2}(\Omega(t))}
                \phi (t, \x)
                \, d \x \, dt
        = \int_{\Rep_{\R^2}(\Omega(t))}
                \brac{\phi(t, \x)
                    - \phi(0, \x)} \, d \x \\
        &
        = \int_{\R^2}
                \brac{({\Rep_{\R^2}(\theta(t))} \phi)(t, \x)
                    - {\Rep_{\R^2}(\theta_0}(\x)) \phi(0, \x)} \, d \x,
\end{align*}
showing that $\Rep_{\R^2}(\theta)$ is  a weak solution with initial scalar $\Rep_{\R^2}(\theta_0)$.
\end{proof}

\begin{remark*}
    We know from \cref{R:graduBlowup} that $\grad \uu(\x)$ becomes infinite as $\x \to \prt \Omega$; nonetheless, the Jacobian $\det \grad X(t, \x) \equiv 1$, and that is sufficient to obtain equality in the integrations we made in the proof of \cref{P:WeakIfCDE}.
\end{remark*}

We used \cref{L:uIsCinf} above.

\begin{lemma}\label{L:uIsCinf}
    Let $\al \in (0, 1]$. Suppose that $\Omega_\Pi$ is a bounded domain in $\Pi$ with  $C^1$ boundary and 
    $\Omega$ is a suitable representation of $\Omega_\Pi$ as in \cref{D:SuitableRep}. Let
    $
        \theta
            = \CharFunc_{\Omega}
    $ and $\uu = K^\al_p * \theta$.
    Then $\uu \in C^\iny(\Rep(\Omega_\Pi))$ and $\uu \in C^\iny(\ol{\Rep(\Omega_\Pi)}^C)$.
\end{lemma}
\begin{proof}
    First assume that $\al \in (0, 1)$. Then
    \begin{align}\label{e:uKHDecomp}
        \uu(\x)
            &= \int_{\R^2} K^\al_p(\x - \y) \theta(\y) \, d \y
            = \int_{\R^2} K^\al(\x - \y) \theta(\y) \, d \y
                + \int_{\R^2} H^\al(\x - \y) \theta(\y) \, d \y.
     \end{align}

     Assume that $\x \in \Rep(\Omega_\Pi)$. Then $\x - \n_j \in \Omega$ for some $j \in \Z$, but $\x - \n_k \notin \Omega$ for any $k \ne j$. Since $\uu$ is $1$-periodic in $x_1$, we can assume, by translating $\Omega$ by $\n_j$, that $\x \in \Omega$ and $\x - \n_k \notin \Omega$ for all $k \ne 0$. Hence, $H^\al(\x - \cdot) \in C^\iny(\Rep(\Omega_\Pi))$, so the second integral on the right-hand side of \cref{e:uKHDecomp} is in $C^\iny(\Rep(\Omega_\Pi))$ with, for any multi-index $\xi$,
    \begin{align*}
        \abs[\bigg]{\int_{\R^2} D^\xi H^\al(\x - \y) \theta(\y) \, d \y}
            &\le \norm{D^\xi H^\al(\x - \cdot)}_{L^1(\Omega)}
                \norm{\theta}_{L^\iny}
            \le C(\xi).
    \end{align*}

    Let $r = \dist(\x_0, \prt \Omega)$. Then for $\x \in B_r(\x_0)$,
    \begin{align*}
        \vv(\x)
            &:= \int_{\R^2} K^\al(\x - \y) \theta(\y) \, d \y
            = \int_{B_r(\x)} K^\al(\x - \y) \, d \y
                + \int_{B_r(\x)^C \cap \Omega} K^\al(\x - \y) \, d \y \\
            &= \int_{B_r(\x)^C \cap \Omega} K^\al(\x - \y) \, d \y.
    \end{align*}
    We used here that $K^\al$ is radially symmetric, so the integral over $B_r(\x)$ vanishes.

    Then,
    \begin{align*}
        \prt_i \vv(\x)
            &= \int_{B_r(\x)^C \cap \Omega}
                \prt_i K^\al(\x - \y) \, d \y \\
            &\qquad
                + \lim_{h \to 0}
                    \frac{1}{h} \brac{
                        \int_{B_r(\x + h \e_i)^C \cap \Omega}
                            K^\al(\x - \y) \, d \y
                        - \int_{B_r(\x)^C \cap \Omega}
                            K^\al(\x - \y) \, d \y
                        }\\
            &=: \ww_1(\x) + \ww_2(\x).
    \end{align*}
    But,
    \begin{align*}
        \abs{\ww_1(\x)}
            &\le \norm{\prt_i K^\al(\x - \cdot)}_{L^1(\Omega)}
            \le C(\abs{\xi}, \dist(\x_0, \prt \Omega)),
    \end{align*}
    and
    \begin{align*}
        \abs{B_r(\x + h \e_i) \symdiff B_r(\x)}
            = O(2 \pi r h)
    \end{align*}
    as $h \to 0$, so
    \begin{align*}
        \abs{\ww_2(\x)}
            &\le \lim_{h \to 0}
                \int_{B_r(\x + h \e_i) \symdiff B_r(\x)}
                            \abs{K^\al(\x - \y)} \, d \y
            \\
            &\le \lim_{h \to 0}
                    \frac{1}{h}
                    O(2 \pi r h)
                    \frac{\al c_\al}{(r - h)^{1 + \al}}
            \le \frac{C}{r^\al}
            = \frac{C}{\dist(\x, \prt \Omega)^\al}.
    \end{align*}
    Although this bound on $\abs{\prt_i \vv(\x_0)}$ blows up as $\x_0 \to \prt \Omega$, we do have that $\vv$ and so $\uu$ are in $C^1(\Rep(\Omega_\Pi))$. Taking higher derivatives, we see that $\uu \in C^\iny(\Rep(\Omega_\Pi))$, and a similar argument gives $\uu \in C^\iny(\ol{\Rep(\Omega_\Pi)}^C)$.

    The proof for $\al = 1$ is essentially the same, except that the integrals must be considered in the principal value sense.
\end{proof}

\begin{remark}\label{R:graduBlowup}
    For $\al = 0$, Bertozzi and Constantin in \cite{ConstantinBertozzi1993} use their geometric lemma to show that $\grad \uu \in L^\iny(\R^2)$. Their argument relies on $\grad K^0(\x)$, which acts as the kernel of a singular integral operator, integrating to zero over circles centered at the origin. That property fails to hold, however, for $\grad K^\al(\x)$ when $\al \in (0, 1)$: in fact, $\grad \uu(\x)$ blows up as $\x \to \prt \Omega$.
\end{remark}

%
%
\section{Remarks on the periodic Green's function}\label{S:AdaptingGancedo}

\noindent The fundamental difference between our setup and that of Gancedo in \cite{Gancedo2008} is the Green's function---$G^\al$ in \cite{Gancedo2008} versus $G^\al_\Pi$ or $G^\al_p$ in this paper. Gancedo, who always writes the Green's function explicitly for $\al \in (0, 1]$ as $c_\al/\abs{\x}^{\al}$, takes advantage of the following key properties:

\begin{enumerate}
    \item
        $G^\al$ is singular only at $\x = 0$, where $G^\al(\x) \sim \abs{\x}^{-\al}$.

    \item
        $G^\al(-\x) = G^\al(\x)$.

    \item
        $G^\al(\x) = G^\al(\abs{\x})$; that is, $G^\al$ is radially symmetric.

    \item
        $G^\al(\x)$ is bounded (in fact, decays) as $\abs{\x} \to \iny$.
\end{enumerate}

These properties are most critical in the analysis of the CDE, where they appear in two forms:
\begin{align*}
    G^\al_p(\bgamma(\beta) - \bgamma(\eta))
        \text{ and }
    G^\al_p(\bgamma(\eta) - \bgamma(\eta - \beta))
        = G^\al_p(\bdelta_\beta(\eta)),
\end{align*}
where $\bdelta_\beta(\eta) := \bgamma(\eta) - \bgamma(\eta - \beta)$, as in \cref{e:CDEPeriodic}. In both cases, these factors appear in the integrand, with $\eta$, and in the second form also $\beta$, integrated over the domain, $\Tn$, of a chain $\bgamma$. $G^\al_p(\x)$ has a singularity like $\abs{\x - \n_k}^{-\al}$ at each lattice point $\n_j \in \L$, but as long as $\bgamma$ is a lift of the boundary $\bgamma_\Pi$ of a domain in $\Pi$, the arguments of $G^\al_p$ above will encounter only the singularity at $\eta = \beta$. Otherwise, two points on $\bgamma$ would differ by a lattice point $\n_j$, meaning that $\bgamma$ could not be a lift. Hence, in effect, we maintain property (1), at least in principle.

We have (2), but we do not have property (3).
In \cite{Gancedo2008}, property (3) (which implies property (2)) is used to conclude that certain integrands are perfect derivatives, so their integral over a closed curve vanishes. The first use is to demonstrate, as part of bounding the $H^3$ norm of $\bgamma$, that, supposing $\prt_t \bgamma = L(\bgamma)$,
\begin{align}\label{e:GancedoL2NormConserved}
    \diff{}{t} \norm{\bgamma(t)}^2_{L^2(\Tn)}
        &= 2 (\bgamma(t), L(\bgamma(t))
        = 0.
\end{align}
Gancedo starts by using property (2) to obtain
\begin{align*}
    \frac{1}{2} \diff{}{t} &\norm{\bgamma(t)}_{L^2(\Tn)}^2
        = \int_\Tn \int_\Tn
                G^\al_p(\bdelta_\beta(\eta))
                    \prt_\eta \bdelta_\beta(\eta)
                    \cdot \bdelta_\beta(\eta)
                    \, d \beta \, d \eta.
\end{align*}
He then uses that
\begin{align*}
    \int_\Tn \int_\Tn
        &G^\al(\bdelta_\beta(\eta))
            \prt_\eta \bdelta_\beta(\eta)
            \cdot \bdelta_\beta(\eta)
            \, d \beta \, d \eta 
        = \frac{c_\al}{2} \int_\Tn \int_\Tn
            \abs{\bdelta_\beta(\eta)}^{-\al}
                \prt_\eta \abs{\bdelta_\beta(\eta)}^2
                \, d \eta \, d \beta \\
        &= \frac{c_\al}{2 -\al}
            \int_\Tn \int_\Tn
            \prt_\eta \abs{\bdelta_\beta(\eta)}^{2 -\al}
                \, d \eta \, d \beta
        = 0,
\end{align*}
since the inner integral vanishes because the integrand is a perfect derivative in $\eta$.

The first part of this argument works in our setting as well since it uses only property (2), but the second part---which requires that $G(\x)$ be radially symmetric, the specific form being unimportant---fails in our setting.

Property (4) is used implicitly in \cite{Gancedo2008}, but we will need to account for the lack of decay of $G^\al_p$, as evidenced in $R^\al$ in \cref{P:BoundsOnR}. Because we assume that the patch at time zero is compactly supported in the vertical direction, we can control the growth of the patch, as the growth of $R^\al$ is only in the vertical direction.

In adapting the arguments in \cite{Gancedo2008}, we use the decomposition $G^\al_p = G^\al + R^\al$, and bound the terms involving $G^\al$ separately from those involving $R^\al$, since they succumb to slightly different techniques.

As the patch boundary evolves, the singularity in $R^\al$ at points in $\L^*$ must be avoided. To do this, we will modify in \cref{S:SelfIntersection} the function $F$ that Gancedo employs to measure the self-intersection of (one component of) the patch boundary. 

Furthermore, because we allow patches to be multiply connected, we must also control for two boundary components evolving to intersect. We do this at the end of \cref{S:CombiningTheBounds}.

%
%
\section{Measuring self-intersection}\label{S:SelfIntersection}

The initial data we consider for the patch boundary is smooth and
non-self-intersecting.  In addition to maintaining regularity at
positive times, we must demonstrate that we maintain the 
non-self-intersecting property at positive times.  This is a feature
of all existence theories for water waves, vortex sheets, and vortex 
patches, although the need for a non-self-intersection condition is 
sometimes obviated by considering the simpler geometry of graphs 
rather than general parameterized curves.  In the general setting,
following Wu's work on water waves \cite{wu2D}, we enforce the
chord-arc condition both on the initial data and at positive times.

As the first author explains in some detail in \cite{ambroseJMFM},
there are two primary ways of controlling the chord-arc quantity at
positive times.  The approach of \cite{ambroseJMFM} and other works
by the first author relies mostly on continuity: since solutions
initially satisfy the chord-arc condition, if the time derivative of
the chord-arc quantity is bounded, then the condition continues to
hold for some positive time interval.  The approach taken by Gancedo
in \cite{Gancedo2008} and in related works 
instead directly estimates the 
$L^{\infty}$-norm of the chord-arc quantity, by taking the limit
of $L^{p}$-norms as $p$ goes to infinity.

In the present work, 
we take a mixed approach; since many of our estimates are analogous
to estimates of \cite{Gancedo2008}, we begin by introducing an 
analogous quantity, $F,$ adapted to our periodic setting, 
for which we will make estimates of the 
$L^{\infty}$-norm.  A significant difference, however, is that in 
the present work there are $n$ boundary components of our patch/
layer regions, and we must also measure how near different 
components come to each other.  For this, we will follow 
\cite{ambroseJMFM} more closely, in that we will rely on the fact
that initially the boundary components are 
separated by a finite distance, and that their velocities are 
bounded.

We now proceed to describe the chord-arc quantities for 
individual boundary components, analogously to \cite{Gancedo2008}.
We will bound the distance between distinct components when the time
comes, specifically, in \Cref{S:CombiningTheBounds} below.

We define a function $F$ much as in (16) of \cite{Gancedo2008} that measures, inversely, how close to self-intersecting a path $\bgamma$ is. Letting 
$
    \bdelta_\beta(\eta)
            = \bgamma(\eta) - \bgamma(\eta - \beta)
$, as in \cref{e:CDEPeriodic}, for a given $C^1$ path $\bgamma$ in $\R^2$, define the functions $F, F_j \colon [-\pi, \pi]^2 \to \R$,
\begin{align}\label{e:FDef}
    \begin{split}
    F_0(\bgamma)(\beta, \eta)
        &=
        \begin{cases}
            \frac{\abs{\beta}}
                {\abs{\bdelta_\beta(\eta)}}
                    &\text{if } \beta \ne 0, \\
            \frac{1}{\abs{\prt_\eta \bgamma(\eta)}}
                    &\text{if } \beta = 0,
        \end{cases} \\
    F_j(\bgamma)(\beta, \eta)
        &= \frac{1}
                {\abs{\bdelta_\beta(\eta) - \n_j}}
            \text{ for } j \in \Z^*, \\[0.2em]
    F(\bgamma)(\beta, \eta)
        &= \max_{j \in \Z} F_j(\bgamma)(\beta, \eta).
    \end{split}
\end{align}
As long as $\bgamma$ is non-self-intersecting, $F_j$ is continuous for all $j \ge 0$, with $F_j(\bgamma)(-\pi, \eta) = F_j(\bgamma)(\pi, \eta)$.
We have the simple bound,
\begin{align}\label{e:ForcingFIntoRBound}
    \frac{1}{\dist(\bdelta_\beta(\eta), \L^*)}
        &= \max_{j \in \Z^*} F_j(\bgamma)(\beta, \eta)
        \le F(\bgamma)(\beta,\eta).
\end{align}

Moreover, $\bdelta_\beta$, $F_j$, and $F$ naturally extend to apply to a chain $\bgamma$.

\begin{remark}\label{R:FParameterization}
First, we note that although we do not label them as such, we
actually have one such $F$ quantity for each of our $n$
boundary components.  This will not be critical, as we bound all
of them in the same manner.  Next, we remark that
    the functions $F_j$, for $j \in \Z^*$, are naturally defined on $\Tn^2$, but $F_0$ and so $F$ are not. To make $F_0$ properly defined on $\Tn^2$, we could replace $\abs{\beta}$ in its definition by $\abs{\sin(\beta/2)}$. We use $\abs{\beta}$, however, to more closely align with \cite{Gancedo2008}. This will force us to use a specific parameterization of $\Tn$ by $\beta \in [-\pi, \pi]$, though we will not point this out in every instance in which it arises.
\end{remark}

\subsection{Bounding derivatives of $G^\al_p(\bdelta_\beta(\eta))$}

We will now put the chord-arc quantity $F$ to use.
We have the following, where repeated indices are implicitly summed over:
\begin{align*} 
    \begin{split}
    \prt_\eta G^\al_p(\bdelta_\beta(\eta))
        &= \prt_j G^\al_p(\bdelta_\beta(\eta))
            \prt_\eta \bdelta_\beta^j(\eta), \\
    \prt_\eta^2 G^\al_p(\bdelta_\beta(\eta))
        &= \prt_{jk} G^\al_p(\bdelta_\beta(\eta))
            \prt_\eta \bdelta_\beta^j(\eta)
                \prt_\eta \bdelta_\beta^k(\eta)
            + \prt_j G^\al_p(\bdelta_\beta(\eta))
                \prt_\eta^2 \bdelta_\beta^j(\eta), \\
    \prt_\eta^3 G^\al_p(\bdelta_\beta(\eta))
        &= \prt_{jk\ell} G^\al_p(\bdelta_\beta(\eta))
            \prt_\eta \bdelta_\beta^j(\eta)
                \prt_\eta \bdelta_\beta^k(\eta)
                \prt_\eta \bdelta_\beta^\ell(\eta) \\
        &\qquad
            + \prt_{jk} G^\al_p(\bdelta_\beta(\eta))
            \brac{\prt_\eta^2 \bdelta_\beta^j(\eta)
                \prt_\eta \bdelta_\beta^k(\eta)
                + \prt_\eta \bdelta_\beta^j(\eta)
                \prt_\eta^2 \bdelta_\beta^k(\eta)} \\
        &\qquad
            + \prt_{jk} G^\al_p(\bdelta_\beta(\eta))
                \prt_\eta^2 \bdelta_\beta^j(\eta)
                \prt_\eta \bdelta_\beta^k(\eta)
            + \prt_j G^\al_p(\bdelta_\beta(\eta))
                \prt_\eta^3 \bdelta_\beta^j(\eta) \\
        &= \prt_{jk\ell} G^\al_p(\bdelta_\beta(\eta))
            \prt_\eta \bdelta_\beta^j(\eta)
                \prt_\eta \bdelta_\beta^k(\eta)
                \prt_\eta \bdelta_\beta^\ell(\eta) \\
        &\qquad
            + 3 \prt_{jk} G^\al_p(\bdelta_\beta(\eta))
                \prt_\eta^2 \bdelta_\beta^j(\eta)
                \prt_\eta \bdelta_\beta^k(\eta)
            + \prt_j G^\al_p(\bdelta_\beta(\eta))
                \prt_\eta^3 \bdelta_\beta^j(\eta).
    \end{split}
\end{align*}

More generally, we can write, for $k \ge 0$,
\begin{align}\label{e:prtRalGen}
	\prt_\eta^k G^\al_p(\bdelta_\beta(\eta))
		&= \sum_{j = 1}^k \grad^j G^\al_p(\bdelta_\beta(\eta)) \cdot
			\sum_{\vec{\ell} \in V_{jk}} c_{\vec{\ell}} \prt_\eta^{\ell_1} \bdelta_\beta(\eta)
					\otimes \cdots \prt_\eta^{\ell_j} \bdelta_\beta(\eta),
\end{align}
where each $c_{\vec{\ell}}$ is a positive integer and, for $j \le k$,
\begin{align*}
	V_{jk}
		&:= \set{\vec{\ell} \in \set{1, \dots, k}^j
			\colon \ell_1 + \cdots + \ell_j = k},
	\quad
	V_{00} = 1.
\end{align*}

Setting $m_n = 2$ if $n = 0$, $m_n = n$ if $n > 0$, \cref{P:BoundsOnR,e:ForcingFIntoRBound} give
\begin{align}\label{e:prtRBound}
    \begin{split}
    \abs{\grad^n R^\al(\bdelta_\beta(\eta))}
        &\le \frac{C}{\dist(\bdelta_\beta(\eta), \L^*)^{n + \al}}
            +
            C \abs{\bdelta_\beta(\eta)}^{m_n} \\
        &\le C \left(F(\bgamma)(\beta, \eta)\right)^{n + \al}
            + C \norm{\bgamma}_{L^2}^{m_n}.
    \end{split}
\end{align}

Define
\begin{align}\label{e:Sn}
    \begin{split}
    S_0(\beta, \eta)
        &= S_0(\bgamma)(\beta, \eta)
        := \left(F(\bgamma)(\beta, \eta)\right)^\al
            + \norm{\bgamma}_{L^2}^2, \\
    S_n(\beta, \eta)
        &= S_n(\bgamma)(\beta, \eta)
        :=  \sum_{j = 1}^n
            \brac{\left(F(\bgamma)(\beta, \eta)\right)^{j + \al}
            + \norm{\bgamma}_{L^2}^j},
    \end{split}
\end{align}
for $n \ge 1$. Then \cref{e:prtRBound} gives
\begin{align}\label{e:gradjRalBound}
    \sum_{j = 1}^n \abs{\grad^j R^\al(\bdelta_\beta(\eta))}
        &\le C S_n(\beta, \eta)
        \text{ for } n \ge 0.
\end{align}

The derivatives of $G^\al$ are simply bounded in terms of $F$:
\begin{align}\label{e:gradGEstimate}
    \abs{\grad^n G^\al(\bdelta_\beta(\eta))}
        &\le \frac{C}{\abs{\bdelta_\beta(\eta)}^{n + \al}}
        = \pr{\frac{\abs{\beta}}{\abs{\bdelta_\beta(\eta)}}}^{n + \al}
        \frac{C}{\abs{\beta}^{n + \al}}
        \le C \frac{\norm{F(\bgamma)}_{L^\iny(\Tn^2)}^{n + \al}}
            {\abs{\beta}^{n + \al}}.
\end{align}
Then \cref{e:gradjRalBound,e:gradGEstimate} give
\begin{align}\label{e:gradGalpEstimate}
    \abs{\grad^n G^\al_p(\bdelta_\beta(\eta))}
        &\le C \frac{\norm{F(\bgamma)}_{L^\iny(\Tn^2)}^{n + \al}}
            {\abs{\beta}^{n + \al}}
            + \norm{S_n(\bgamma)}_{L^\iny(\Tn^2)}
        \le C \frac{\norm{S_n(\bgamma)}_{L^\iny(\Tn^2)}}
            {\abs{\beta}^{n + \al}},
\end{align}
and thus,
\begin{align}\label{e:absGpBound}
	\begin{split}
	&\abs{\prt_\eta^k G^\al_p(\bdelta_\beta(\eta))}
		\le C \sum_{j = 1}^k
				\abs{\grad^j G^\al_p(\bdelta_\beta(\eta))}
			\sum_{\vec{\ell} \in V_{jk}}
				\abs{\prt_\eta^{\ell_1} \bdelta_\beta(\eta)}
				\cdots
				\abs{\prt_\eta^{\ell_j} \bdelta_\beta(\eta)} \\
		&\qquad
		\le C \sum_{j = 1}^k
				\brac{\frac{\norm{F(\bgamma)}_{L^\iny(\Tn^2)}^{j + \al}}
            		{\abs{\beta}^{j + \al}}
					+ \norm{S_j(\bgamma)}_{L^\iny(\Tn^2)}
					}
			\sum_{\vec{\ell} \in V_{jk}}
				\abs{\prt_\eta^{\ell_1} \bdelta_\beta(\eta)}
				\cdots
				\abs{\prt_\eta^{\ell_j} \bdelta_\beta(\eta)} \\
		&\qquad
		\le C \norm{S_k(\bgamma)}_{L^\iny(\Tn^2)}
			\sum_{j = 1}^k
				\frac{1}
            		{\abs{\beta}^{j + \al}}
			\sum_{\vec{\ell} \in V_{jk}}
				\abs{\prt_\eta^{\ell_1} \bdelta_\beta(\eta)}
				\cdots
				\abs{\prt_\eta^{\ell_j} \bdelta_\beta(\eta)}.
	\end{split}
\end{align}

%
%
\section{Estimates on the CDE operator $L$}\label{S:LEstimates}

\noindent  In this section, we obtain a number of estimates on the operator $L$ of \cref{e:CDEPeriodic}. We give as many of these estimates as possible for all $\al \in (0, 1]$, even though a workable version of $L$ for $\al = 1$ involves a second term and a ``constant speed'' parameterization of $\bgamma$, as done in \cite{Gancedo2008}. But first, we establish a symmetric property of $L$ in \cref{P:LSymmetry}.

\begin{remark*}
	In what follows, we bring derivatives inside integrals over $\Tn$ and integrate by parts on $\Tn$. These are all justified by applying \cref{L:DerivInsideIntegral,C:IBP} to the estimates that follow these procedures, though we will not always state that explicitly.
\end{remark*}

\begin{prop}\label{P:LSymmetry}
	If $\bgamma$ is a quasi-closed $H^n(\T)$ chain in $\R^2$ for $n \ge 4$ then for all $m < n$,
    \begin{align}\label{e:ddtGammaHm}
        (\prt_\eta^m \bgamma, &\prt_\eta^m L(\bgamma))_{L^2(\Tn)}
            = \frac{1}{2} \int_\Tn \int_\Tn
    				\prt_\eta^m \pr{G^\al_p(\bdelta_\beta(\eta))
    					\prt_\eta \bdelta_\beta(\eta)}
                        \cdot \prt_\eta^m \bdelta_\beta(\eta)
    					\, d \beta \, d \eta.
    \end{align}
\end{prop}
\begin{proof}
    Using
    \begin{align*}
        \bmu(\beta, \eta)
            &:= \bgamma(\beta) - \bgamma(\eta),
    \end{align*}
    write the operator $L$ of \cref{e:CDEPeriodic} in the form of \cref{e:CDEOnPiProto}:
    \begin{align}\label{e:CDEOnPiDeltam}
    	L(\bgamma(\eta))
                = \int_\Tn
    				G^\al_p(\bmu(\beta, \eta))
    					\prt_2 \bmu(\beta, \eta)
    					\, d \beta.
    \end{align}

    Recalling the definition of iterated integrals in \cref{D:ChainLift}, we apply $\prt_\eta^m$ to both sides of \cref{e:CDEOnPiDeltam}, take the scalar product with $\prt_\eta^m \bgamma(\beta)$, and integrate over $I$ to give
    \begin{align*}
        \begin{split}
        (\prt_\eta^m \bgamma, &\prt_\eta^m L(\bgamma))_{L^2{}}
            = \int_\Tn \int_\Tn
    				\prt_\eta^m \pr{G^\al_p(\bmu(\beta, \eta))
    					\prt_2 \bmu(\beta, \eta)}
                        \cdot \prt_\eta^m \bgamma(\beta)
    					\, d \beta \, d \eta \\
            &= \int_\Tn \int_\Tn
    				\prt_\eta^m \pr{G^\al_p(\bmu(\eta, \beta))
    					\prt_2 \bmu(\eta, \beta)}
                        \cdot \prt_\eta^m \bgamma(\eta)
    					\, d \eta \, d \beta \\
            &= -\int_\Tn \int_\Tn
    				\prt_\eta^m \pr{G^\al_p(\bmu(\beta, \eta))
    					\prt_2 \bmu(\beta, \eta)}
                        \cdot \prt_\eta^m \bgamma(\eta)
    					\, d \eta \, d \beta.
        \end{split}
    \end{align*}
    In the second equality, we simply switched variable names $\eta$ and $\beta$, while in the final equality we used that $G^\al_p(-\x) = G^\al_p(\x)$ and $\bmu(\eta, \beta) = - \bmu(\beta, \eta)$. Taking the average of the second and fourth expressions, we see that
    \begin{align}\label{e:diffnormgammaHmsquared}
        (\bgamma, &L(\bgamma))_{L^2{}}
            = \frac{1}{2} \int_\Tn \int_\Tn
    				\prt_\eta^m \pr{G^\al_p(\bmu(\beta, \eta))
    					\prt_2 \bmu(\beta, \eta)}
                        \cdot \prt_\eta^m \bmu(\beta, \eta)
    					\, d \beta \, d \eta.
    \end{align}

    Now let
        $\bdelta_\beta(\eta) := \bgamma(\eta) - \bgamma(\eta - \beta)$
    as in \cref{e:CDEPeriodic}.
    Making the change of variables, $\beta \mapsto \eta - \beta$, but leaving $\eta$ unchanged, using that $\bgamma$ is quasi-closed
    and that $\prt_2 \bdelta_\beta(\eta) = \prt_\eta \bdelta_\beta(\eta)$, \cref{e:diffnormgammaHmsquared} becomes \cref{e:ddtGammaHm}.
\end{proof}

\begin{prop}\label{P:BoundsOnL}
    Let $\al \in (0, 1]$ and define the operator $L$ as in \cref{e:CDEPeriodic} (even for $\al = 1$). Let $\bgamma$ be an $H^n(\T)$ quasi-closed chain in $\R^2$ for $n \ge 3$. Then
\begin{align*}
	\norm{\prt_\eta^m L(\bgamma)}_{L^2(\Tn)}
		\le A_m(\bgamma, \al)
		:=
		\begin{cases}
			C \norm{S_m(\bgamma)}_{L^\iny(\Tn^2)}
				\sum_{j = 0}^{m + 1} \norm{\bgamma}_{H^{m + 1}(\Tn)}^j
				&\text{ if } \al \in (0, 1), \\[2ex]
			C \norm{S_m(\bgamma)}_{L^\iny(\Tn^2)}
				\sum_{j = 0}^{m + 1} \norm{\bgamma}_{H^{m + 2}(\Tn)}^j
				&\text{ if } \al = 1
		\end{cases}
\end{align*} 
holds for all $m \le n - 2$ and, if $\al \in (0, 1)$, for $m = n - 1$ as well.
\end{prop}
\begin{proof}
We have,
\begin{align*}
	\prt_\eta^m L(\bgamma)(\eta)
		&= \int_\Tn 
                \prt_\eta^m \pr{G^\al_p(\bdelta_\beta(\eta))
                    \prt_\eta \bdelta_\beta(\eta)}
                    \, d \beta \\
		&= \sum_{k = 0}^m \binom{m}{k}  \int_\Tn
                \prt_\eta^k (G^\al_p(\bdelta_\beta(\eta))
                    \prt_\eta^{m - k + 1} \bdelta_\beta(\eta))
                    \, d \beta.
\end{align*}
From \cref{e:gradGalpEstimate}, and applying Minkowski's integral inequality,
\begin{align*}
	\norm{G^\al(\bdelta_\beta(\eta)
                    \prt_\eta^{m + 1} \bdelta_\beta(\eta))}_{L^2_\eta}
    	&\le C \int_\Tn
				C \frac{\norm{S_0(\bgamma)}_{L^\iny(\Tn^2)}}
            {\abs{\beta}^\al}
				\abs{\prt_\eta^{m + 1} \bdelta_\beta(\eta)}
				\, d \beta, 
\end{align*}
and, using \cref{e:absGpBound}, for $k \ge 1$,
\begin{align*}
	&\norm{\prt_\eta^k (G^\al(\bdelta_\beta(\eta))
                    \prt_\eta^{m - k + 1} \bdelta_\beta(\eta))}_{L^2_\eta} \\
		&\qquad
		\le C \norm{S_k(\bgamma)}_{L^\iny(\Tn^2)}
			\sum_{j = 1}^k
			\int_\Tn
				\frac{1}{\abs{\beta}^{j + \al}}
            \abs{\prt_\eta^{m - k + 1} \bdelta_\beta(\eta))}
			\sum_{\vec{\ell} \in V_{jk}}
				\abs{\prt_\eta^{\ell_1} \bdelta_\beta(\eta)}
				\cdots
				\abs{\prt_\eta^{\ell_j} \bdelta_\beta(\eta)}
			\, d \eta.
\end{align*}
If $m \le n - 2$ and $\al \in (0, 1]$ then, observing that the highest possible value of $l_i$ for $\ell \in V_{jk}$ is $k - j + 1$, and applying \cref{L:prtDeltaSobolevBound},
\begin{align*}
	&\norm{\prt_\eta^k (G^\al(\bdelta_\beta(\eta))
                    \prt_\eta^{m - k + 1} \bdelta_\beta(\eta))}_{L^2_\eta} \\
        &\qquad
\le			C \sum_{j = 1}^k \norm{S_k(\bgamma)}_{L^\iny(\Tn^2)}
			\int_\Tn
				\frac{1}{\abs{\beta}^{j + \al}}
            \norm{\prt_\eta^{m - k + 1} \bdelta_\beta(\eta))}_{L^2_\eta}
            \norm{\bdelta}_{H^{k - j + 1}}
            \norm{\bdelta}_{H^{k - j + 1, \iny}}^{j - 1}
			\, d \beta \\
		&\qquad
			\le C \norm{S_m(\bgamma)}_{L^\iny(\Tn^2)}
			\brac{
			\norm{\bgamma}_{H^{m + 1}}
			+
				\sum_{j = 1}^k
			\norm{\bgamma}_{H^{m - k + 2}}
			\norm{\bgamma}_{H^{k - j + 2}}
			\norm{\bgamma}_{H^{k - j + 3}}^{j - 1}
			\int_\Tn
				\frac{\abs{\beta}^{j + 1}}{\abs{\beta}^{j + \al}}
			}
			\, d \beta \\
		&\qquad
			\le C \norm{S_m(\bgamma)}_{L^\iny(\Tn^2)}
			\brac{
			\norm{\bgamma}_{H^{m + 1}}
			+
			\norm{\bgamma}_{H^{m - k + 2}}
			\norm{\bgamma}_{H^{k + 1}}
			\norm{\bgamma}_{H^{k + 2}}^{k - 1}
			} \\
		&\qquad
			\le \norm{S_m(\bgamma)}_{L^\iny(\Tn^2)}
			\brac{\norm{\bgamma}_{H^{m + 1}}
			+
			\norm{\bgamma}_{H^{k + 2}}^{k + 1}}.
\end{align*}
Hence,
\begin{align*}
	\norm{\prt_\eta^m L(\bgamma)}_{L^2(\Tn)}
		&\le \norm{S_m(\bgamma)}_{L^\iny(\Tn^2)}
			\sum_{k = 0}^m \binom{m}{k}
			\brac{
				\norm{\bgamma}_{H^{m + 1}}
				+ 
				\norm{\bgamma}_{H^{k + 2}}^{k + 1}
			} \\
		&\le C \norm{S_m(\bgamma)}_{L^\iny(\Tn^2)}
			\sum_{j = 0}^{m + 1} \norm{\bgamma}_{H^{m + 2}(\Tn)}^j.
\end{align*}

In applying \cref{L:prtDeltaSobolevBound}, we used the bound $\norm{\bdelta_\beta}_{H^{k - j + 1, \iny}} \le \norm{\bgamma}_{H^{k - j + 3}} \abs{\beta} \le \norm{\bgamma}_{H^{k + 2}} \abs{\beta}$ to obtain the factor $\abs{\beta}^{j + 1}$ to fully cancel the singularity in the Green's function when $\al = 1$. For $\al < 1$, that is not needed, so for $m \le n - 1$ and $\al \in (0, 1)$ we can instead apply the \cref{L:prtDeltaSobolevBound} inequality,
$
	\norm{\bdelta_\beta}_{H^{k - j + 1, \iny}}
		\le \norm{\bgamma}_{H^{k + 1}}
$,
giving
\begin{align*}
	\norm{\prt_\eta^m L(\bgamma)}_{L^2(\Tn)}
		&\le C \norm{S_m(\bgamma)}_{L^\iny(\Tn^2)}
			\sum_{j = 0}^{m + 1} \norm{\bgamma}_{H^{m + 1}(\Tn)}^j.
		\qedhere
\end{align*}
\end{proof}

\begin{prop}\label{P:BoundsOnLPaired}
    Let $\al \in (0, 1]$ and let $\bgamma$ be an $H^n(\T)$ quasi-closed chain in $\R^2$ for $n \ge 4$. Then
\begin{align*}
	\abs{(\prt_\eta \bgamma, \prt_\eta L(\bgamma))_{L^2(\Tn)}}
		&\le
		\begin{cases}
			C \norm{S_1(\bgamma)}_{L^\iny(\Tn^2)}
					\norm{\bgamma}_{H^3} \norm{\bgamma}_{L^2}^2
				&\text{if } \al \in (0, 1), \\
			C \norm{S_1(\bgamma)}_{L^\iny(\Tn^2)}
					\norm{\bgamma}_{H^2} \norm{\bgamma}_{H^1}^2
				&\text{if } \al = 1 \\
		\end{cases}
\end{align*}
and, if $\al \in (0, 1)$, for all $3 \le m \le n - 1$,
\begin{align*}
	\abs{(\prt_\eta^m \bgamma, \prt_\eta^m L(\bgamma))_{L^2(\Tn)}}
		&\le \norm{S_m(\bgamma)}_{L^\iny(\Tn^2)}
			\sum_{j = 3}^{m + 2} \norm{\bgamma}_{H^m(\Tn)}^j.
\end{align*}
\end{prop}
\begin{proof}
To bound $\abs{(\prt_\eta \bgamma, \prt_\eta L(\bgamma))_{L^2(\Tn)}}$, from \cref{P:LSymmetry}, we have
\begin{align*}
	(\prt_\eta \bgamma, &\prt_\eta L(\bgamma))_{L^2(\Tn)}
		 = \int_\Tn
                G^\al_p(\bdelta_\beta(\eta))
    			\prt_\eta \bdelta_\beta(\eta)
                \cdot \bdelta_\beta(\eta)
    			\, d \eta \\
		 &= \frac{1}{2} \int_\Tn
                G^\al_p(\bdelta_\beta(\eta))
    			\prt_\eta \bdelta_\beta(\eta)
                \cdot \bdelta_\beta(\eta)
    			\, d \eta
         = -\frac{1}{2} \int_\Tn
                \prt_\eta G^\al_p(\bdelta_\beta(\eta))
                    \abs{\bdelta_\beta(\eta)}^2
    			\, d \eta,
\end{align*}
where we integrated by parts using \cref{C:IBP}. This is as in \cite{Gancedo2008}, where the integrand is a perfect derivative, making the integral vanish. As we observed in \cref{S:AdaptingGancedo}, the integrand is not a perfect derivative in the present horizontally periodic setting, 
so we are forced to obtain a weaker bound.

From \cref{e:absGpBound} for $k = 1$, we have
    \begin{align*}
        \abs{(\bgamma, L(\bgamma))_{L^2{}}}
            &\le \frac{1}{2} \int_\Tn \int_\Tn
                \abs{\prt_\eta G^\al_p(\bdelta_\beta(\eta))}
                \abs{\bdelta_\beta(\eta)}^2
                \, d \eta \, d \beta \\
            &\le
            	C \norm{S_1(\bgamma)}_{L^\iny(\Tn^2)}
					\int_\Tn \frac{1}{\abs{\beta}^{1 + \al}}
						\norm{\prt_\eta \bdelta(\eta)}_{L^\iny}
						\norm{\bdelta_\beta(\eta)}_{L^2}^2
						\, d \beta.
    \end{align*}
    
We now apply \cref{L:prtDeltaSobolevBound}. 
    If $\al \in (0, 1)$ then
    \begin{align*}
        \abs{(\bgamma, L(\bgamma))_{L^2{}}}
            &\le C \norm{S_1(\bgamma)}_{L^\iny(\Tn^2)}
					\norm{\bgamma}_{H^3} \norm{\bgamma}_{L^2}^2
					\int_\Tn \frac{\abs{\beta}}{\abs{\beta}^{1 + \al}}
						\, d \beta
				\\
            &= C \norm{S_1(\bgamma)}_{L^\iny(\Tn^2)}
					\norm{\bgamma}_{H^3} \norm{\bgamma}_{L^2}^{2}.
    \end{align*}
If $\al = 1,$ this integral does not converge, so we apply \cref{L:prtDeltaSobolevBound} differently, finding instead
    \begin{align*}
        \abs{(\bgamma, L(\bgamma))_{L^2{}}}
            &\le C \norm{S_1(\bgamma)}_{L^\iny(\Tn^2)}
					\norm{\bgamma}_{H^2} \norm{\bgamma}_{H^1}^2
					\int_\Tn \frac{\abs{\beta}^2}{\abs{\beta}^{1 + \al}}
						\, d \beta
				\\
            &= C \norm{S_1(\bgamma)}_{L^\iny(\Tn^2)}
					\norm{\bgamma}_{H^2} \norm{\bgamma}_{H^1}^2.
    \end{align*}

Next we establish the bound on $\abs{(\prt_\eta^m \bgamma, \prt_\eta^m L(\bgamma))_{L^2(\Tn)}}$. From \cref{P:LSymmetry}
\begin{align}\label{e:gammaLgammaSum}
	\begin{split}
        (\prt_\eta^m \bgamma, &\prt_\eta^m L(\bgamma))_{L^2(\Tn)}
            = \frac{1}{2} \int_\Tn \int_\Tn
    				\prt_\eta^m \pr{G^\al_p(\bdelta_\beta(\eta))
    					\prt_\eta \bdelta_\beta(\eta)}
                        \cdot \prt_\eta^m \bdelta_\beta(\eta)
    					\, d \beta \, d \eta \\
		&= \frac{1}{2} \sum_{k = 0}^m \binom{m}{k}  \int_\Tn \int_\Tn
                \prt_\eta^m \bdelta_\beta(\eta)
                	\cdot \prt_\eta^k (G^\al_p(\bdelta_\beta(\eta))
                    \prt_\eta^{m - k + 1} \bdelta_\beta(\eta))
                    \, d \eta \, d \beta.
	\end{split}
\end{align}
We wish to obtain a bound using only the $H^m$-norm of $\bgamma$, but in the $k = 0$ term, $\prt_\eta^{m - k + 1} \bdelta_\beta(\eta) = \prt_\eta^{m + 1} \bdelta_\beta(\eta)$. That term we integrate by parts using \cref{C:IBP}:
\begin{align*}
	\frac{1}{2} \int_\Tn &\int_\Tn
			\prt_\eta^m \bdelta_\beta(\eta)
    		\cdot G^\al_p(\bdelta_\beta(\eta)
        	\prt_\eta^{m + 1} \bdelta_\beta(\eta))
         	\, d \eta \, d \beta \\
		&= \frac{1}{4} \int_\Tn \int_\Tn
			G^\al_p(\bdelta_\beta(\eta)
        	\prt_\eta \abs{\prt_\eta^m \bdelta_\beta(\eta))}^2
         	\, d \eta \, d \beta
		= -\frac{1}{4} \int_\Tn \int_\Tn
			\prt_\eta G^\al_p(\bdelta_\beta(\eta)
        	\abs{\prt_\eta^m \bdelta_\beta(\eta))}^2
         	\, d \eta \, d \beta \\
		&= -\frac{1}{4} \int_\Tn \int_\Tn
			\prt_\eta^m \bdelta_\beta(\eta) \cdot
			\prt_\eta G^\al_p(\bdelta_\beta(\eta)
        	\prt_\eta^m \bdelta_\beta(\eta)
         	\, d \eta \, d \beta,
\end{align*}
which is a constant multiple of the $k = 1$ term in \cref{e:gammaLgammaSum}. Hence, using  \cref{e:absGpBound}, we have
\begin{align*} 
	\begin{split}
    \abs{(\prt_\eta^m \bgamma, \prt_\eta^m L(\bgamma))_{L^2(\Tn)}}
		&\le C \sum_{k = 1}^m \int_\Tn \int_\Tn
                \abs{\prt_\eta^m \bdelta_\beta(\eta)}
                	\abs{\prt_\eta^k (G^\al_p(\bdelta_\beta(\eta))}
                    \abs{\prt_\eta^{m - k + 1} \bdelta_\beta(\eta))}
                    \, d \eta \, d \beta \\
        &
		\le C \norm{S_m(\bgamma)}_{L^\iny(\Tn^2)}
			\sum_{k = 1}^m \sum_{j = 1}^k \sum_{\vec{\ell} \in V_{jk}}
				P_{kj\vec{\ell}},
	\end{split}
\end{align*}
where, for $\ell \in V_{jk}$,
\begin{align*}
       P_{kj\vec{\ell}}
        	&:= \int_\Tn \int_\Tn
                    \frac{1}{\abs{\beta}^{j + \al}}
                	\abs{\prt_\eta^m \bdelta_\beta(\eta)}
                    \abs{\prt_\eta^{m - k + 1} \bdelta_\beta(\eta))}
					\abs{\prt_\eta^{\ell_1} \bdelta_\beta(\eta)}
					\cdots
					\abs{\prt_\eta^{\ell_j} \bdelta_\beta(\eta)}
                    \, d \beta \, d \eta.
\end{align*}

As long as $i < m - 1$, we can use \cref{L:prtDeltaSobolevBound} to bound the factors $\abs{\prt_\eta^i \bdelta_\beta(\eta)}$ appearing in $P_{kj\vec{\ell}}$ either in the $L^\iny$ or $L^2$ norms; for $i = m - 1$, we can only bound $\abs{\prt_\eta^i \bdelta_\beta(\eta)}$ in the $L^2$ norm; and for $i = m$ we must bound $\abs{\prt_\eta^i \bdelta_\beta(\eta)}$ as is.  Since there is always at least one $\abs{\prt_\eta^m \bdelta_\beta(\eta)}$ factor, there are four cases to consider:

\medskip\noindent\textbf{Case 1}: All but the one factor $\abs{\prt_\eta^m \bdelta_\beta(\eta)}$ in $P_{kj\vec{\ell}}$ are of the form $\abs{\prt_\eta^i \bdelta_\beta(\eta)}$ for $i < m - 1$. Then
\begin{align*}
	P_{kj\vec{\ell}}
        	&\le \int_\Tn
                    \frac{1}{\abs{\beta}^{j + \al}}
					\norm{\prt_\eta^{\ell_1} \bdelta_\beta(\eta)}_{L^\iny}
					\cdots
					\norm{\prt_\eta^{\ell_j} \bdelta_\beta(\eta)}_{L^\iny}
					 \int_\Tn
	                \abs{\prt_\eta^m \bdelta_\beta(\eta)}
                    \abs{\prt_\eta^{m - k + 1} \bdelta_\beta(\eta))}
                    \, d \eta \, d \beta \\
        	&\le C \norm{\bgamma}_{H^m}^j
				\int_\Tn
                    \frac{1}{\abs{\beta}^{j + \al}}
					\abs{\beta}^j
	                \norm{\bgamma}_{H^m}
                    \norm{\bgamma}_{H^m}\abs{\beta}
                    \, d \beta
           = C \norm{\bgamma}_{H^m}^{j + 2}.
\end{align*}
We note that this bound would hold for $\al = 1$ as well.

\medskip\noindent\textbf{Case 2}: $P_{kj\vec{\ell}}$ contains two factors of $\prt_\eta^m \bdelta_\beta(\eta)$. There are two terms of this form: (1) $k = j = 1$, which also has one factor of $\prt_\eta \bdelta_\beta(\eta)$ and (2) $k = m$, $j = 1$, for which the inner sum has only one term, which contains $m$ factors of $\prt_\eta \bdelta_\beta(\eta)$. Using $\abs{\prt_\eta \bdelta_\beta(\eta)}_{L^\iny} \le C \norm{\bgamma}_{H^2} \abs{\beta}$ by \cref{L:prtDeltaSobolevBound}, we bound $P_{kj\vec{\ell}}$ for (1), (2), respectively, by
\begin{align*}
	P_{11\vec{\ell}}
				&\le C \norm{\bgamma}_{H^2}
			\int_\Tn
			\frac{1}{\abs{\beta}^{1 + \al}}
			\abs{\beta}
			\int_\Tn
			 \abs{\prt_\eta^m \bdelta_\beta(\eta)}^2
			 \, d \eta
			 \, d \beta
			 \le C \norm{\bgamma}_{H^2} \norm{\bgamma}_{H^m}^2, \\
	P_{mm\vec{\ell}}
				&\le C \norm{\bgamma}_{H^2}^m
			\int_\Tn
			\frac{1}{\abs{\beta}^{m + \al}}
			\abs{\beta}^m
			\int_\Tn
			 \abs{\prt_\eta^m \bdelta_\beta(\eta)}^2
			 \, d \eta
			 \, d \beta
			 \le C \norm{\bgamma}_{H^2}^m \norm{\bgamma}_{H^m}^2.
\end{align*}

\medskip\noindent\textbf{Case 3}: $P_{kj\vec{\ell}}$ contains one factor of $\prt_\eta^m \bdelta_\beta(\eta)$ and two factors of $\prt_\eta^{m - 1} \bdelta_\beta(\eta)$. Those two factors cannot both come from $\abs{\prt_\eta^{\ell_1} \bdelta_\beta(\eta)} \cdots \abs{\prt_\eta^{\ell_j} \bdelta_\beta(\eta)}$, since $\ell_1 + \cdots + \ell_j \le m$, so it must be that $m - k + 1 = m - 1$, which requires that $k = 2$. But then $\abs{\prt_\eta^{\ell_1} \bdelta_\beta(\eta)} \cdots \abs{\prt_\eta^{\ell_j} \bdelta_\beta(\eta)}$ can contain no factor with a higher derivative than $\abs{\prt_\eta^2 \bdelta_\beta(\eta)}$, and that only when $j = 1$. So this case can only occur when $m = 3$, so that $2 = m - 1$. In this case,
\begin{align*}
    P_{21\vec{\ell}}
        &\le
            \int_\Tn \int_\Tn
            \frac{1}{\abs{\beta}^{1 + \al}}
            \abs{\prt_\eta^3 \bdelta_\beta^j(\eta)}
            \abs{\prt_\eta^2 \bdelta_\beta(\eta)}
            \abs{\prt_\eta^2 \bdelta_\beta(\eta)}
            \, d \eta \, d \beta \\
        &\le C
            \norm{\bgamma}_{H^3}
            \int_\Tn
            \frac{1}{\abs{\beta}^{1 + \al}}
            \int_\Tn
            \abs{\prt_\eta^2 \bdelta_\beta^j(\eta)}
            \abs{\prt_\eta^3 \bdelta_\beta(\eta)}
            \, d \eta \, d \beta \\
        &\le C 
            \norm{\bgamma}_{H^3}
            \int_\Tn
            \frac{1}{\abs{\beta}^{1 + \al}}
            \norm{\prt_\eta^2 \bdelta_\beta^j(\eta)}_{L^2_\eta}
            \norm{\prt_\eta^3 \bdelta_\beta(\eta)}_{L^2_\eta}
            \, d \beta \\
        &\le C 
            \norm{\bgamma}_{H^3}^3
            \int_\Tn
            \frac{1}{\abs{\beta}^{1 + \al}}
            \abs{\beta}
            \, d \beta
        = C 
            \norm{\bgamma}_{H^3}^3
        \le C 
            \norm{\bgamma}_{H^m}^3.
\end{align*}

\medskip\noindent\textbf{Case 4}: $P_{kj\vec{\ell}}$ contains one factor of $\prt_\eta^m \bdelta_\beta(\eta)$, one factor of $\prt_\eta^{m - 1} \bdelta_\beta(\eta)$, and all other factors are of the form $\abs{\prt_\eta^i \bdelta_\beta(\eta)}$ for $i < m - 1$. This occurs only when (1) $m > 3$ and $k = 2$ or (2) $j = 1, k = m - 1$. For (1), the bound is the same as in Case 1. For (2), it must be that $\ell_1 = m - 1$, so
\begin{align*}
	P_{m - 1, 1, \vec{\ell}}
        	&\le \int_\Tn
                    \frac{1}{\abs{\beta}^{1 + \al}}
					\norm{\prt_\eta^2 \bdelta_\beta(\eta)}_{L^\iny}
					 \int_\Tn
	                \abs{\prt_\eta^m \bdelta_\beta(\eta)}
                    \abs{\prt_\eta^{m - 1} \bdelta_\beta(\eta))}
                    \, d \eta \, d \beta \\
        	&\le C \norm{\bgamma}_{H^m}^3
				\int_\Tn
                    \frac{1}{\abs{\beta}^{1 + \al}}
					\abs{\beta}^2
                    \, d \beta
           = C \norm{\bgamma}_{H^m}^3,
\end{align*}
a bound that would hold for $\al = 1$ as well.

Combined, these four cases give the bound on $\abs{(\prt_\eta^m \bgamma, \prt_\eta^m L(\bgamma))_{L^2(\Tn)}}$.
\end{proof}

\begin{prop}\label{P:ContinuityOfL}
    Let $\al \in (0, 1]$ and let $\bgamma_1$ and $\bgamma_2$ be $H^n(\T)$ be quasi-closed chains in $\R^2$ for $n \ge 3$. Then for any $r \in (0, 1)$,
\begin{align*}
        \abs{L(\bgamma)(\eta_1) - L(\bgamma)(\eta_2)}
            &\le
                \begin{cases}  
                    A_2(\bgamma, \al) \abs{\eta_1 - \eta_2}
                        &\text{if } \al \in (0, 1), \\
                    A_1(\bgamma, 1) \abs{\eta_1 - \eta_2}^{\frac{1}{2}}
                        &\text{if } \al = 1,
                \end{cases}
\end{align*}
where $A_1$ and $A_2$ are as in \cref{P:BoundsOnL}. If $\al \in (0, 1)$ then
\begin{align*}
        \norm{L(\bgamma_1) - L(\bgamma_2)}_{L^2(\Tn)}
            &\le
                   B_1(\bgamma_1, \bgamma_2) \norm{\bgamma_1 - \bgamma_2}_{H^1},
\end{align*}
where
\begin{align*}
	B_1(\bgamma_1, \bgamma_2)
		&:= C \norm{S_0(\bgamma_2)}_{L^\iny(\Tn^2)}
			+ C \sum_{j = 1}^2 \norm{S_1(\bgamma_j)}_{L^\iny(\Tn^2)}
                \norm{\bgamma_1}_{H^2}
            + C_0 \norm{\bgamma_1}_{H^3}.
\end{align*}
Define $\mu \colon [0, \iny) \to [0, \iny)$ by $\mu(0) = 0$ and
\begin{align}\label{e:mu}
	\mu(x)
		&=
		\begin{cases}
			-e x \log x &\text{if } x < e^{-1}, \\
			x/e &\text{if } x \ge e^{-1}.
		\end{cases}
\end{align}
If $\al = 1$ then
\begin{align*}
        \norm{L(\bgamma_1) - L(\bgamma_2)}_{L^2(\Tn)}
            &\le
                  B_2(\bgamma_1, \bgamma_2) \mu(\norm{\bgamma_1 - \bgamma_2}_{H^1}),
\end{align*}
where 
\begin{align*}
	B_2(\bgamma_1, \bgamma_2)
		&:=  C \max_{1, 2} \norm{\bgamma_j}_{H^3}^{1 - \frac{1}{e}}
            	\norm{S_0(\bgamma_2)}_{L^\iny(\Tn^2)}
			+ C \sum_{j = 1}^2 \norm{S_1(\bgamma_j)}_{L^\iny(\Tn^2)}
                \norm{\bgamma_1}_{H^2} \\
        &\qquad
            + C_0 \norm{\bgamma_1}_{H^3} \max_{1, 2} \norm{\bgamma_j}_{H^2}^{1 - \frac{1}{e}}.
\end{align*}
 
\end{prop}

\begin{proof}
	The bound on $\abs{L(\bgamma)(\eta_1) - L(\bgamma)(\eta_2)}$ follows from \cref{L:prtDeltaSobolevBound} for $\al \in (0, 1)$ and from Sobolev embedding for $\al = 1$.
	
    Defining
    \begin{align*}
        \bdelta_\beta^j(\eta)
            &:= \bgamma_j(\eta) - \bgamma_j(\eta - \beta),
    \end{align*}
    we make the decomposition, $L(\bgamma_1) - L(\bgamma_2) = P_1 + P_2 + P_3$, where
    \begin{align*}
        P_1
            &= \int_\Tn
                G^\al_p(\bdelta_\beta^2(\eta))
                    \prt_\eta \brac{\bdelta_\beta^1(\eta)
                        - \bdelta_\beta^2(\eta)}
                    \, d \beta, \\
        P_2
            &= \int_\Tn
                \brac{R^\al(\bdelta_\beta^1(\eta))
                    - R^\al_p(\bdelta_\beta^2(\eta))}
                    \prt_\eta \bdelta_\beta^1(\eta)
                    \, d \beta, \\
        P_3
            &= \int_\Tn
                \brac{G^\al(\bdelta_\beta^1(\eta))
                    - G^\al(\bdelta_\beta^2(\eta))}
                    \prt_\eta \bdelta_\beta^1(\eta)
                    \, d \beta.
    \end{align*}
    Now,
    \begin{align}\label{e:deltabetadiffTwoSols}
    	\begin{split}
        \norm{\prt_\eta \bdelta_\beta^1(\eta)}_{L^2}
            &\le 2 \norm{\prt_\eta \bgamma_1}_{L^2}
            \le 2 \norm{\bgamma_1}_{H^1}, \\
        \norm{\bdelta_\beta^1(\eta) - \bdelta_\beta^2(\eta)}_{L^2}
            &\le \norm{\bgamma_1(\eta) - \bgamma_2(\eta)}_{L^2}
                + \norm{\bgamma_1(\eta - \beta)
                    - \bgamma_2(\eta -\beta)}_{L^2} \\
            &\le 2 \norm{\bgamma_1 - \bgamma_2}_{L^2}, \\
        \norm{\prt_\eta[\bdelta_\beta^1(\eta) - \bdelta_\beta^2(\eta)]}_{L^2}
            &\le \norm{\prt_\eta[\bgamma_1(\eta) - \bgamma_2(\eta)]}_{L^2}
                + \norm{\prt_\eta[\bgamma_1(\eta - \beta)
                    - \bgamma_2(\eta -\beta)]}_{L^2} \\
            &\le 2 \norm{\bgamma_1 - \bgamma_2}_{H^1}.
          \end{split}
    \end{align}
    For $\al = 1$, we will also use
    \begin{align*}
        &\abs{\bdelta_\beta^1(\eta) - \bdelta_\beta^2(\eta)}
            \le  \norm{\bdelta_\beta^1(\eta)}_{L^\iny}
                + \norm{\bdelta_\beta^2(\eta)}_{L^\iny}
            \le C \max_{1, 2} \norm{\bgamma_j}_{H^2} \abs{\beta}, \\
        &\abs{\prt_\eta[\bdelta_\beta^1(\eta) - \bdelta_\beta^2(\eta)]}
            \le  \norm{\prt_\eta \bdelta_\beta^1(\eta)}_{L^\iny}
                + \norm{\prt_\eta \bdelta_\beta^2(\eta)}_{L^\iny}
            \le C \max_{1, 2} \norm{\bgamma_j}_{H^3} \abs{\beta},
    \end{align*}
    by \cref{L:prtDeltaSobolevBound}. Thus, by interpolation, for any $r \in (0, 1)$,
    \begin{align}\label{e:InterpBound2}
    	\begin{split}
        \norm{\bdelta_\beta^1(\eta) - \bdelta_\beta^2(\eta)}_{L^2}
            &\le C \max_{1, 2} \norm{\bgamma_j}_{H^2}^{1 - r}
            	\norm{\bgamma_1 - \bgamma_2}_{L^2}^r \abs{\beta}^{1 - r}, \\
        \norm{\prt_\eta[\bdelta_\beta^1(\eta) - \bdelta_\beta^2(\eta)]}_{L^2}
            &\le C \max_{1, 2} \norm{\bgamma_j}_{H^3}^{1 - r}
            	\norm{\bgamma_1 - \bgamma_2}_{H^1}^r \abs{\beta}^{1 - r}.
		\end{split}
    \end{align}
    
    Using \cref{e:gradGalpEstimate} for $n = 0$, and applying Minkowski's integral inequality, for $\al \in (0, 1)$,
    \begin{align*}
        \norm{P_1}_{L^2}
            &\le 
            \int_\Tn
                \norm{G^\al_p(\bdelta_\beta^2(\eta))
                    \prt_\eta \brac{\bdelta_\beta^1(\eta)
                        - \bdelta_\beta^2(\eta)}}_{L^2_\eta}
                    \, d \beta \\
            &\le 
            \int_\Tn
                \norm{G^\al_p(\bdelta_\beta^2(\eta))}_{L^\iny_\eta}
                    \norm{\prt_\eta \brac{\bdelta_\beta^1(\eta)
                        - \bdelta_\beta^2(\eta)}}_{L^2_\eta}
                    \, d \beta \\
            &\le 2 \norm{F(\bgamma_2)}_{L^\iny}^\al
                \norm{\bgamma_1 - \bgamma_2}_{H^1} 
                \int_{\Tn}
                    \frac{c_\al}{\abs{\beta}^\al} \, d \beta
            \le C \norm{S_0(\bgamma_2)}_{L^\iny(\Tn^2)}
                \norm{\bgamma_1 - \bgamma_2}_{H^1}.
    \end{align*}
    For $\al = 1$, we use instead \cref{e:InterpBound2} to give
    \begin{align*}
        \norm{P_1}_{L^2}
            &\le C \max_{1, 2} \norm{\bgamma_j}_{H^3}^{1 - r}
            	\norm{S_0(\bgamma_2)}_{L^\iny(\Tn^2)}
                \norm{\bgamma_1 - \bgamma_2}_{H^1}^r
                \int_{\Tn}
                    \frac{\abs{\beta}^{1 - r}}{\abs{\beta}} \, d \beta \\
            &\le C \max_{1, 2} \norm{\bgamma_j}_{H^3}^{1 - r}
            	\norm{S_0(\bgamma_2)}_{L^\iny(\Tn^2)}
                \frac{\norm{\bgamma_1 - \bgamma_2}_{H^1}^r}{1 - r}.
    \end{align*}

    For $P_2$, we have for any $\al \in (0, 1]$,
    \begin{align*}
        \begin{split}
        &\abs{R^\al(\bdelta_\beta^1(\eta))
                    - R^\al(\bdelta_\beta^2(\eta))}
            = \frac{\abs{R^\al(\bdelta_\beta^1(\eta))
                    - R^\al(\bdelta_\beta^2(\eta))}}
                    {\abs{\bdelta_\beta^1(\eta)
                        - \bdelta_\beta^2(\eta)}}
                \abs{\bdelta_\beta^1(\eta)
                        - \bdelta_\beta^2(\eta)} \\
            &\qquad
            \le \norm{\grad R^\al}_{L^\iny
                    (supp(\delta_\beta^1(\eta)) \cup supp(\delta_\beta^2(\eta))}
                \abs{\bdelta_\beta^1(\eta)
                        - \bdelta_\beta^2(\eta)} \\
            &\qquad
            \le C [\norm{S_1(\bgamma_1)}_{L^\iny(\Tn^2)}
                    + \norm{S_1(\bgamma_2)}_{L^\iny(\Tn^2)}]
                \abs{\bdelta_\beta^1(\eta)
                        - \bdelta_\beta^2(\eta)}.
        \end{split}
    \end{align*}
    Hence,
    \begin{align*}
        \norm{P_2}_{L^2}
            &\le C \sum_{j = 1}^2 \norm{S_1(\bgamma_j)}_{L^\iny(\Tn^2)}
                \norm{\prt_\eta \bdelta_\beta^1(\eta)}_{L^\iny}
                \norm{\bdelta_\beta^1 - \bdelta_\beta^2}_{L^2} \\
            &\le C \sum_{j = 1}^2 \norm{S_1(\bgamma_j)}_{L^\iny(\Tn^2)}
                \norm{\bgamma_1}_{H^2}
                \norm{\bgamma_1 - \bgamma_2}_{L^2}.
    \end{align*}
    
    This leaves $P_3$.
    Applying the inequality (see (27) of \cite{Gancedo2008})
    \begin{align*}
        \abs{a^\al - b^\al}
            &\le \al \min\set{a, b}^{\al - 1} \abs{a - b},
            \qquad \forall a>0, b>0,
    \end{align*}
    we have
    \begin{align*}
        &\abs{G^\al(\bdelta_\beta^1(\eta))
                    - G^\al(\bdelta_\beta^2(\eta))}
            = \abs[\bigg]
                {\frac{c_\al}{\abs{\bdelta_\beta^1(\eta)}^\al}
                    - \frac{c_\al}{\abs{\bdelta_\beta^2(\eta)}^\al}
                }
            = c_\al \abs[\bigg]
                {\frac{\abs{\bdelta_\beta^2(\eta)}^\al
                    - \abs{\bdelta_\beta^1(\eta)}^\al}
                {\abs{\bdelta_\beta^1(\eta)}^\al
                \abs{\bdelta_\beta^2(\eta)}^\al}}
            \\
            &\qquad
            = c_\al\abs[\bigg]
                {\pr{\frac{\abs{\bdelta_\beta^2(\eta)}}{\abs{\beta}}}^\al
                    - \pr{\frac{\abs{\bdelta_\beta^1(\eta)}}{\abs{\beta}}}^\al}
                \pr{\frac{\abs{\beta}}
                {\abs{\bdelta_\beta^1(\eta)}}}^\al
                \pr{\frac{\abs{\beta}}
                {\abs{\bdelta_\beta^2(\eta)}}}^\al
                \abs{\beta}^{-\al}
            \\
            &\qquad
            \le
            \min \set{{\frac{\abs{\bdelta_\beta^2(\eta)}}{\abs{\beta}},
                    \frac{\abs{\bdelta_\beta^1(\eta)}}{\abs{\beta}}}}
                    ^{\al - 1}
                \abs[\bigg]
                    {
                       \abs[\bigg]{\frac{\bdelta_\beta^2(\eta)}
                            {\beta}}
                    - 
                        \abs[\bigg]{\frac{\bdelta_\beta^1(\eta)}
                            {\beta}}
                    }
                    \abs{\beta}^{-\al}
                    \norm{F}_{L^\iny(\Tn^2)}^{2 \al}
            \\
            &\qquad
            \le
            \min_{1, 2}  \pr{\frac{C \norm{\bgamma_j}_{H^2}}{\abs{\beta}}
                \abs{\beta}}^{\al - 1}
                \abs[\bigg]
                    {
                       \abs[\bigg]{\frac{\bdelta_\beta^2(\eta)}
                            {\beta}}
                    - 
                        \abs[\bigg]{\frac{\bdelta_\beta^1(\eta)}
                            {\beta}}
                    }
                    \abs{\beta}^{-\al}
                    \norm{F(\bgamma_1)}_{L^\iny(\Tn^2)}^\al
                    \norm{F(\bgamma_2)}_{L^\iny(\Tn^2)}^\al\\
            &\qquad
            \le
            C \min_{1, 2} \norm{\bgamma_j}_{H^2}^{\al - 1}
                    \abs{\bdelta_\beta^2(\eta) - \bdelta_\beta^1(\eta)}
                    \abs{\beta}^{-1 -\al}
                    \norm{F(\bgamma_1)}_{L^\iny(\Tn^2)}^\al
                    \norm{F(\bgamma_2)}_{L^\iny(\Tn^2)}^\al \\
            &\qquad
            = C_0 \abs{\bdelta_\beta^2(\eta) - \bdelta_\beta^1(\eta)}
                    \abs{\beta}^{-1 -\al},
    \end{align*}
    where we used the reverse triangle inequality.
    
    For $\al \in (0, 1)$, we then have
    \begin{align*}
        \norm{P_3}_{L^2}
            &\le C_0
                \int_\Tn \frac{1}{\abs{\beta}^{1 + \al}}
                    \norm{\prt_\eta \bdelta_\beta^1}_{L^\iny}
                    \norm{\bdelta_\beta^2(\eta) - \bdelta_\beta^1(\eta)}_{L^2_\eta}
                    \, d \beta \\
            &\le C_0 \norm{\bgamma_1}_{H^3}
            		\norm{\bgamma_1- \bgamma_2}_{L^2}
                    \int_\Tn \frac{\abs{\beta}}{\abs{\beta}^{1 + \al}}
                    \, d \beta
            \le C_0 \norm{\bgamma_1}_{H^3} \norm{\bgamma_1- \bgamma_2}_{L^2}.
    \end{align*}
    
    For $\al = 1$, we apply \cref{e:InterpBound2}, so that, for any $r \in (0, 1)$,
    \begin{align*}
        \norm{P_3}_{L^2}
            &\le C_0
                \int_\Tn \frac{1}{\abs{\beta}^{1 + 1}}
                    \norm{\prt_\eta \bdelta_\beta^1}_{L^\iny}
                    \max_{1, 2} \norm{\bgamma_j}_{H^2}^{1 - r}
					\norm{\bgamma_1 - \bgamma_2}_{L^2}^r
					\abs{\beta}^{1 - r}
                    \, d \beta \\
            &\le C_0 \norm{\bgamma_1}_{H^3}
            		\max_{1, 2} \norm{\bgamma_j}_{H^2}^{1 - r}
            		\norm{\bgamma_1- \bgamma_2}_{L^2}^r
                    \int_\Tn \frac{\abs{\beta}^{1 + 1 - r}}{\abs{\beta}^2}
                    \, d \beta \\
            &\le C_0 \norm{\bgamma_1}_{H^3}
            		\max_{1, 2} \norm{\bgamma_j}_{H^2}^{1 - r}
            		\frac{\norm{\bgamma_1- \bgamma_2}_{L^2}^{r}}{1 - r}.
    \end{align*}
    
    Combining these bounds and applying \cref{L:MinGivesLL} below for $\al = 1$ gives the bounds on $\norm{L(\bgamma_1) - L(\bgamma_2)}_{L^2(\Tn)}$.
\end{proof}

\begin{remark}\label{R:LIsLL}
	We can parallel the argument in \cref{P:ContinuityOfL} that led to the bound on
	$\norm{L(\bgamma_1) - L(\bgamma_2)}_{L^2(\Tn)}$ to show that
	$\abs{L(\bgamma)(\eta_1) - L(\bgamma)(\eta_2)} \le C \mu(\abs{\eta_1 - \eta_2})$.
	Since $\mu$ is a modulus of continuity for log-Lipschitz functions, this means
	that $L(\gamma)$ is log-Lipschitz continuous along the boundary.
\end{remark}

We used the following lemma above.

\begin{lemma}\label{L:MinGivesLL}
	Let $a > 0$. Then
	\begin{align*}
		\min_{r \in [0, 1)}
			\frac{a^r}{1 - r}
				&=
				\begin{cases}
					-e a \log a &\text{if } a < e^{-1}, \\
					1 &\text{if } a \ge e^{-1}.
				\end{cases}
	\end{align*}
\end{lemma}
\begin{proof}
	Let $f(r) := \frac{a^r}{1 - r}$. To find the minimum of $f$, set
	\begin{align*}
		0
			&= f'(r)
			= \frac{(1 - r) (\log a) a^r + a^r}{1 - r},
	\end{align*}
	which occurs when $1 - r = -1/\log a$, so $r = r_0 := 1 + 1/\log a$. To have $r_0 < 1$, we must have $a < 1$, and to have $r_0 > 0$, we must have $1/\log a > -1$; that is, $a < 1/e$. This gives a minimum value,
	\begin{align*}
		f(r_0)
			&= \frac{a^{1 + \frac{1}{\log a}}}{-\frac{1}{\log a}}
			= - a a^{\frac{1}{\log a}} \log a
			= - a e^{\frac{\log a}{\log a}} \log a
			= - e a \log a.
	\end{align*}
	Otherwise, the minimum value of $f(r)$ occurs at $r = 1$ with a value of $1$.
\end{proof}

%
%
\section{Well-posedness for $\al \in (0, 1)$}\label{S:WellPosednessAlphalt1}

\noindent In this section we prove the existence and uniqueness of a periodic CDE solution (as defined in 
\cref{D:CDESolutionPeriodic}) for $\alpha\in(0,1).$

\begin{theorem}\label{T:ExistenceAlphaLt1}
    Let $\al \in (0, 1)$, and let $\Omega_0$ and $\bgamma_0 \in H^m(\Tn)$ be as in \cref{D:CDESolutionPeriodic}. 
    There exists a time $T > 0$ so that on $[0, T]$ there is a unique periodic CDE solution
    $\bgamma \in C([0, T]; H^{m}(\Tn))$ with 
    $\prt_t \bgamma \in C([0, T]; L^\iny(\Tn)) \cap L^\iny(0, T; H^{m - 1}(\Tn))$ for which $\bgamma(0)= \bgamma_0$.
\end{theorem}

We will adapt the proof of this same theorem for a single compactly supported  patch as given by Gancedo in \cite{Gancedo2008}. In outline, the steps in the proof are as follows:
\begin{itemize}
    \setlength\itemsep{0.5em}

    \item[Step 1]
        Assume that $\bgamma$ is an $H^{m + 1}(\Tn)$ chain in $\R^2$ and bound $(\bgamma, L(\bgamma))_{H^m(\Tn)}$.

    \item[Step 2]
        Given $\bgamma \in C^1([0,T]; H^{m + 1}(\Tn))$ solving $\prt_t \bgamma = L(\bgamma)$, obtain a bound on $\norm{F(\bgamma(t))}_{L^\iny(\Tn^2)}$ in terms of itself and $\norm{\bgamma(t)}_{H^m(\Tn)}$.

    \item[Step 3]
        Given $\bgamma \in C^1([0,T]; H^{m + 1}(\Tn))$ solving $\prt_t \bgamma = L(\bgamma)$, combine the bounds in Steps 1 and 2 to obtain a uniform bound on  $\norm{\bgamma(t)}_{H^m(\Tn)} + \norm{F(\bgamma(t))}_{L^\iny(\Tn^2)}$.

    \item[Step 4]
        Regularize the CDE \cref{e:CDEPeriodic} for a periodic solution $\bgamma^\eps$, and show that the bound in Step 3 applies to $\bgamma^\eps$. Apply Picard's theorem to conclude that a unique solution $\bgamma^\eps$, $\eps > 0$, exists up to the time $T^*_\eps$ at which the $L^\iny$ norm of $F(\bgamma^\eps)$ becomes infinite, thereby obtaining a uniform-in-$\eps$ bound, $T^*$, on $T_\eps^*$.
    
    \item[Step 5]
        Show that some subsequence $(\bgamma^{\eps_n})$ converges to a periodic CDE solution $\bgamma$ for some $T > 0$.

    \item[Step 6]
        Show that any periodic CDE solution is unique.
\end{itemize}

In the subsections that follow, we give each step of the proof.

\subsection{Step 1: Bounding $(\bgamma, L(\bgamma))_{H^m}$}

Because
\begin{align*}
    \pr{\norm{\bgamma}_{L^2(\Tn)}^2
        + \norm{\prt^k \bgamma}_{L^2(\Tn)}^2}^{\frac{1}{2}}
            \text{ and }
    \left({\sum_{j = 0}^k \norm{\prt^j \bgamma}_{L^2(\Tn)}^2}\right)
        ^{\frac{1}{2}}
\end{align*}
are equivalent $H^k(\Tn)$ norms, and similarly for the inner products, it follows from \cref{P:BoundsOnLPaired} that
\begin{align}\label{e:gammLgammaHm}
	\abs{(\bgamma, L(\bgamma))_{H^m}}
		&\le \norm{S_m(\bgamma)}_{L^\iny(\Tn^2)}
			\sum_{j = 3}^{m + 2} \norm{\bgamma}_{H^m(\Tn)}^j.
\end{align}

\subsection{Step 2: Bounding $\norm{F(\bgamma(t))}_{L^\iny(\Tn^2)}$}

\begin{prop}\label{P:FBoundOverTime}
    Let $\al \in (0, 1)$ and $\bgamma \in C^1([0,T]; H^{m + 1}(\Tn))$.
    For all $j \in \Z$,
    \begin{align}\label{e:FBound}
        \norm{F_j(\bgamma(t))}_{L^\iny(\Tn^2)}
            &\le \norm{F_j(\bgamma(0))}_{L^\iny(\Tn^2)}
                + \int_0^t
                    A_2(s)
                    \norm{F_j(\bgamma(s))}_{L^\iny(\Tn^2)}^2
                    \, ds,
        \end{align}
    where $A_2(s) = A_2(\bgamma(s), \al)$ is as in \cref{P:BoundsOnL}. The same bound holds for $F$ in place of $F_j$.
\end{prop}
\begin{proof}
    Let $f_j(\beta) = \abs{\beta}$ if $j = 0$ and $f_j(\beta) = 1$ if $j \ne 0$. Applying \cref{L:DiffPowerOfAbs},
    \begin{align*}
        \diff{}{t} \norm{F_j}_{L^p(\Tn^2)}^p
            &= \diff{}{t}
                \int_\Tn \int_\Tn
                    \pr{
                        \frac{f_j(\beta)}{\abs{\bdelta_\beta(\eta) - \n_j}}
                    }^p
                \, d \eta \, d \beta \\
            &= - p
                \int_\Tn \int_\Tn f_j(\beta)^p
                    \frac{(\bdelta_\beta(\eta) - \n_j)
                        \cdot \prt_t \bdelta_\beta(\eta)}
                        {\abs{\bdelta_\beta(\eta) - \n_j}^{p + 2}}
                \, d \eta \, d \beta \\
            &\le p
                \int_\Tn \int_\Tn f_j(\beta)^{-1}
                    \pr{\frac{f_j(\beta)}
                        {\abs{\bdelta_\beta(\eta) - \n_j}}}^{p + 1}
                    \abs{\prt_t \bdelta_\beta(\eta)}
                \, d \eta \, d \beta \\
            &= p
                \int_\Tn \int_\Tn f_j(\beta)^{-1}
                    F_j(\beta, \eta)
                    \abs{\prt_t \bdelta_\beta(\eta)}
                \, d \eta \, d \beta.
    \end{align*}
    But, by \cref{P:BoundsOnL},
    \begin{align*}
        \abs{\prt_t \bdelta_\beta(\eta)}
            &\le
                \sup_{\eta \in \Tn} \frac{\abs{\prt_t \bgamma(\eta) - \prt_t \bgamma(\eta - \beta)}}
                    {\abs{\beta}}
                \abs{\beta}
            \le A_2(\bgamma, \al)
                \abs{\beta}.
    \end{align*}
    Then, since $f_j(\beta)^{-1} \abs{\beta} \le 1$ for all $j$,
    \begin{align*}
        &p \norm{F_j}_{L^p(\Tn^2)}^{p - 1} \diff{}{t} \norm{F_j}_{L^p(\Tn^2)}
            = \diff{}{t} \norm{F_j}_{L^p(\Tn^2)}^p \\
            &\qquad
            \le C p
                \int_\Tn \int_\Tn
                    A_2
                    \abs{\beta}
                    f_j(\beta)^{-1}
                    F_j(\beta, \eta)^{p + 1}
                \, d \eta \, d \beta
            \le C p
                A_2
                \norm{F_j}_{L^p(\Tn^2)}^{p + 1} \\
            &\qquad
            \le C p
                A_2
                \norm{F_j}_{L^\iny(\Tn^2)}
                \norm{F_j}_{L^p(\Tn^2)}^p,
    \end{align*}
    so
    \begin{align*}
        \diff{}{t} \norm{F_j}_{L^p(\Tn^2)}
            &\le C A_2
                \norm{F_j}_{L^\iny(\Tn^2)}
                \norm{F_j}_{L^p(\Tn^2)}.
    \end{align*}

    Integrating in time gives
    \begin{align*}
        \norm{F_j(t)}_{L^p(\Tn^2)}
            &\le \norm{F_j(0)}_{L^p(\Tn^2)}
                +C  \int_0^t A_2(s)
                \norm{F_j(s)}_{L^\iny(\Tn^2)}
                \norm{F_j(s)}_{L^p(\Tn^2)}
                \, ds.
    \end{align*}
    Taking $p \to \iny$ gives \cref{e:FBound}, and taking the maximum over $j$ gives the same bound for $F$.
\end{proof}

\subsection{Step 3: Combining the bounds}\label{S:CombiningTheBounds}

We use $P_n$ to stand for a polynomial of degree $n$ with $P_n(0) = 0$ having nonnegative coefficients (so $P_n(x)$ is increasing for positive $x$).

Let $\bgamma \in C^1([0,T]; H^{m + 1}(\Tn))$. Then $\prt_t \bgamma = L(\bgamma)$, so it follows from \cref{e:gammLgammaHm} that
\begin{align*}
    \diff{}{t} \norm{\bgamma(t)}_{H^m(\Tn)}^2
        &\le C \norm{S_m(\bgamma)}_{L^\iny(\Tn^2)}
            P_{m + 2}(\norm{\bgamma}_{H^m(\Tn)}),
\end{align*}
and integrating in time, 
\begin{align}\label{e:HmBoundOngamma}
    \norm{\bgamma(t)}_{H^m(\Tn)}^2
        &\le \norm{\bgamma(0)}_{H^m(\Tn)}^2
            + C \int_0^t 
            \norm{S_m(\bgamma(s))}_{L^\iny(\Tn^2)}
            P_{m + 2}(\norm{\bgamma(s)}_{H^m(\Tn)})
            \, ds.
\end{align}

Letting
\begin{align*}
    S(t)
        := \norm{F(\bgamma(t))}_{L^\iny(\Tn)}
            + \norm{\bgamma(t)}_{H^m(\Tn)}^2,
\end{align*}
we add the bounds in \cref{e:HmBoundOngamma,e:FBound} to obtain,
\begin{align}\label{e:CombinedEnergyBound}
    \begin{split}
    S(t)
        \le S(0)
            + C \int_0^t
                \pbrac{
                    &\norm{S_m(\bgamma(s))}_{L^\iny(\Tn^2)}
                    P_{m + 2}(\norm{\bgamma(s)}_{H^m(\Tn)}) \\
            &
                    + 
                    A_2(s) \norm{F(\bgamma)(s)}_{L^\iny(\Tn^2)}^2
                    }
            \, ds.
    \end{split}
\end{align}
We see that
\begin{align*}
    \norm{S_m(\bgamma(s))}_{L^\iny(\Tn^2)} P_{m + 1}(\norm{\bgamma(s)}_{H^m(\Tn)})
        &\le P_{m + 1}(S(s)) P_{m + 2}(S(s))
        = P_{2m + 3}(S(s)), \\
    A_2(s) \norm{F(\bgamma)(s)}_{L^\iny(\Tn^2)}^2
        &\le P_3(S(s)) P_2(S(s))
        = P_5(S(s)).
\end{align*}
Hence,
\begin{align}\label{e:IntBoundForS}
    S(t)
        &\le S(0)
            + \int_0^t
                P_{2m + 3}(S(s))
            \, ds.
\end{align}
It follows from Osgood's lemma, \cref{L:Osgood}, that
\begin{align*}
    \int_{S(0)}^{S(t)}\frac{ds}{P_{2m + 3}(s)} \le t.
\end{align*}
This gives an upper bound on $S(t)$ and a lower bound on the time $T^*$ up to which $S(t)$ remains finite.

Although it is not necessary, we can obtain a more explicit bound by considering cases. For $S(0) \ge 2$, we need only obtain the bound under the assumption that $S(t) \ge 2$, in which case $s \ge 2$ in the entire integrand. We can see, then, that on $(2, \iny)$, for some constant $C > 0$,
\begin{align*}
    \frac{C}{2} s^{2m + 3} \le P_{2m + 3}(s) \le C s^{2m + 3}
\end{align*}
so
\begin{align*}
    \int_{S(0)}^{S(t)} &\frac{ds}{P_{2m + 3}(s)}
        \le \frac{1}{C} \int_{S(0)}^{S(t)} s^{-{2m + 3}} \, ds
        = \frac{1}{{(2m + 2)} C} \brac{\frac{1}{S(0)^{2m + 2}} - \frac{1}{S(t)^{2m + 2}}}
        \le t \\
        &\iff \frac{1}{S(0)^{2m + 2}} - \frac{1}{S(t)^{2m + 2}} \le {(2m + 2)} C t
        \iff \frac{1}{S(t)^{2m + 2}} \ge \frac{1}{S(0)^{2m + 2}} - {(2m + 2)} C t \\
        &\iff S(t)^{2m + 2} \le  \frac{1}{S(0)^{-{2m + 2}} - {(2m + 2)} Ct},
\end{align*}
which is possible if and only if $S(t)^{-{2m + 2}} > {(2m + 2)} Ct$; that is, up to $t < T^*$, where
\begin{align}\label{e:TOEBound}
    T^* = \frac{1}{{2m + 2} C S(0)^{2m + 2}}.
\end{align}

It is not hard to see that this bound improves for $S(0) < 2$, and hence this bound serves as an overly pessimistic lower bound on the time of existence for all $S(0) > 0$. Moreover, up to any $T < T^*$, we have the energy bound,
\begin{align}\label{e:SEnergyBound}
    \norm{\bgamma(t)}_{H^m(\Tn)}
            + \norm{F(\bgamma(t))}_{L^\iny(\Tn)}
            &\le
            \brac{\frac{1}{S(0)^{-{2m + 2}} - {(2m + 2)} CT}}^{\frac{1}{{2m + 2}}}.
\end{align}

Finally, we show that at least for some finite time, two components of a multiply connected domain $\Omega$ will not intersect.

Let $\bgamma_0(\eta_1)$, $\bgamma_0(\eta_2)$ be two points in distinct components of $\prt \Omega_0$, and let $r > 0$ be the distance between the two components at time zero. Then 
since $\prt_t \bgamma = L(\bgamma(t))$, we can use \cref{P:BoundsOnL} to estimate,
\begin{align*}
    &\abs{\bgamma(t, \eta_1)) - \bgamma(t, \eta_2)}
        = \abs[\bigg]
            {
            \bgamma_0(\eta_1) - \bgamma_0(\eta_2)
             + \int_0^t (\prt_s \bgamma(s, \eta_1)
                    - \prt_s \bgamma(s, \eta_2)) \, ds
            } \\
        &\qquad
        \ge r - \int_0^t \norm{\prt_s \bgamma(s, \eta)}
            _{L^\iny_\eta(\Tn)} \, ds
        = r - \int_0^t \norm{L(\bgamma(s, \eta))}
            _{L^\iny_\eta(\Tn)}\, ds \\
        &\qquad
        \ge r - t \norm{L}_{L^\iny([0, T] \times \Tn)}
        \ge r - C t \norm{L}_{L^\iny(0, T; H^1(\Tn))} \\
        &\qquad
        \ge r - C t \norm{S_1(\bgamma)}_{L^\iny([0, T] \times \Tn^2)}
                    \pr{\norm{\bgamma}_{L^\iny(0, T; H^2(\Tn))}
                        + \norm{\bgamma}_{L^\iny(0, T[ H^2(\Tn)}^2)} \\
        &\qquad
        \ge r - C(T, \bgamma_0) t,
\end{align*}
since we have a uniform bound on $\bgamma$ in $H^m(\Tn) \subseteq H^1(\Tn)$ over $[0, T]$. This shows that the two components will not intersect, at least up to a time that depends upon the initial data.

We cannot rule out the possibility that the CDE equation could continue beyond the time of intersection of two boundary components, though it would lose any obvious physical meaning.

\subsection{Step 4: The regularized CDE}\label{S:RegularizeCDE}

\phantom{x}\\

Departing slightly from (26) of \cite{Gancedo2008}, we define the regularized CDE,
\begin{align}\label{e:Le[sCDE}
	\prt_t \bgamma^\eps(t)
            &= L_\eps(\bgamma^\eps(t)), \quad
        L_\eps(\bgamma^\eps)
            := \phi_\eps * L(\phi_\eps * \bgamma^\eps),
\end{align}
where $L$ is defined in \cref{e:CDEPeriodic} and $\phi_\eps$ is a Friedrich's mollifier on $\Tn$: Parameterizing $\Tn$ by $[-\pi, \pi]$, we choose an even function $\phi_1 \ge 0$ in $C^\iny(\Tn)$ supported in $[-1, 1]$ with total mass $1$, and set $\phi_\eps(\cdot) = \eps \phi_1(\cdot/\eps)$. Then $\phi_\eps$ is well-defined on $\Tn$ for all $\eps \le 2 \pi$. Also, because $\phi_\eps$ is even,
\begin{align*}
    (\phi_\eps * f, g)_{H^k(\Tn)}
        &= (f, \phi_\eps * g)_{H^k(\Tn)}.
\end{align*}

Assume that $\bgamma$ is a non-self intersecting $H^m$ chain. Then $F(\bgamma)(\beta, \eta) < \iny$, which implies that $\bdelta_\beta(\eta) = 0$ if and only if $\beta = 0$. From this, we can see that $L_\eps(\bgamma^\eps) \in C^\iny(\Tn)$ as long as $\phi_\eps * \bdelta(\eta) = 0$ if and only if $\beta = 0$ if and only if $F(\phi_\eps * \bgamma)(\beta, \eta) < \iny$. By \cref{P:InitF} below, in fact, $F(\phi_\eps * \bgamma)(\beta, \eta) < \iny$ for all $\eps < \eps_0$, for some $\eps_0$ depending upon the initial data.

Hence, assuming that $\eps < \eps_0$, we can apply the Picard theorem, \cref{T:Picard}, with the open set
\begin{align*}
    O = \set{\bgamma \in C^\iny(\Tn) \colon F(\bgamma) < \iny,\ r(\bgamma)>0}
\end{align*}
to obtain a unique solution $\bgamma^\eps$ to \cref{e:Le[sCDE}.
Here, we have again used $r$ (now $r(\bgamma)$) to denote the minimum distance between pairs of
boundary components.
We let $T^*_\eps > 0$ give the maximal interval of existence $[0, T)$ for all $T \le T^*_\eps$ of the solution. (The time $T^*_\eps$ will depend upon $\eps_0$, an issue we will explore later.)

Since $\prt_t \bgamma^\eps = L_\eps(\bgamma^\eps)$, exploiting \cref{P:BoundsOnLPaired}, we have
\begin{align*}
    \frac{1}{2} \diff{}{t} \norm{\bgamma^\eps}_{H^m}^2
        &= (\bgamma^\eps, L_\eps(\bgamma^\eps))
                _{H^m(\Tn)}
        = (\bgamma^\eps,
            \phi_\eps * L(\phi_\eps * \bgamma^\eps))
                _{H^m(\Tn)} \\
        &= (\phi_\eps * \bgamma^\eps,
            L(\phi_\eps * \bgamma^\eps))
                _{H^m(\Tn)} \\
        &
        \le C \norm{S_m(\phi_\eps * \bgamma^\eps)}_{L^\iny(\Tn^2)}
            P_{m + 2}\left(\norm{\phi_\eps * \bgamma^\eps(t)}_{H^3(\Tn)}\right) \\
        &\le C \norm{S_m(\bgamma^\eps)}_{L^\iny(\Tn^2)}
            P_{m + 2}\left(\norm{\bgamma^\eps(t)}_{H^3(\Tn)}\right)
\end{align*}
for all $\eps < \eps_0$ by \cref{P:InitF} below.

Whereas the bound in \cref{e:gammLgammaHm} from Step 1 did not require assuming that $\bgamma$ satisfies the periodic CDE, the bounds in Step 2 do: in \cref{P:FBoundOverTime}, we assumed that $\prt_t \bgamma = L(\bgamma)$. This estimate must be replaced with
\begin{align*}
    \abs{\prt_t \bdelta_\beta^\eps(\eta)}
        &= \abs{\prt_t \bgamma^\eps(\eta)
            - \prt_t \bgamma^\eps(\eta - \beta)}
        = \abs{L_\eps(\bgamma^\eps)(\eta)
            - L_\eps(\bgamma^\eps)(\eta - \beta)} \\
        &= \abs{\phi_\eps * L(\phi_\eps * \bgamma^\eps)(\eta)
            - \phi_\eps * L(\phi_\eps *\bgamma^\eps)(\eta - \beta)} \\
        &\le \int_\Tn \phi_\eps(\xi)
            \abs{L(\phi_\eps *\bgamma^\eps)(\eta - \xi)
                - L(\phi_\eps *\bgamma^\eps)(\eta - \beta - \xi)}
                \, d \beta \\
        &\le \int_\Tn \phi_\eps(\xi)
            A_2(\phi_\eps * \bgamma^\eps, \al) \abs{\beta}
                \, d \beta
        \le C A_2(\phi_\eps * \bgamma^\eps, \al) \abs{\beta},
 \end{align*}
since by \cref{P:ContinuityOfL}, 
\begin{align*}
    \abs{L(\phi_\eps * \bgamma^\eps)(\eta)
            - L(\phi_\eps *\bgamma^\eps)(\eta - \beta)} 
        &\le A_2(\phi_\eps * \bgamma^\eps, \al) \abs{\beta}.
\end{align*}
Then, applying \cref{L:ROCMollification} with $f = L(\phi_\eps * \bgamma^\eps)$ (so $f$ itself in \cref{L:ROCMollification} depends on $\eps$),  
\begin{align*}
    \abs{\prt_t \bdelta_\beta^\eps(\eta)}
        &\le A_2(\bgamma^\eps, \al) (\abs{\beta} + C \eps).
\end{align*}
Applying \cref{P:InitF} below, we see that the bound from Step 2 in \cref{e:FBound} becomes
\begin{align}\label{e:FepsBound}
    \norm{F(\bgamma^\eps(t))}_{L^\iny(\Tn^2)}
        &\le \norm{F(0)}_{L^\iny(\Tn^2)}
            + \int_0^t A_2(\bgamma^\eps(s), \al)
                \norm{F(\bgamma^\eps(s))}_{L^\iny(\Tn^2)}^2
                \, ds,
\end{align}
for all $\eps < \eps_0$ up to time $T^*_\eps$.

These estimates give the following:
\begin{prop}\label{P:BoundsUnifIneps}
    The bounds in \cref{e:CombinedEnergyBound} and \cref{e:SEnergyBound} of Step 3 hold for $\bgamma^\eps$ in place of $\bgamma$ uniformly for $\eps < \eps_0$, and we obtain a bound on the time of existence, $T^*_\eps$, of the regularized solutions that is uniform over $\eps < \eps_0$.    
\end{prop}

There is a caveat, however, to our estimate in \cref{e:FepsBound}, as to continue to apply \cref{P:InitF} at time $t \in [0, T^*]$, we must recalculate $\eps_0$, creating in effect $\eps_0(t)$. So \cref{e:FepsBound} holds only up to some non-constructive time given by the Picard theorem. This applies then to the time $T^*$ we derived from \cref{e:CombinedEnergyBound} and \cref{e:SEnergyBound}. By decreasing $\eps_0$ arbitrarily close to $0$, we can extend the estimate on $T^*$ to arbitrarily close to that given by \cref{e:CombinedEnergyBound} and \cref{e:SEnergyBound}.

\begin{prop}\label{P:InitF}
    Let $\bgamma$ be an $H^3$ chain with $\norm{F_0(\bgamma)}_{L^\iny(\Tn^2)} < \iny$ and nonzero and set
    \begin{align*}
        \eps_0
            &= \frac{C}{\norm{F_0(\bgamma)}_{L^\iny(\Tn^2)}
                    \norm{\bgamma}_{H^3}}
                \min
                \set{
                    \frac
                        {1}
                        {\norm{F_0(\bgamma)}_{L^\iny(\Tn^2)}
                            \norm{\bgamma}_{H^3}^2},
                    1
                },  
    \end{align*}
    where $C$ is an absolute constant. For all $\eps \in (0, \eps_0)$,
    \begin{align*}
        \norm{F(\phi_\eps * \bgamma)}_{L^\iny(\Tn^2)}
            &\le 2 \norm{F(\bgamma)}_{L^\iny(\Tn^2)}, \\
        \norm{S_n(\phi_\eps * \bgamma)}_{L^\iny(\Tn^2)}
            &\le 2 \norm{S_n(\bgamma)}_{L^\iny(\Tn^2)},
                \, n = 0, 1, 2, 3.
    \end{align*}
    \end{prop}
\begin{proof}
    Because $\norm{F_0(\bgamma)}_{L^\iny(\Tn^2)} < \iny$, it must be that $\bdelta_\beta(\eta)$ vanishes only when $\beta = 0$. Moreover,
    \begin{align*}
        \sup_\eta \frac{1}{\abs{\prt_\eta \bgamma(\eta)}}
            &= \sup_\eta F_0(\bgamma)(0, \eta)
            \le \norm{F_0(\bgamma)}_{L^\iny(\Tn^2)}
            < \iny,
    \end{align*}
    so for all $\eta$,
    \begin{align*}
        \abs{\prt_\eta \bgamma(\eta)}
            > a
            := \frac{1}{\norm{F_0(\bgamma)}_{L^\iny(\Tn^2)}}.
    \end{align*}
    But then by \cref{L:ROCMollification}, there exists $\eps_0 \in (0, 2 \pi)$ such that, for all $\eps \in (0, \eps_0)$,
    \begin{align*}
        \abs{\phi_\eps * \prt_\eta \bgamma(\eta)} \ge \frac{\sqrt{3} a}{2}.
    \end{align*}

    Let $\beta \ne 0$. By the mean value theorem, for some $\eta_1, \eta_2$ between $\eta$ and $\eta - \beta$,
    \begin{align*}
       &\frac{1}{F_0(\phi_\eps * \bgamma)(\beta, \eta)^2}
            = \frac{\abs{\phi_\eps * \bdelta_\beta(\eta)}^2}{\abs{\beta}^2}
            = \frac{\abs{\phi_\eps * \bgamma(\eta)
                    - \phi_\eps * \bgamma(\eta - \beta)}^2}
                    {\abs{\beta}^2} \\
            &\qquad
            = \frac{\abs{\phi_\eps * \bgamma^1(\eta)
                    - \phi_\eps * \bgamma^1(\eta - \beta)}^2}
                    {\abs{\beta}^2}
                + \frac{\abs{\phi_\eps * \bgamma^2(\eta)
                    - \phi_\eps * \bgamma^2(\eta - \beta)}^2}
                    {\abs{\beta}^2} \\
            &\qquad
            = \abs{\phi_\eps * \prt_\eta \bgamma^1(\eta_1)}^2
                    + \abs{\phi_\eps * \prt_\eta \bgamma^2(\eta_2)}^2 \\
            &\qquad
            = \abs{\phi_\eps * \prt_\eta \bgamma^1(\eta_1)}^2
                    + \abs{\phi_\eps * \prt_\eta \bgamma^2(\eta_1)}^2
                + [\abs{\phi_\eps * \prt_\eta \bgamma^2(\eta_2)}^2
                    - \abs{\phi_\eps * \prt_\eta \bgamma^2(\eta_1)}^2] \\
            &\qquad
            = \abs{\phi_\eps * \prt_\eta \bgamma(\eta_1)}^2
                - [\abs{\phi_\eps * \prt_\eta \bgamma^2(\eta_1)}^2
                    - \abs{\phi_\eps * \prt_\eta \bgamma^2(\eta_2)}^2].
    \end{align*}
    Now,
    \begin{align*}
        &\abs{\abs{\phi_\eps * \prt_\eta \bgamma^2(\eta_1)}^2
                    - \abs{\phi_\eps * \prt_\eta \bgamma^2(\eta_2)}^2} \\
            &\qquad
            = \brac{\abs{\phi_\eps * \prt_\eta \bgamma^2(\eta_1)}
                    + \abs{\phi_\eps * \prt_\eta \bgamma^2(\eta_2)}}
                \abs{\abs{\phi_\eps * \prt_\eta \bgamma^2(\eta_1)}
                    - \abs{\phi_\eps * \prt_\eta \bgamma^2(\eta_2)}} \\
            &\qquad
            \le C \norm{\phi_\eps * \bgamma}_{H^2}
                \abs{\prt_\eta \phi_\eps * \bgamma^2(\eta_1)
                    - \prt_\eta \phi_\eps * \bgamma^2(\eta_2)} \\
            &\qquad
            \le C \norm{\bgamma}_{H^2}
                \norm{\prt_\eta^2 \phi_\eps * \bgamma}_{L^\iny}
                \abs{\eta_1 - \eta_2} \\
            &\qquad
            \le C \norm{\bgamma}_{H^2}
                \norm{\phi_\eps * \bgamma}_{H^3} \abs{\eta_1 - \eta_2}
            \le C \norm{\bgamma}_{H^3}^2 \abs{\beta},
    \end{align*}
    where we used the reverse triangle inequality and \cref{L:C2BoundFromAgmon}. Hence,
    \begin{align*}
        \frac{1}{F_0(\phi_\eps * \bgamma)(\beta, \eta)^2}
            &\ge \abs{\phi_\eps * \prt_\eta \bgamma(\eta_1)}^2
                - C_0 \norm{\bgamma}_{H^3}^2 \abs{\beta}
            \ge \frac{3 a}{4} - C_0 \norm{\bgamma}_{H^3}^2 \abs{\beta} \\
            &= \frac{3}{4 F_0(\bgamma)(\beta, \eta)}
                - C_0 \norm{\bgamma}_{H^3}^2 \abs{\beta}
            \ge \frac{1}{4 F_0(\bgamma)(\beta, \eta)}
    \end{align*}
    so
    \begin{align}\label{e:F0F0Ineq}
        F_0(\phi_\eps * \bgamma)(\eta, \beta)
            \le 2 (\bgamma)(\eta, \beta)
    \end{align}
    for all $\abs{\beta} \le \beta_0$, where
    \begin{align*}
        \beta_0
            &:= \frac{1}{2 C_0 F_0(\bgamma)(\beta, \eta) \norm{\bgamma}_{H^3}^2}.
    \end{align*}
    
    Then for $\abs{\beta} > \beta_0$
    \begin{align*}
        \sup_\eta \frac{\abs{\beta}}{\abs{\bdelta_\beta(\eta)}}
            &= \sup_\eta F_0(\bgamma)(\beta, \eta)
            \le \norm{F_0(\bgamma)}_{L^\iny(\Tn^2)}
            < \iny,
    \end{align*}
    so for all $\eta$ and all $\abs{\beta} > \beta_0$, $\frac{\abs{\bdelta_\beta(\eta)}}{\abs{\beta}} \ge a$, and hence
    \begin{align*}
            \abs{\bdelta_\beta(\eta)}
                \ge \abs{\beta} a
                \ge \beta_0 a.
    \end{align*}
    Then by \cref{L:ROCMollification},
    \begin{align*}
        \abs{\phi_\eps * \bdelta_\beta(\eta) - \bdelta_\beta(\eta)}
            \le 2 \pi \eps \norm{\bdelta_\beta}_{C^{0, 1}}
            \le C_1 \eps \norm{\bgamma}_{H^2}.
    \end{align*}
    If
    \begin{align*}
        \eps
            < \eps_0
                := \frac
                {\beta_0 a}
                {2 C_1 \norm{\bgamma}_{H^2}}
            < \frac
                {\bdelta_\beta(\eta)}
                {2 C_1 \norm{\bgamma}_{H^2}}
    \end{align*}
    then
    \begin{align*}
        \abs{\phi_\eps * \bdelta_\beta(\eta) - \bdelta_\beta(\eta)}
                < \frac{1}{2} \abs{\bdelta_\beta(\eta)}
            \text{ so }
            \abs{\phi_\eps * \bdelta_\beta(\eta)}
                \ge \frac{1}{2} \abs{\bdelta_\beta(\eta)}
    \end{align*}
    and
    \begin{align*}
        \frac{\abs{\phi_\eps * \bdelta_\beta(\eta)}}{\abs{\beta}}
            &\ge \frac{\abs{\bdelta_\beta(\eta)}}{2 \abs{\beta}},
    \end{align*}
    which gives \cref{e:F0F0Ineq} for all $\eta$ and $\abs{\beta} > \beta_0$.

    Noting that
    \begin{align*}
        \eps_0
            &= \frac
                {1}
                {2 C_1 \norm{F_0(\bgamma)}_{L^\iny(\Tn^2)} \norm{\bgamma}_{H^2}}
                \frac{1}{2 C_0 F_0(\bgamma)(\beta, \eta) \norm{\bgamma}_{H^3}^2},
    \end{align*}
    gives the bound for $F_0$.

    The bound for $F_j$, $j \ne 0$, is obtained similarly, though now we use that $\bdelta_0(\eta) = 0$, so
    \begin{align*}
        \abs{\phi_\eps * \bdelta_\beta(\eta) - \n_j}
            &\ge \abs{\n_j}
                - \norm{\phi_\eps * \prt_\beta \bdelta_\beta(\eta)}_{L^\iny}
                    \abs{\beta}
            = \abs{j}
                - \norm{\phi_\eps * \prt_\beta \bgamma(\eta - \beta)(\eta)}_{L^\iny}
                    \abs{\beta} \\
            &= \abs{j} - \norm{\prt_\beta \phi_\eps * \bgamma}_{L^\iny}
                    \abs{\beta}
            \ge 1 - \norm{\bgamma}_{H^3} \abs{\beta}.
    \end{align*}
    So setting $\beta_0 = (2 \norm{\bgamma}_{H^3})^{-1}$, \cref{e:F0F0Ineq} holds for all $\abs{\beta} \le \beta_0$, The argument for $\abs{\beta} > \beta_0$ is the same as for $F_0$, because the $\abs{\beta}$ in the numerator of $F_0(\bgamma)(\beta, \eta)$ played only a passive role in that estimate. Noting that
    \begin{align*}
        \eps_0
            = \frac
                {\beta_0 a}
                {2 C_1 \norm{\bgamma}_{H^2}}
            = \frac{1}{2 \norm{F_0(\bgamma)}_{L^\iny(\Tn^2)} \norm{\bgamma}_{H^3}},
    \end{align*}
     we obtain the bound for $F_j$ and, combined with the bound for $F_0$, we obtain the bound for $F$. The bound for $S_n$ then follows from the bound for $F$.
\end{proof}

\begin{remark} Of course, we also have that if $r(\bgamma)>0,$ then for sufficiently
small $\varepsilon,$ we also have $r(\phi_{\epsilon}*\bgamma)>0.$
\end{remark}

\subsection{Step 5: Convergence to a solution}\label{S:Convergence}    

    Let $(\bgamma^\eps)_{\eps < \eps_0}$ be the family of solutions from Step 4. Then
    \begin{align*}
        \norm{\prt_t \bgamma^\eps}_{L^2(\Tn)}
            &= \norm{L_\eps(\bgamma^\eps)}_{L^2(\Tn)}
            = \norm{\phi_\eps * L(\phi_\eps * \bgamma^\eps)}
                _{L^2(\Tn)}
            \le \norm{L(\phi_\eps * \bgamma^\eps)}
                _{L^2(\Tn)} \\
            &\le C \norm{S_0(\phi_\eps * \bgamma^\eps)}_{L^\iny}
                    \norm{\phi_\eps * \bgamma^\eps}_{H^1}
            \le C \norm{S_0\bgamma^\eps)}_{L^\iny}
                    \norm{\bgamma^\eps}_{H^1},
    \end{align*}
    using \cref{P:BoundsOnL,P:InitF}.

    Let $0 < r < m$ and fix $T \le T^*$. Using the uniform bound in \cref{P:BoundsUnifIneps}, we apply \cref{L:AubinLions} with $X_0 = H^m$, $X_1 = H^r$, $X_0 = L^2$ to conclude that there exists $\bgamma \in C([0, T]; H^r(\Tn))$ and some sequence $(\eps_n)$ with $\eps_n$ decreasing to $0$ for which $\bgamma^{\eps_n} \to \bgamma$ in $C([0, T]; H^r(\Tn))$.
    
    Let $E = L^1(0, T; L^2(\Tn))$, which is separable. Then $E^* = L^\iny(0, T; L^2(\Tn))$ and $(\prt_t \bgamma^\eps)_{0 < \eps < \eps_0}$ is uniformly bounded in $E^*$, so \cref{L:WeakStarConv} gives a $\bnu \in E^*$ and some subsequence $(\eps_n)$ which we relabel $(\eps_n)$ for which
    \begin{align*}
        (\prt_t \bgamma^{\eps_n}, \varphi)_{E^*, E}
            \to (\bnu, \varphi)_{E^*, E}
            \text{ for all } \varphi \in E.
    \end{align*}
    But, for any $\varphi \in \Cal{D}((0, T) \times \Tn) \subseteq E$,
    \begin{align*}
        (\prt_t \bgamma^{\eps_n}, \varphi)_{E^*, E}
            &= - (\bgamma^{\eps_n}, \prt_t \varphi)_{\Cal{D}', \Cal{D}}
            \to - (\bgamma, \prt_t \varphi)_{\Cal{D}', \Cal{D}}
            = (\prt_t \bgamma, \varphi)_{\Cal{D}', \Cal{D}}
    \end{align*}
    so $\bnu = \prt_t \bgamma$ in $E^*$.
    
    Also, using \cref{P:ContinuityOfL},
    \begin{align*}
        &\lim_{n \to \iny} \norm{L_{\eps_n}(\bgamma^{\eps_n}) - L(\bgamma)}_{L^2}
            = \lim_{n \to \iny}\norm{\phi_{\eps_n} * L(\phi_{\eps_n} * \bgamma^{\eps_n})
                - L(\bgamma)}_{L^2} \\
            &\qquad
            \le
                \lim_{n \to \iny} \norm{\phi_{\eps_n} * L(\phi_{\eps_n} * \bgamma^{\eps_n})
                - \phi_{\eps_n} * L(\bgamma)
                }_{L^2}
                + \lim_{n \to \iny}\norm{\phi_{\eps_n} * L(\bgamma)
                - L(\bgamma)
                }_{L^2} \\
            &\qquad
            \le
                \lim_{n \to \iny} \norm{L(\phi_{\eps_n} * \bgamma^{\eps_n})
                - L(\bgamma)
                }_{L^2}
                + 0
            \le
                \lim_{n \to \iny} B_1(\bgamma, \phi_\eps * \bgamma) 
                \norm{\phi_{\eps_n} * \bgamma^{\eps_n} - \bgamma}_{H^1} \\
            &\qquad
            \le
                B_1(\bgamma, \bgamma) \lim_{n \to \iny}
                \norm{\phi_{\eps_n} * \bgamma^{\eps_n}
                    - \phi_{\eps_n} * \bgamma}_{H^1}
                +
                B_1(\bgamma, \bgamma)  \lim_{n \to \iny}
                \norm{\phi_{\eps_n} * \bgamma
                    - \bgamma}_{H^1} \\
            &\qquad
            \le
                B_1(\bgamma, \bgamma)  \lim_{n \to \iny}
                \norm{\bgamma^{\eps_n}
                    - \bgamma}_{H^1}
                +
                0
            = 0,
    \end{align*}
    since $\bgamma^{\eps_n} \to \bgamma$ in $C([0, T]; H^r(\Tn))$ for any $r < m$. We also used \cref{P:InitF} to conclude that $\lim_{\eps \to 0} B_1(\bgamma, \phi_\eps * \bgamma) = B_1(\bgamma, \bgamma)$ (unless $\bgamma \equiv 0$, in which case there is nothing to be proven).
    
    Then because $L_{\eps_n}(\bgamma^{\eps_n}) = \prt_t \bgamma^{\eps_n}$ and $\prt_t \bgamma^{\eps_n}$ converges weak-$*$ in  $L^\iny(0, T; L^2(\Tn))$ to $\prt_t \bgamma$, we conclude that $\prt_t \bgamma = L(\bgamma)$ and, for all $0 < r < m$,
    \begin{align*}
        \bgamma
            \in W
            := &\set{
                \bgamma \in C([0, T]; H^r(\Tn))
                    \cap L^\iny([0, T]; H^m(\Tn)), \,
                \prt_t \bgamma \in C([0, T]; L^2(\Tn))
                }.
    \end{align*}
    
    We now show that, in fact, $\bgamma \in C([0, T]; H^m(\Tn))$ following along the lines of the similar argument for the Euler equations on pages 110-111 of \cite{MB2002}.
    
    First, we show $\bgamma \in C_W([0, T]; H^m(\Tn))$, meaning that for any $\varphi \in H^{-m}(\Tn)$, $\innp{\varphi, \bgamma(t)}$ is a continuous function of $t$ over $[0, T]$. Here, $\innp{\cdot, \cdot} = \innp{\cdot, \cdot}_{H^m, H^{-m}}$ is the pairing in the duality between $H^m(\Tn)$ and $H^{-m}(\Tn)$. We will use that if $r < m$, then $H^{-r} \subseteq H^{-m}$ and for any $f \in H^{-r}$, $g \in H^m$, $\innp{f, g}_{H^m, H^{-m}} = \innp{f, g}_{H^r, H^{-r}}$.
    
    Because $\bgamma^\eps \to \bgamma \in C([0, T]; H^r(\Tn))$, for any $t \in [0, T]$, any $r < m$, and any $\psi \in H^{-r}(\Tn)$,
    \begin{align*}
        \innp{\psi, \bgamma^\eps(t)} 
            &\to \innp{\psi, \bgamma(t)} 
    \end{align*}
    uniformly on $[0, T]$. Let $\varphi \in H^{-m}$. Since $H^{-r}$ is dense in $H^{-m}$, we can find a sequence $(\psi_n)$ in $H^{-r}$ with $\norm{\psi_n - \varphi}_{H^{-m}} \le n^{-1}$. Then,
    \begin{align*}
        &\abs{\innp{\varphi, \bgamma(t) - \bgamma^\eps(t)}} 
            \le \abs{\innp{\varphi - \psi_n,
                \bgamma(t) - \bgamma^\eps(t)}} 
                                + \abs{\innp{\psi_n,
                    \bgamma(t) - \bgamma^\eps(t)}} \\ 
            &\qquad
            \le \norm{\varphi - \psi_n}_{H^{-m}}
                    [\norm{\bgamma}_{L^\iny(0, T; H^m)}
                        + \norm{\bgamma^\eps}_{L^\iny(0, T; H^m)}]
                    + \abs{\innp{\psi_n, \bgamma(t) - \bgamma^\eps(t)}} \\
                                &\qquad
            \le C n^{-1}
                    + \abs{\innp{\psi_n, \bgamma(t) - \bgamma^\eps(t)}},
    \end{align*}
    using that $(\bgamma^\eps)$ is uniformly bounded in $L^\iny(0, T; H^m(\Tn))$.
    
    Given $n$, we can choose $\eps = \eps(n)$ sufficiently small that
    \begin{align*}
        \abs{\innp{\psi_n, \bgamma(t) - \bgamma^\eps(t)}}
            \le \norm{\psi_n}_{H^{-r}}
                \norm{\bgamma(t) - \bgamma^\eps(t)}_{H^r}
            \le n^{-1}
    \end{align*}
    uniformly over $[0, T]$, so that 
    \begin{align*}
        &\abs{\innp{\varphi, \bgamma(t) - \bgamma^\eps(t)}} 
            \le C n^{-1}.
    \end{align*}
    This shows that $\innp{\varphi, \bgamma^\eps(t)} \to \innp{\varphi, \bgamma(t)}$ uniformly on $[0, T]$ for any $\varphi \in H^{-m}$, which is enough to imply that $\bgamma \in C_W([0, T]; H^m(\Tn))$.
  
    The second step is to show that $\norm{\bgamma(t)}_{H^m(\Tn)}$ is continuous in time. The proof proceeds as for the $\nu = 0$ case in \cite{MB2002}, using the uniform energy bound coming from \cref{{e:IntBoundForS}} in place of the bound in (3.60) of \cite{MB2002} and using that $\prt_t \bgamma^\eps(t) = L^\eps(\bgamma^\eps(t))$, like the Euler equations, is time reversible.
    
    The time continuity of $\bgamma$ comes from \cref{P:BoundsOnL} and that $\prt_t \bgamma = L(\bgamma)$.

\subsection{Uniqueness}\label{S:Uniqueness}

We now obtain uniqueness in \cref{T:ExistenceAlphaLt1}.

Let $\bgamma_1$, $\bgamma_2$ be two periodic CDE solutions in $C([0, T]; H^m(\Tn))$ with $\prt_t \bgamma \in C([0, T]; L^2(\Tn)$ for which $\bgamma_1(0)= \bgamma_2(0) = \bgamma_0$. Let $\ol{\bgamma} := \bgamma_1 - \bgamma_2$.

Arguing as in the proof of \cref{P:LSymmetry}, let 
    \begin{align*}
        \bmu_j(\beta, \eta)
            &:= \bgamma_j(\beta) - \bgamma_j(\eta)
    \end{align*}
    and write $L(\bgamma_j)$ in the form of \cref{e:CDEOnPiProto}:
    \begin{align*}
    	L(\bgamma_j(\beta))
                = \int_\Tn
    				G^\al_p(\bmu_j(\beta, \eta))
    					\prt_2 \bmu_j(\beta, \eta)
    					\, d \eta.
    \end{align*}
	Then
    \begin{align*}
        \begin{split}
	    \frac{1}{2} &\diff{}{t} \norm{\ol{\bgamma}(t)}_{L^2(\Tn)}^2
        	= (\ol{\bgamma}, \, L(\bgamma_1) - L(\bgamma_2))_{L^2{}} \\
            &= \int_\Tn \int_\Tn
    				\ol{\bgamma}(\beta) \cdot
						\brac{G^\al_p(\bmu_1(\beta, \eta)) \prt_2 \bmu_1(\beta, \eta)
							- G^\al_p(\bmu_2(\beta, \eta)) \prt_2 \bmu_2(\beta, \eta)}
    					\, d \eta \, d \beta \\
            &= \int_\Tn \int_\Tn
    				\ol{\bgamma}(\eta) \cdot
						\brac{G^\al_p(\bmu_1(\eta, \beta)) \prt_2 \bmu_1(\eta, \beta)
							- G^\al_p(\bmu_2(\eta, \beta)) \prt_2 \bmu_2(\eta, \beta)}
    					\, d \beta \, d \eta \\
             &= -\int_\Tn \int_\Tn
    				\ol{\bgamma}(\eta) \cdot
						\brac{G^\al_p(\bmu_1(\beta, \eta)) \prt_2 \bmu_1(\beta, \eta)
							- G^\al_p(\bmu_2(\beta, \eta)) \prt_2 \bmu_2(\beta, \eta)}
    					\, d \eta \, d \beta.
        \end{split}
    \end{align*}
    In the second equality, we simply switched variable names $\eta$ and $\beta$, while in the final equality we used that $G^\al_p(-\x) = G^\al_p(\x)$ and $\bmu_j(\eta, \beta) = - \bmu_j(\beta, \eta)$ and switched the order of integration. Taking the average of the second and fourth expressions, we see that
    \begin{align*}
        \diff{}{t} &\norm{\ol{\bgamma}(t)}_{L^2(\Tn)}^2 \\
            &= \int_\Tn \int_\Tn
    				(\ol{\bgamma}(\beta) - \ol{\bgamma}(\eta)) \cdot
						\pbrac{G^\al_p(\bmu_1(\beta, \eta)) \prt_2 \bmu_1(\beta, \eta)
							- G^\al_p(\bmu_2(\beta, \eta)) \prt_2 \bmu_2(\beta, \eta)}
    					\, d \eta \, d \beta.
    \end{align*}
    Now let
\begin{align*}
    \bdelta_\beta^j(\eta)
        &:= \bgamma_j(\eta) - \bgamma_j(\eta - \beta), \\
    \ol{\bdelta}_\beta(\eta)
        &:= \bdelta_\beta^1(\eta) - \bdelta_\beta^2(\eta)
        = \ol{\bgamma}(\eta) - \ol{\bgamma}(\eta - \beta).
\end{align*}
    Making the change of variables, $\beta \mapsto \eta - \beta$, but leaving $\eta$ unchanged, using that $\bgamma$ is quasi-closed, that $\bmu_j(\eta - \beta, \eta) = \bgamma_j(\eta - \beta) - \bgamma_j(\eta) = - \bdelta_\beta^j(\eta)$, that $\ol{\bgamma}(\eta - \beta) - \ol{\bgamma}(\eta) = \bdelta_\beta^2(\eta) - \bdelta_\beta^1(\eta) = -\ol{\bdelta}_\beta(\eta)$, and that $\prt_2 \bdelta_\beta^j(\eta) = \prt_\eta \bdelta_\beta^j(\eta)$, we have,
    \begin{align*}
        \diff{}{t} &\norm{\ol{\bgamma}(t)}_{L^2(\Tn)}^2 \\
            &= \int_\Tn \int_\Tn
    				(-\ol{\bdelta}_\beta(\eta)) \cdot
						\pbrac{G^\al_p(- \bdelta_\beta^1(\eta)) \prt_2 (-\bdelta_\beta^1(\eta))
							- G^\al_p(- \bdelta_\beta^2(\eta)) \prt_2 (- \bdelta_\beta^2(\eta))}
    					\, d \eta \, (-d \beta) \\
            &= -\int_\Tn \int_\Tn
    				\ol{\bdelta}_\beta(\eta) \cdot
						\pbrac{G^\al_p(\bdelta_\beta^1(\eta)) \prt_2 \bdelta_\beta^1(\eta)
							- G^\al_p(\bdelta_\beta^2(\eta)) \prt_\eta \bdelta_\beta^2(\eta)}
    					\, d \eta \, d \beta.
    \end{align*}

We make the decomposition as in \cite{Gancedo2008},
\begin{align*}
    \diff{}{t} \norm{\ol{\bgamma}(t)}_{L^2(\Tn)}^2 
        = I_1 + I_2,
\end{align*}
where
\begin{align*}
    I_1
        &:= -\int_\Tn \int_\Tn
            \ol{\bdelta}_\beta(\eta) \cdot
            \brac{G^\al_p(\bdelta_\beta^1(\eta))
                - G^\al(\bdelta_\beta^2(\eta))}
				\prt_\eta \bdelta_\beta^1(\eta)
            \, d \eta \, d \beta, \\
    I_2
        &:= -\int_\Tn \int_\Tn
            \ol{\bdelta}_\beta(\eta) \cdot
            G^\al_p(\bdelta_\beta^2(\eta))
				\prt_\eta \ol{\bdelta}_\beta(\eta)
            \, d \eta \, d \beta.
\end{align*}

From the proof of \cref{P:ContinuityOfL},
    \begin{align*}
        &\abs{G^\al(\bdelta_\beta^1(\eta))
                    - G^\al(\bdelta_\beta^2(\eta))}
            \le C_0 \abs{\bdelta_\beta^2(\eta) - \bdelta_\beta^1(\eta)}
                    \abs{\beta}^{-1 -\al}
            = C_0 \abs{\ol{\bdelta}_\beta(\eta)}
                    \abs{\beta}^{-1 -\al},
    \end{align*}
so
\begin{align*}
	\abs{I_1}
		&\le
		C_0 \int_\Tn \int_\Tn
			\frac{1}{\abs{\beta}^{1 + \al}}
            \abs{\ol{\bdelta}_\beta(\eta)}^2
            \norm{\prt_\eta \bdelta_\beta^1(\eta)}_{L^\iny}
            \, d \eta \, d \beta \\
		&\le
		C_0 \norm{\bgamma_1}_{H^3} \int_\Tn
			\frac{\abs{\beta}}{\abs{\beta}^{1 + \al}}
            \int_\Tn \abs{\ol{\bdelta}_\beta(\eta)}^2
            \, d \eta \, d \beta
        \le C_0 \norm{\bgamma_1}_{H^3} \norm{\ol{\bgamma}}_{L^2}^2,
\end{align*}
using \cref{L:prtDeltaSobolevBound} in the last inequality.

For $I_2$, we have,
\begin{align*}
	\abs{I_2}
        &= - \frac{1}{2} \int_\Tn \int_\Tn
            G^\al_p(\bdelta_\beta^2(\eta))
				\prt_\eta \abs{\ol{\bdelta}_\beta(\eta)}^2
            \, d \eta \, d \beta
        = \frac{1}{2} \int_\Tn \int_\Tn
            \prt_\eta G^\al_p(\bdelta_\beta^2(\eta))
				\abs{\ol{\bdelta}_\beta(\eta)}^2
            \, d \eta \, d \beta \\
        &= \frac{1}{2} \int_\Tn \int_\Tn
            (\grad G^\al_p(\bdelta_\beta^2(\eta))
            \prt_\eta \bdelta_\beta^2(\eta))
				\abs{\ol{\bdelta}_\beta(\eta)}^2
            \, d \eta \, d \beta \\
        &\le C \int_\Tn \int_\Tn
        	\frac{1}{\abs{\beta}^{1 + \al}}
				\norm{\prt_\eta \bdelta_\beta^2(\eta)}_{L^\iny}
				\abs{\ol{\bdelta}_\beta(\eta)}^2
            \, d \eta \, d \beta \\
        &\le C \norm{\bgamma_2}_{H^3} \int_\Tn
        	\frac{\abs{\beta}}{\abs{\beta}^{1 + \al}}
			\int_\Tn \abs{\ol{\bdelta}_\beta(\eta)}^2
            \, d \eta \, d \beta
        = C \norm{\bgamma_2}_{H^3} \norm{\ol{\bgamma}}_{L^2}^2.
\end{align*}

Combined, these bounds give
\begin{align*}
    \diff{}{t}
        \norm{\bgamma(t)}_{L^2(\Tn)}^2
        \le C \norm{\bgamma(t)}_{L^2(\Tn)}^2,
\end{align*}
where $C$ depends upon the $H^3$ norms of $\bgamma_1$ and $\bgamma_2$. We conclude from \Gronwalls lemma that $\bgamma \equiv 0$, giving uniqueness, and completing the proof.

\begin{remark}
In Section 5 of \cite{Gancedo2008}, Gancedo formulates an alternate CDE for $\al = 1$ for which the parameterization $\bgamma(t, \cdot)$ is ``constant-speed'' in the sense that
$
	\abs{\prt_\eta \bgamma(t, \eta)}^2 = A(t)
		\text{ for all } \eta \in \Tn
$
for some differentiable function $A$. To do this, an additional term must be added to the CDE, transforming it, as in (31) of \cite{Gancedo2008}, to
\begin{align*}
	\prt_t \bgamma(t, \eta)
		&= L(\bgamma(t, \eta))
			+ \la(\bgamma(t, \eta)) \prt_\eta \bgamma(t, \eta).
\end{align*}
	
	Without the addition of $\la \prt_\eta \bgamma$, one can easily adapt the uniqueness proof
	we just gave to use the
	log-Lipschitz modulus of continuity $\mu$ (see \cref{R:LIsLL}), and complete the uniqueness
	argument by applying Osgood's lemma in place of \Gronwalls lemma. With the addition of
	$\la(\bgamma(t, \eta)) \prt_\eta \bgamma(t, \eta)$, however, such an argument fails,
	because, ultimately, $\norm{\lambda(\bgamma_1) - \lambda(\bgamma_2)}_{L^2}$ cannot be
	bounded sufficiently well in terms of $\norm{\bgamma_1 - \bgamma_2}_{L^2}$.
\end{remark}

\appendix

%
%
\section{Bounds on $R^\al$}\label{A:RalBounds}

\begin{proof}[\textbf{Proof of \cref{P:BoundsOnR}}]
    In outline, we proceed as follows. First, we show that bounding $R^\al$ on $\Pi_p$ gives the core of the estimates. Second, we establish \cref{e:RalSimplified}. Third, we use \cref{e:RalSimplified} to bound $R^\al$ in $L^\iny(\Pi_p)$, a step that involves the most delicate estimates. Fourth, we bound derivatives of $\Pi_p$ using \cref{e:RalSimplified}. Fifth, we incorporate simple estimates on the singularities of $R^\al$ at each point in $\L^*$ to obtain the bounds in \cref{e:RBoundOnR2MinusLStar}.
    
    \medskip
    \noindent\textbf{Step 1}: First, we show that it is almost sufficient to bound $R^\al$ just on $\Pi_p$.
    
    From \cref{e:Galp}, $R^\al = G^\al_p - G^\al$. Now, $\grad^\perp G^\al_p = K^\al_p$ is periodic and defined on $\R^2 \setminus \L$. Hence, $G^\al_p$ is the sum of a periodic and an affine function, so
    \begin{align*} 
        R^\al(\x) = f_{per}(\x) + (a + \bb \cdot \x) - G^\al(\x)
            \text{ for } \x \in \R^2 \setminus \L
    \end{align*}
    for some periodic function $f_{per}$ defined on $\R^2 \setminus \L$, scalar constant $a$, and vector constant $\bb$. Hence, if we can bound $R^\al$ on $\Pi_p$, the bounds on all of $\R^2 \setminus \L$ will follow.

    More precisely,  since $R^\al$ is $C^\iny$ on $\Pi_p$, it must be that $f_{per}(\x) - G^\al(\x) \in C^\iny(\Pi_p)$. Let $\x \in \Pi_p$ and $\y = \x + \n_j$. Then
    \begin{align*}
        R^\al(\y)
            &= f_{per}(\x + \n_j) + (a + \bb \cdot (\x + \n_j))
                - G^\al(\y) \\
            &= f_{per}(\x) + (a + \bb \cdot \x) - G^\al(\x)
                + \bb \cdot \n_j + G^\al(\x) - G^\al(\y) \\
            &= R^\al(\x)
                + b_1 j + G^\al(\x) - G^\al(\y).
    \end{align*}
    We can write this as
    \begin{align}\label{e:RalOnR2}
        R^\al(\x)
            &= R^\al(\Cover(\x)) + G^\al(\Cover(\x))
                + b_1 (\x - \Cover(\x))
                - G^\al(\x),
    \end{align}
    where $\Cover$ is the covering map of \cref{e:CoveringMap}.

    \medskip
    \noindent\textbf{Step 2}: Next, we obtain the expression for $R^\al$ in \cref{e:RalSimplified}.

    Let $\x = (x_1, x_2)$ be a point off the $x_2$ axis, $\x = (x_1, x_2)$, supposing $x_1, x_2 > 0$, and let $\wh{\x} = \x/\abs{\x}$. Integrating as in \cref{e:RAsPathIntegral} from the origin to $0$ in a straight line, $\BoldTau = \wh{\x}$, and since $R^\al(0) = 0$,
    \begin{align*}
        R^\al(&\x)
            = \int_0^{\abs{\x}} H^\al(r \wh{\x})^\perp \cdot \wh{\x} \, dr
            = -\int_0^{\abs{\x}} H^\al(r \wh{\x}) \cdot \wh{\x}^\perp \, dr \\
            &= -\int_0^{\abs{\x}} \sum_{j \in \Z^*}
                K^\al(r \wh{\x} - \n_j) \cdot
                    (-\wh{x}_2, \wh{x}_1)  \, dr
            = \al c_\al \int_0^{\abs{\x}} \sum_{j \in \Z^*}
                \frac{(r \wh{\x} - \n_j)^\perp \cdot
                (\wh{x}_2, -\wh{x}_1)}
                {\abs{r \wh{\x} - \n_j}^{\al + 2}} \, dr \\
            &= \al c_\al \int_0^{\abs{\x}} \sum_{j \in \Z^*}
                \frac{(-r \wh{x}_2, r \wh{x}_1 - j)
                    \cdot (\wh{x}_2, -\wh{x}_1)}
                {\pr{(r \wh{x}_1 - j)^2
                    + r^2 \wh{x}_2^2}^{\frac{\al + 2}{2}}}
                    \, dr \\
            &= -\al c_\al \int_0^{\abs{\x}} \sum_{j \in \Z^*}
                \frac{r \wh{x}_1^2 + r \wh{x}_2^2 - \wh{x}_1 j}
                {\pr{(r \wh{x}_1 - j)^2
                    + r^2 \wh{x}_2^2}^{1 + \frac{\al}{2}}} \, dr.
    \end{align*}
    Making the change of variables, $u = (r \wh{x}_1 - j)^2 + r^2 \wh{x}_2^2$, so $du = (2 (r \wh{x}_1 - j) \wh{x}_1 + 2 r \wh{x}_2^2) \, dr = 2(r \wh{x}_1^2 + r \wh{x}_2^2 - \wh{x}_1 j) \, dr$, gives
    \begin{align*}
        \begin{split}
        R^\al(\x)
            &= -\frac{\al c_\al}{2}
                    \sum_{j \in \Z^*}
                    \int_{j^2}^{\abs{\x}^2 \wh{x}_2^2
                        + (\abs{\x} \wh{x}_1 - j)^2} 
                    \frac{d u}
                    {u^{1 + \frac{\al}{2}}} \, dr
            = \frac{\al c_\al}{2 \tfrac{\al}{2}}
                \sum_{j \in \Z^*}
                u^{-\frac{\al}{2}}\Big|
                    _{j^2}^{x_2^2 + (x_1 - j)^2}
                \\
            &= c_\al \sum_{j \in \Z^*}^\iny 
                \brac{
                    \frac{1}{(j^2)^{\frac{\al}{2}}}
                - 
                        \frac{1}{(x_2^2 + (x_1 - j)^2)
                            ^{\frac{\al}{2}}}
                },
        \end{split}
    \end{align*}
    giving \cref{e:RalSimplified}.

    \medskip
    \noindent\textbf{Step 3}:
    Next, we bound $R^\al$ in $L^\iny$ on $\Pi_p$.

     We cannot split the sum in \cref{e:RalSimplified} into two sums, since neither sum would converge separately. Instead, we need to bound the difference that appears in each term. For this, we write,
     \begin{align*}
         x_2^2 + (x_1 - j)^2
            &= (x_1^2 + x_2^2)
                - 2 x_1 j + j^2
            = \abs{\x}^2 - 2 x_1 j + j^2
     \end{align*}
    so
    \begin{align*}
        T_j
            &:= \frac{1}{\abs{j}^\al}
                    - 
                            \frac{1}{(x_2^2 + (x_1 - j)^2)
                                ^{\frac{\al}{2}}}
            = \frac{
                (\abs{\x}^2  - 2 x_1 j + j^2)^{\frac{\al}{2}}
                                - \abs{j}^\al
                }
                {
                    \abs{j}^\al
                    (\abs{\x}^2 - 2 x_1 j + j^2)
                                ^{\frac{\al}{2}}
                }
            \\
            &=
            \frac
            {
                 \abs{j}^\al \pr{1 + \frac{\abs{\x}^2 - 2 x_1 j}{j^2}}
                    ^\frac{\al}{2}
                - \abs{j}^\al
            }
            {
                \abs{j}^\al \abs{j}^\al
                \pr{1 + \frac{\abs{\x}^2 - 2 x_1 j}{j^2}}
                    ^\frac{\al}{2}
            }
            =
            \abs{j}^{-\al}   
            \frac
            {
                 \pr{1 + \frac{\abs{\x}^2 - 2 x_1 j}{j^2}}
                    ^\frac{\al}{2}
                - 1
            }
            {
                \pr{1 + \frac{\abs{\x}^2 - 2 x_1 j}{j^2}}
                    ^\frac{\al}{2}
            }.
    \end{align*}
    If $\frac{\abs{\x}^2 - 2 x_1 j}{j^2} > -1$ then we can apply \cref{L:CalcInequality}. Now,
    \begin{align*}
        \frac{\abs{\x}^2 - 2 x_1 j}{j^2} > -1
            &\iff \abs{\x}^2 - 2 x_1 j > - j^2
            \iff g(j) := j^2 - 2 x_1 j + \abs{\x}^2 > 0.
    \end{align*}
    The discriminant of the polynomial $g(j)$ is $4 x_1^2 - 4 \abs{\x}^2 = - 4 x_2^2 < 0$, and $g(0) = \abs{\x}^2 > 0$, so we conclude that we can, in fact, apply \cref{L:CalcInequality} for all $j \in \Z^*$. This gives,
    \begin{align*}
        \abs{T_j}
            &\le \abs{j}^{-\al}
                \frac{\al}{2}
                \frac
                {
                    \frac{\abs{\x}^2 - 2 x_1 j}{j^2} 
                }
                {
                    a_j^{\frac{\al}{2}}
                },
    \end{align*}
    where
    \begin{align*}
        a_j
            &:= 1 + \frac{\abs{\x}^2 - 2 x_1 j}{j^2}
            = \frac{g(j)}{j^2}.
    \end{align*}
    We know from our earlier analysis that $a_j > 0$ for all $j$. To find its minimal value, let $h(j) := \frac{\abs{\x}^2 - 2 x_1 j}{j^2}$, and calculate, for $j \ne 0$,
    \begin{align*}
        j^4 h'(j)
             &= j^2 (-2 x_1) - (\abs{\x}^2 - 2 x_1 j) 2 j
             = 2 x_1 j^2 - 2 \abs{\x}^2 j
             = 2 j (x_1 j - \abs{\x}^2) = 0 \\
             &\iff j = \abs{\x}^2/x_1.
    \end{align*}
    Now, $\lim_{j \to \pm \iny} h(j) = 0$, $\lim_{j \to 0^\pm} h(j) = +\iny$, and
    \begin{align*}
        h(\abs{\x}^2/x_1)
            &= \frac{\abs{\x}^2 - 2 \abs{\x}^2}
                {\abs{\x}^4/x_1^2}
            = - \frac{x_1^2}{\abs{\x}^2},
    \end{align*}
    giving the minimum value of $h(j)$. Hence, $a_j$ is at least this minimum value plus 1; that is,
    \begin{align}
        a_j
            &\ge 1 - \frac{x_1^2}{\abs{\x}^2}
            = \frac{x_2^2}{\abs{\x}^2}
    \end{align}
    Hence,
    \begin{align*}
        \abs{T_j}
            &\le \abs{j}^{-\al}
                \frac{\al}{2}
                \pr{\frac{\abs{\x}}{x_2}}^\al
                \frac{\abs{\x}^2 - 2 x_1 j}{j^2}. 
    \end{align*}
    Thus,
    \begin{align*}
        \abs{R^\al(\x)}
            &\le
                \frac{\al c_\al}{2}
                \pr{\frac{\abs{\x}}{x_2}}^\al
                \sum_{j \in Z^*}
                    \frac{\abs{\x}^2 - 2 x_1 j}{\abs{j}^{2 + \al}}
            =
                \al c_\al
                \pr{\frac{\abs{\x}}{x_2}}^\al
                \abs{\x}^2
                \sum_{j = 1}^\iny
                    \frac{1}{j^{2 + \al}} \\
            &=
                \al c_\al \zeta(2 + \al)
                \pr{\frac{\abs{\x}}{x_2}}^\al
                \abs{\x}^2.
    \end{align*}
    By symmetry, this bound holds for all points not on the vertical or horizontal axes.

    If we restrict $\x$ to lie in $\Pi_p$ with $\abs{x_2} > 1$ then $\abs{\x}/x_2 \le \sqrt{2}$, so
    \begin{align*}
        \abs{R^\al(\x)}
            &\le
                 \al c_\al 2^{\frac{\al}{2}} \zeta(2 + \al)
                \abs{\x}^2.
    \end{align*}
    On the remainder of $\Pi_p$, we just use that $H^\al$ is $C^\iny$, and we conclude that
    \begin{align*}
        \abs{R^\al(\x)}
            &\le C (1 + \abs{\x}^2)
            \text{ for all } \x \in \Pi_p.
    \end{align*}

    \medskip
    \noindent\textbf{Step 4}:  We now turn to derivatives of $R^\al$, again restricting $\x$ to lie in $\Pi_p$ with $\abs{x_2} > 1$.
    
    The estimates on $T_j$ show that the sum in \cref{e:RalSimplified} converges uniformly on compact subsets, so we can differentiate the sum term by term. Hence,
    \begin{align*}
        \prt_1 R^\al(\x)
            &= \al c_\al \sum_{j \in \Z^*}^\iny 
                    (x_1 - j) (x_2^2 + (x_1 - j)^2)
                        ^{-1 - \frac{\al}{2}}, \\
        \prt_2 R^\al(\x)
            &= \al c_\al x_2 \sum_{j \in \Z^*}^\iny 
                    (x_2^2 + (x_1 - j)^2)
                        ^{-1 - \frac{\al}{2}}.
    \end{align*}
    Thus,
    \begin{align*}
        \abs{\prt_1 R^\al(\x)}
            &\le C \sum_{j \in \Z^*}
                \frac{1 + \abs{j}}{(x_2^2 - 1 + j^2)^{1 + \frac{\al}{2}}}
            \le C + C \sum_{j \in \Z^*} \frac{1}{j^{1 + \al}}
            \le C, \\
        \abs{\prt_2 R^\al(\x)}
            &\le C \sum_{j \in \Z^*}
                \frac{x_2}{(x_2^2 - 1 + j^2)^{1 + \frac{\al}{2}}}
            \le C x_2.
    \end{align*}

    We can see that repeated derivatives with respect to $x_1$ do not lead to increased growth in $x_2$, while each derivative with respect to $x_2$ introduces growth by another factor of $x_2$.

    We conclude that on $\Pi_p$,
    \begin{align}\label{e:RBoundOnPip}
        \abs{\grad^n R^\al(\x)}
            &\le
            \begin{cases}
                C (1 + \abs{x_2}^2) & \text{if } n = 0, \\
                C (1 + \abs{x_2}^n) & \text{if } n > 0.
            \end{cases}
    \end{align}

    \medskip
    \noindent\textbf{Step 5}: Finally, we incorporate estimates on the singularities arising from $G^\al$.

    For $\x \in \R^2 \setminus \Pi_p$, $G^\al(\x)$ is $C^\iny$ and decays, along with all of its derivatives, to zero at infinity. Hence, returning to \cref{e:RalOnR2}, we see that near any point $\n_j$ in $\L^*$, $R^\al$ has a singularity like that of $G^\al(\x - \n_j)$. Combined with \cref{e:RBoundOnPip}, and noting that the singularity in $G^\al(\x - \n_j)$ is larger than the constant terms in \cref{e:RBoundOnPip}, we obtain the bounds in \cref{e:RBoundOnR2MinusLStar}.
\end{proof}

We used \cref{L:CalcInequality} in the proof above of \cref{P:BoundsOnR}.

\begin{lemma}\label{L:CalcInequality}
    Let $b \in (0, 1)$. Then for all $x > -1$,
    \begin{align*}
        (1 + x)^b \le 1 + b x.
    \end{align*}
\end{lemma}
\begin{proof}
    Let $f(x) := 1 + b x - (1 + x)^b$. Then $f(0) = 0$ and if $x>0,$
    \begin{align*}
        f'(x)
            &= b - b (1 + x)^{b - 1} 
                > b - b = 0.
\end{align*}
Likewise, if instead $-1<x<0,$ then
    \begin{align*}
        f'(x)
< b - b = 0.
    \end{align*}
    which establishes the inequality.
\end{proof}

\section{Some lemmas}\label{A:SomeLemmas}

\cref{T:Picard,T:PicardContinuation} are adapted from Theorems 3.1 and 3.3 of \cite{MB2002}.

\begin{theorem}[Picard]\label{T:Picard}
    Let $O$ be an open subset of the Banach space $X$ and let
    $L$ be a nonlinear operator from $X$ to $X$ that is locally Lipschitz continuous on $O$; that is, for any $x \in O$ there exists an open neighborhood $U_x \subseteq O$ of $x$ and $M > 0$ for which
    \begin{align*}
        \norm{L(x) - L(x')}_X
            &\le M \norm{x - x'}_X
            \text{ for all } x, x' \in U_x.
    \end{align*}
    Then for any $x_0 \in O$ there exists a time $T > 0$ such that the initial value problem (ODE),
    \begin{align*}
        \begin{cases}
            \displaystyle
                \diff{x(t)}{t} = L(x(t)), \\[0.5em]
            x(0) = x_0
        \end{cases}
    \end{align*}
    has a unique solution in $C^1((-T, T); O)$.
\end{theorem}

\begin{theorem}[Continuation of ODE]\label{T:PicardContinuation}
    Let $X$, $O$, and $L$ be as in \cref{T:Picard}. A unique solution $x$ for $x'(t) = L(x(t))$ exists on $[0, \iny)$ or leaves the open set $O$ as $t \to \iny$.
\end{theorem}

The following is adapted from Theorem II.5.16 of \cite{BoyerFabrie2013}:
\begin{lemma}[Aubin-Lions-Simon]\label{L:AubinLions}
    Let
    \begin{align*}
        X_0 \subset\subset X_1 \subset X_2
    \end{align*}
    be Banach spaces with $X_0$ compactly embedded in $X_1$ and $X_1$ continuously embedded in $X_2$. Let $1 \le p,r \le \iny$ and for $T > 0$ define
    \begin{align*}
        E_{p, r}
            &:= \set{v \in L^p(0, T; X_0) \colon
                \prt_t v \in L^r(0, T; X_2)}.
    \end{align*}
    \begin{enumerate}
        \item
            If $p < \iny$ then
            $E_{p, r} \subset\subset L^p(0, T; X_1)$.

            \item
                If $p = \iny$ and $r > 1$ then
                $E_{p, r} \subset\subset C([0, T]; X_1)$.
\end{enumerate}
\end{lemma}

The following is a version of the Banach–Alaoglu theorem (see Theorem II.2.7 of \cite{BoyerFabrie2013}):
\begin{lemma}\label{L:WeakStarConv}
    If $E$ is a separable Banach space then any bounded sequence in $E^*$ has a weak-$*$ convergent subsequence in $E^*$.
\end{lemma}

We often bring a derivative inside an integral using \cref{L:DerivInsideIntegral}, from Theorem 2.27 of \cite{Folland}:
\begin{lemma}\label{L:DerivInsideIntegral}
    Let $f \colon \Tn^2 \to \R$ with $\abs{\prt_\beta f(\beta, \cdot)} \le g$ for some $g \in L^1(\Tn)$. Then
    \begin{align*}
        \prt_\beta
            \int_\Tn f(\beta, \eta) \, d \eta
                = \int_\Tn \prt_\beta f(\beta, \eta) \, d \eta.
    \end{align*}
\end{lemma}

\cref{L:IBP} allows integration by parts in the presence of a controlled singularity: 
\begin{lemma}\label{L:IBP}
    As in \cref{R:FParameterization}, parameterize $\Tn$ by $\beta \in [-\pi, \pi]$ and define, for any $\eps > 0$, $\Tn_\eps := [-\pi, \pi] \setminus [-\eps, \eps]$.
    Suppose that $f, g \colon \Tn \to \R$ with
    \begin{align*}
        f \prt_\eta g, (\prt_\eta f) g \in L^1(\Tn), \,
        f, g \in H^1(\Tn_\eps) \text{ for all } \eps > 0,
            \text{ and }
            \lim_{\eta \to 0} f(\eta) g(\eta) = 0.
    \end{align*}
    Then
    \begin{align*}
        \int_\Tn (\prt_\eta f(\eta)) g(\eta) 
            &= - \int_\Tn f(\eta)  \prt_\eta g(\eta) .
    \end{align*}
\end{lemma}
\begin{proof}
    Using our assumptions on $f$ and $g$,
    \begin{align*}
        \int_\Tn f \prt_\eta g \, d \eta
            &= \lim_{\eps \to 0} \int_{\Tn_\eps}
                f \prt_\eta g \, d \eta
            = -\lim_{\eps \to 0}
                \brac{\int_{\Tn_\eps}
                (\prt_\eta f) g \, d \eta
                +
                fg|_{-\eps}^\eps}
            = - \int_\Tn (\prt_\eta f) g \, d \eta.
            \qedhere
    \end{align*}
\end{proof}

We use \cref{L:IBP} through the following immediate corollary:
\begin{cor}\label{C:IBP}
    Suppose that $\abs{f(\eta)} \le C \abs{\eta}^r$, $\abs{\prt_\eta f(\eta)} \le C \abs{\eta}^{r - 1}$, $\abs{g(\eta)} \le C \abs{\eta}^p, \abs{\prt_\eta g(\eta)} \le C \abs{\eta}^q$ with
    $p + r > 0$, $q + r > -1$. Then
    \begin{align*}
        \int_\Tn \prt_\eta f(\eta) g(\eta) 
            &= - \int_\Tn f(\eta)  \prt_\eta g(\eta) .
    \end{align*}
\end{cor}

\begin{lemma}\label{L:DiffPowerOfAbs}
    Let $\vv \colon \R \to \R^2$ be differentiable. For any $a \in \R$,
    \begin{align*}
        \diff{}{r} \abs{\vv(r)}^a
            &= a \abs{\vv(r)}^{a - 2}
                \vv \cdot \diff{\vv}{r}.
    \end{align*}
\end{lemma}
\begin{proof}
    We have,
    \begin{align*}
        \diff{}{r} \abs{\vv(r)}^a
            &= \diff{}{r} (\abs{\vv(r)}^2)^{\frac{a}{2}}
            = \frac{a}{2} (\abs{\vv(r)}^2)^{\frac{a}{2} - 1}
                \diff{}{r} \abs{\vv}^2
            = \frac{a}{2} \abs{\vv(r)}^{a - 2}
                2 \vv \cdot \diff{\vv}{r}.
            \qedhere
    \end{align*}
\end{proof}

\cref{L:C2BoundFromAgmon} allows us to control \Holder norms up to $C^{1, 1}$ by the $H^3$ norm:
\begin{lemma}\label{L:C2BoundFromAgmon}
    Let $\bgamma$ be an $H^3$ path in $\R^2$. Then
    \begin{align*}
        \norm{\bgamma}_{L^\iny(\Tn)}
            &\le C \norm{\bgamma}_{H^1(\Tn)}, \\\
        \norm{\prt_\eta \bgamma(\eta)}_{L^\iny(\Tn)}
            &\le \norm{\bgamma}_{C^{0, 1}(\Tn)}
            \le C \norm{\bgamma}_{H^2(\Tn)}, \quad \\
        \norm{\prt_\eta^2 \bgamma(\eta)}_{L^\iny(\Tn)}
            &\le \norm{\bgamma}_{C^{1, 1}(\Tn)}
            \le C \norm{\bgamma}_{H^3(\Tn)}.
    \end{align*}
\end{lemma}
\begin{proof}
    Apply the 1D Agmon's inequality in the form
    \begin{align*}
        \norm{\prt^j_\eta \bgamma}_{L^\iny(\Tn)}
            &\le C \norm{\prt^j_\eta \bgamma}_{L^2(\Tn)}
                    ^{\frac{1}{2}}
                \norm{\prt^{j + 1}_\eta \bgamma}_{L^2(\Tn)}
                    ^{\frac{1}{2}}
            \le C \norm{\bgamma}_{H^{j + 1}(\Tn)},
    \end{align*}
    where we note that the constant $C$ depends only on $\Tn$, and so is effectively absolute.
\end{proof}

\begin{lemma}\label{L:prtDeltaSobolevBound}
    Let $\bgamma$ be an $H^k$ path in $\R^2$, $k \ge 1$. For all $\beta, \eta \in \Tn$,
    \begin{align*}
        \abs{\prt_\eta^j \bdelta_\beta(\eta)}
            &\le C \norm{\bgamma}_{H^{j + 1}(\Tn)}
            \text{ for } 0 \le j \le k - 1, \\
       \abs{\prt_\eta^j \bdelta_\beta(\eta)}
            &\le C \norm{\bgamma}_{H^{j + 2}(\Tn)} \abs{\beta}
            \text{ for } 0 \le j \le k - 2, \\
        \norm{\prt_\eta^j \bdelta_\beta(\eta)}_{L^2_\eta(\Tn)}
            &\le C \norm{\bgamma}_{H^{j + 1}(\Tn)} \abs{\beta}
            \text{ for } 0 \le j \le k - 1, \\
        \norm{\prt_\eta^j \bdelta_\beta(\eta)}_{L^4_\eta}
            &\le C \norm{\bgamma}_{H^{j + 2}(\Tn)} \abs{\beta}
                \text{ for } 0 \le j \le k - 2.
    \end{align*}   
    Moreover,
	\begin{align*}
		\norm{\bdelta_\beta(\eta)}_{H^j_\eta(\Tn)}
            &\le C \norm{\bgamma}_{H^{j + 1}(\Tn)} \abs{\beta}
            \text{ for } 0 \le j \le k - 1, \\
        \norm{\bdelta_\beta(\eta)}_{H^{4, j}_\eta}
            &\le C \norm{\bgamma}_{H^{j + 2}(\Tn)} \abs{\beta}
                \text{ for } 0 \le j \le k - 2.
	\end{align*}
\end{lemma}
\begin{proof}
    Applying \cref{L:C2BoundFromAgmon},
    \begin{align*}
        \abs{\prt_\eta^j \bdelta_\beta(\eta)}
            &\le \abs{\prt_\eta^j \bgamma(\eta)}
                    + \abs{\prt_\eta^j \bgamma(\eta - \beta)}
            \le 2 \norm{\prt_\eta^j \bgamma}_{L^\iny}
            \le C \norm{\bgamma}_{H^{j + 1}}. \\
        \abs{\prt_\eta^j \bdelta_\beta(\eta)}
            &\le \abs{\prt_\eta^j \bgamma(\eta) - \prt_\eta^j \bgamma(\eta - \beta)}
            \le \norm{\prt_\eta^{j + 1}}_{L^\iny} \abs{\beta}
            \le  C \norm{\bgamma}_{H^{j + 2}} \abs{\beta}.
    \end{align*}
    
    For the third inequality, we adapt the argument on pages 2577-2578 of \cite{Gancedo2008}:
    \begin{align*}
        \prt_\eta^j \bdelta_\beta(\eta)
            &= \prt_\eta^j \bgamma(\eta)
                    - \prt_\eta^j \bgamma(\eta - \beta)
            = \int_0^1 \diff{}{s} \prt_\eta^j \bgamma
                    (\eta + (s - 1) \beta) \, ds \\
            &= \beta \int_0^1
                \prt_\eta^{j + 1} \bgamma
                (\eta + (s - 1) \beta) \, ds.
    \end{align*}
    Applying Minkowski's integral inequality for $L^p$, $p \in [1, \iny]$,
    \begin{align*}
        \norm{\prt_\eta^j \bdelta_\beta(\eta)}_{L^2_\eta}
             \le \abs{\beta} \int_0^1
                \norm{\prt_\eta^{j + 1} \bgamma(\eta + (s - 1) \beta)}_{L^p_\eta}
                    \, ds
             = \abs{\beta} \int_0^1
                \norm{\prt_\eta^{j + 1} \bgamma
                    (\eta)}_{L^p_\eta}
                    \, ds,
    \end{align*}
    which gives the third inequality when setting $p = 2$. Summing over $j$ gives the bound on $\norm{\bdelta_\beta(\eta)}_{H^j_\eta(\Tn)}$.

    For the final inequality, we use the 1D Ladyzhenskaya's inequality, to give
    \begin{align*}
        \norm{\prt_\eta^j \bdelta_\beta(\eta)}_{L^4_\eta}
            &\le C \norm{\prt_\eta^j \bdelta_\beta(\eta)}_{L^2_\eta}
                    ^{\frac{3}{4}}
                \norm{\prt_\eta^{j + 1} \bdelta_\beta(\eta)}_{L^2_\eta}
                    ^{\frac{1}{4}}
            \le C \norm{\bgamma}_{H^{j + 1}}^{\frac{3}{4}}
                    \abs{\beta}^{\frac{3}{4}}
                \norm{\bgamma}_{H^{j + 2}}^{\frac{1}{4}}
                    \abs{\beta}^{\frac{1}{4}} \\
            &\le C \norm{\bgamma}_{H^{j + 2}} \abs{\beta}.
    \end{align*}
    Summing over $j$ gives the bound on $\norm{\bdelta_\beta(\eta)}_{H^{4, j}_\eta}$.
\end{proof}

\begin{lemma}\label{L:ROCMollification}
    Let $f \in C^1(\Tn)$. For any $p \in [1, \iny]$,
    \begin{align*}
        \norm{\phi_\eps * f - f}_{L^p(\Tn)}
            \le 2 \pi \eps \norm{f}_{C^{0, 1}(\Tn)}
            \le 2 \pi \eps \norm{f}_{H^2(\Tn)}.
    \end{align*}
\end{lemma}
\begin{proof}
    Because $\int_\Tn \phi_\eps = 1$,
    \begin{align*}
        &\norm{\phi_\eps * f(\eta) - f(\eta)}_{L^p_\eta(\Tn)}
            = \norm[\bigg]
                {
                    \int_\Tn (\phi_\eps(\beta) [f(\eta - \beta) - f(\eta)] \, d \beta
                }_{L^p_\eta(\Tn)} \\
            &\qquad
            = \norm[\bigg]
                {
                    \int_\Tn (\phi_1(\eps^{-1} \beta) [f(\eta - \beta) - f(\eta)] \, d \beta
                }_{L^p_\eta(\Tn)} \\
            &\qquad
            = \norm[\bigg]
                {
                    \int_\Tn (\phi_1(\beta) [f(\eta - \eps \beta) - f(\eta)] \, d \beta}
                _{L^p_\eta(\Tn)} \\
            &\qquad
            \le 
                \eps \int_\Tn (\phi_1(\beta) \abs{\beta}
                    \norm[\bigg]
                    {
                        \frac{f(\eta - \eps \beta) - f(\eta)}{\eps \beta}
                    }_{L^p_\eta(\Tn)}
                    \, d \beta \\
            &\qquad
            \le 
                \eps \int_\Tn (\phi_1(\beta) \abs{\beta}
                    \norm[\bigg]
                    {
                        \frac{f(\eta - \eps \beta) - f(\eta)}{\eps}
                    }_{L^\iny_\eta(\Tn)}
                    \, d \beta
            \le \eps \norm{\prt_\eta f}_{L^\iny(\Tn)}
                \int_\Tn \abs{\beta} \phi_1(\beta) \, d \beta \\
            &\qquad
            \le 2 \pi \eps \norm{f}_{C^{0, 1}(\Tn)}
                \int_\Tn \phi_1(\beta) \, d \beta
            = 2 \pi \eps \norm{f}_{C^{0, 1}(\Tn)}
            \le 2 \pi \eps \norm{f}_{H^2(\Tn)}
    \end{align*}
    by \cref{L:C2BoundFromAgmon}.
\end{proof}

\begin{lemma}[Osgood's lemma, integral form]\label{L:Osgood}
    Let $L \colon [0, \iny) \to [0, \iny)$ be measurable, $g \colon [0, \iny) \to [0, \iny)$ be integrable, and let $\mu \colon [0, \iny) \to [0, \iny)$ be a continuous non-decreasing function with $\mu(0) = 0$. If
    \begin{align*}
        L(t)
            \le a + \int_0^t g(s) \mu(L(s)) \, ds
    \end{align*}
    for all $t \ge 0$
    then
    \begin{align*}
        \int_a^{L(t)}\frac{ds}{\mu(s)} \le \int_0^t g(s) \, ds.
    \end{align*}
    If $a = 0$ and $\mu$ is an Osgood function, meaning that
    \begin{align*}
    	\int_0^1 \frac{ds}{\mu(s)} = \iny,
    \end{align*}
	then $L(t) \equiv 0$.
\end{lemma}

\section*{Acknowledgments}
DMA is grateful to the National Science Foundation for support through grant
DMS-2307638.

\bibliography{Refs} 

\bibliographystyle{plain}

\end{document}